\documentclass[11pt]{amsart}
\usepackage{amssymb,amscd,amsmath,amsthm,latexsym,epsfig,hyperref}
\usepackage[mathscr]{euscript}
\usepackage{verbatim}
\usepackage{tikz-cd}
\usepackage{graphicx}

\usepackage[para]{bigfoot}
\usepackage{bbold}

\DeclareFontFamily{OT1}{rsfs}{}
\DeclareFontShape{OT1}{rsfs}{n}{it}{<-> rsfs10}{}
\DeclareMathAlphabet{\curly}{OT1}{rsfs}{n}{it}

\DeclareFontFamily{U}{rsfs}{%
\skewchar\font127}
\DeclareFontShape{U}{rsfs}{m}{n}{%
<-6>rsfs5<6-8.5>rsfs7<8.5->rsfs10}{}
\DeclareSymbolFont{rsfs}{U}{rsfs}{m}{n}
\DeclareSymbolFontAlphabet
{\mathrsfs}{rsfs}
\DeclareRobustCommand*\rsfs{%
\@fontswitch\relax\mathrsfs}

\newcommand\beq[1]{\begin{equation}\label{#1}}
\newcommand\eeq{\end{equation}}
\newcommand\beqa{\begin{eqnarray*}}
\newcommand\eeqa{\end{eqnarray*}}

\theoremstyle{plain}
\newtheorem{thm}{Theorem}[section]
\newtheorem{prop}[thm]{Proposition}
\newtheorem{lem}[thm]{Lemma}
\newtheorem{lemma}[thm]{Lemma}

\newtheorem{defi}[thm]{Definition}
\newtheorem{defn}[thm]{Definition}

\newtheorem{cor}[thm]{Corollary}

\newtheorem{prop-defi}[thm]{Proposition-Definition}
\newtheorem{thm-defi}[thm]{Theorem-Definition}
\newtheorem{lem-defi}[thm]{Lemma-Definition}

\newtheorem{conj}[thm]{Conjecture}

\theoremstyle{remark}
\newtheorem{rmk}[thm]{Remark}
\newtheorem{example}[thm]{Example}

\newcommand{\aA}{\mathcal{A}}
\newcommand{\bB}{\mathcal{B}}
\newcommand{\cC}{\mathcal{C}}
\newcommand{\dD}{\mathcal{D}}
\newcommand{\eE}{\mathcal{E}}

\newcommand{\hH}{\mathcal{H}}
\newcommand{\iI}{\mathcal{I}}

\newcommand{\mM}{\mathcal{M}}

\newcommand{\oO}{\mathcal{O}}
\newcommand{\pP}{\mathcal{P}}

\newcommand{\sS}{\mathcal{S}}

\newcommand{\wW}{\mathcal{W}}

\newcommand{\yY}{\mathcal{Y}}

\newcommand{\Supp}{\mathop{\rm Supp}\nolimits}

\newcommand{\dR}{\mathbf{R}}

\newcommand{\NS}{\mathop{\rm NS}\nolimits}

\newcommand{\Pic}{\mathop{\rm Pic}\nolimits}

\newcommand{\Amp}{\mathop{\rm Amp}\nolimits}

\newcommand{\id}{\textrm{id}}

\newcommand{\ch}{\mathop{\rm ch}\nolimits}
\newcommand{\rk}{\mathop{\rm rk}\nolimits}

\newcommand{\Ext}{\mathop{\rm Ext}\nolimits}
\newcommand{\Spec}{\mathop{\rm Spec}\nolimits}

\newcommand{\Coh}{\mathop{\rm Coh}\nolimits}

\newcommand{\cneq}{\mathrel{\raise.095ex\hbox{:}\mkern-4.2mu=}}
\newcommand{\eqcn}{\mathrel{=\mkern-4.5mu\raise.095ex\hbox{:}}}

\newcommand{\Aut}{\mathop{\rm Aut}\nolimits}

\newcommand{\codim}{\mathop{\rm codim}\nolimits}

\newcommand{\SL}{\mathop{\rm SL}\nolimits}

\newcommand{\Stab}{\mathop{\rm Stab}\nolimits}

\newcommand{\DT}{\mathop{\rm DT}\nolimits}

\newcommand{\End}{\mathop{\rm End}\nolimits}

\newcommand{\Imm}{\mathop{\rm Im}\nolimits}
\newcommand{\imm}{\mathop{\rm im}\nolimits}

\newcommand{\Ker}{\mathop{\rm Ker}\nolimits}
\newcommand{\Ree}{\mathop{\rm Re}\nolimits}

\newcommand{\GL}{\mathop{\rm GL}\nolimits}

\newcommand{\cl}{\mathop{\rm cl}\nolimits}

\newcommand{\BA}{{\mathbb{A}}}

\newcommand{\BC}{{\mathbb{C}}}

\newcommand{\BG}{{\mathbb{G}}}
\newcommand{\BH}{{\mathbb{H}}}

\newcommand{\BL}{{\mathbb{L}}}

\newcommand{\BQ}{{\mathbb{Q}}}
\newcommand{\BR}{{\mathbb{R}}}

\newcommand{\BT}{{\mathbb{T}}}

\newcommand{\BZ}{{\mathbb{Z}}}

\newcommand{\CA}{{\mathcal A}}

\newcommand{\CC}{{\mathcal C}}

\newcommand{\CI}{{\mathcal I}}

\newcommand{\CL}{{\mathcal L}}
\newcommand{\CM}{{\mathcal M}}
\newcommand{\CN}{{\mathcal N}}
\newcommand{\CO}{{\mathcal O}}
\newcommand{\CP}{{\mathcal P}}

\newcommand{\CS}{{\mathcal S}}

\newcommand{\CW}{{\mathcal W}}
\newcommand{\CX}{{\mathcal X}}
\newcommand{\CY}{{\mathcal Y}}

\newcommand{\aaa}{\mathsf{a}}
\newcommand{\e}{\mathbf{e}}
\newcommand{\Forg}{\mathrm{Forg}}
\newcommand{\Constr}{\mathrm{Constr}}
\newcommand{\St}{\mathrm{St}}
\newcommand{\A}{{\mathbf{A}}}
\newcommand{\DTb}{\mathop{\mathbf{DT}}\nolimits}
\newcommand{\pt}{\mathbf{p}}

\begin{document}

\title[DT invariants of abelian threefolds]{Donaldson--Thomas invariants of abelian threefolds and Bridgeland stability conditions}
\date{\today}

\author{Georg Oberdieck}
\address{Mathematisches Institut, Universit\"at Bonn}
\email{georgo@math.uni-bonn.de}

\author{Dulip Piyaratne}
\address{Department of Mathematics, University of Arizona }
\email{piyaratne@math.arizona.edu}

\author{Yukinobu Toda}
\address{Kavli Institute for the Physics and Mathematics of the Universe,
University of Tokyo (WPI)}
%5-1-5 Kashiwanoha, Kashiwa, 277-8583, Japan}
\email{yukinobu.toda@ipmu.jp}

\begin{abstract}
We study the reduced Donaldson--Thomas theory of abelian threefolds using Bridgeland stability conditions.
The main result is the invariance of the reduced 
Donaldson--Thomas invariants under
all derived autoequivalences, up to explicitly given wall-crossing terms.
We also present a numerical criterion for the absence of walls in terms of a discriminant function.
For principally polarized abelian threefolds of Picard rank one, the wall-crossing contributions are discussed in detail.
The discussion yield evidence for a conjectural formula for curve counting 
%reduced Donaldson--Thomas 
invariants by Bryan, Pandharipande, Yin, and the first author.

For the proof we strengthen several known results on Bridgeland stability conditions of abelian threefolds.
We show that certain previously constructed stability conditions satisfy the full support property.
In particular, the stability manifold is non-empty. %in the main component of the stability manifold satisfy the full support property
%which implies that the stability manifold is non-empty.
%%implies that the stability manifold is non-empty.
We also prove the existence of a Gieseker chamber and determine all wall-crossing contributions.
A definition of reduced generalized Donaldson--Thomas invariants for arbitrary Calabi--Yau threefolds with abelian actions is given. % via equivariant Hall algebras.

\end{abstract}
\maketitle
\baselineskip=14pt

\setcounter{tocdepth}{1} 
\tableofcontents
%\newpage
\section{Introduction}

\subsection{Overview} \label{subsection_intro_overview}
Let $X$ be a smooth projective Calabi--Yau threefold with an ample divisor $H$,
and let $\Gamma$ be the image of the Chern character map
\begin{align*}
\ch \colon K(X) \twoheadrightarrow \Gamma \subset H^{2\ast}(X, \mathbb{Q}). 
\end{align*}
For every $v \in \Gamma$ consider the Donaldson--Thomas invariant
\[ \DT_H(v) \in \BQ. \]
If the moduli space $M_H(v)$ of $H$-Gieseker  semistable sheaves of Chern character $v$
consists of stable sheaves,
%If $v$ is primitive, 
then $\DT_H(v)$ is defined by 
%the integral
\begin{equation}
\DT_H(v) \cneq \int_{M_H(v)} \nu \ \mathrm{d}e \cneq \sum_{k \in \BZ} k \cdot e\left( \nu^{-1}(k) \right), \label{def_DT}
\end{equation}
where $\nu : M_H(v) \to \BZ$ is the Behrend function \cite{Beh} and $e( - )$ is the topological Euler characteristic.
In general, $\DT_H(v)$ is defined via the motivic Hall algebra \cite{JS}. 
The invariants $\DT_H(v)$ enumerate (with weights) Gieseker semistable sheaves on the threefold.

An interesting question is the following:
\emph{Given a derived autoequivalence $g \in \Aut D^b(X)$,
how are the Donaldson--Thomas invariants $\DT_H(v)$ and $\DT_H(g_{\ast} v)$ related?}
For the dualizing functor and curve counting Donaldson--Thomas invariants
such a relation was established
in \cite{T08, BrH, T16}
and proved the rationality and functional equation part of the GW/DT correspondence conjecture \cite{MNOP}.
Another instance is \cite{OS2} where an autoequivalence on elliptically fibered Calabi--Yau threefolds
yielded modular properties of generating series of Donaldson--Thomas invariants.

In this paper we answer the above question in full generality
for the reduced\footnote{For an abelian threefold $A$ with dual $\widehat{A} = \Pic^0(A)$,
the group $A \times \widehat{A}$ acts on the moduli spaces $M_H(v)$ and forces the
Donaldson--Thomas invariants \eqref{def_DT} to vanish. The theory is only interesting after reduction, see Section~\ref{Subsection_Reduced_DT_intro}.
%The reduction step is an essential ingredient in our methods.
}
Donaldson--Thomas invariants of abelian threefolds.
The results are strong constraints on these invariants, and
may be leveraged later for their explicit computation.
Our approach is based on Bridgeland stability conditions \cite{Brs1} and wall-crossing techniques.
In particular, this paper is the first instance that
Bridgeland stability conditions of a compact Calabi--Yau threefold
have been applied to 
Donaldson--Thomas theory in this context
%the invariance of Donaldson--Thomas invariants under auto-equivalences 
(earlier work either
used weak/limit stability conditions, 
e.g. \cite{T08, BrH, T16, OS2} mentioned 
above, or 
considered Bridgeland stability conditions for 
local surfaces, e.g. \cite{T12, MT} for local K3 surfaces).

Abelian threefolds are `simple' enough among all Calabi--Yau threefolds
such that the technical difficulties regarding Bridgeland stability conditions can be overcome.
Yet they are also `complicated' enough for interesting phenomena to appear.
%Hence it should provide
We hope this intermediate case provides 
insights into the application of Bridgeland stability conditions
to the Donaldson--Thomas theory of compact Calabi--Yau threefolds in general.
%\footnote{
%This paper is the first instance where Bridgeland stability conditions have been used
%on a compact Calabi--Yau threefold
%to constrain Donaldson--Thomas invariants by derived autoequivalences.
%Earlier work either considered local surfaces \cite{T12, MT} or used weak/limit stability conditions \cite{BrH, T08}.}

\subsection{Reduced Donaldson--Thomas invariants} \label{Subsection_Reduced_DT_intro}
Let $A$ be a non-singular abelian threefold over $\mathbb{C}$. With $H$ and $\Gamma$ as before, let
$M_{H}(v)$ be the moduli space of $H$-Gieseker
semistable sheaves on $A$ of Chern character $v \in \Gamma$.
The product $A \times \widehat{A}$ acts on $M_H(v)$ by
\begin{align*}
(a, L) \cdot E=T_a^{\ast} E \otimes L
\end{align*}
where $T_a \colon A \to A$ is the translation $x \mapsto x+a$.

We define \emph{reduced} Donaldson--Thomas invariants
$\DT_H(v) \in \BQ$ which count $A \times \widehat{A}$-orbits of Gieseker semistable sheaves as follows.\footnote{
We have chosen here the same notation for the reduced invariants as for the (standard) Donaldson--Thomas invariants defined in \eqref{def_DT}.
However, from now on all our invariants are reduced, so this choice should not create confusion.}

If the $A \times \widehat{A}$ action has finite stabilizers and $M_H(v)$ consists of $H$-Gieseker stable sheaves,
following Gulbrandsen \cite{Gul} we define reduced Donaldson--Thomas invariants by integrating over the stack quotient:
\begin{align*}
\DT_H(v) \cneq \int_{[\overline{M}_H(v) / (A \times \widehat{A})]} \nu \ \mathrm{d}e
%= \sum_{n \in \BZ} n \cdot e\left( \nu^{-1}(n) \right),
\end{align*}
where $\nu \colon [\overline{M}_H(v) / (A \times \widehat{A})] \to \BZ$ is the Behrend function of the stack
and the topological Euler characteristic is taken in the orbifold sense.
For arbitrary $v \in \Gamma$ the reduced invariant $\DT_H(v)$
is defined via the $A \times \widehat{A}$-equivariant motivic Hall algebra, see Section~\ref{section:reducedDT}.

\subsection{Autoequivalences}
A sheaf $E \in \Coh(A)$ is called \emph{semihomogeneous} if its stabilizer group under the $A \times \widehat{A}$ action
\begin{align}\label{def:Phi(E)}
\Xi(E)=\{(a, L) \in A \times \widehat{A} \, : \,
T_a^{\ast}E \otimes L \cong E\}
\end{align}
is of dimension $3$.
Consider the subset of semihomogeneous classes
\begin{align}\label{def:C0}
\cC \cneq
\{ \pm \ch(E) : E \mbox{ is a semihomogeneous sheaf }\} \subset \Gamma. 
\end{align}
Let also $\chi : \Gamma \times \Gamma \to \BZ$ be the Euler pairing on $\Gamma$.

%The following invariance property is the main result of the paper.
We prove the following invariance property in Section~\ref{subsec:independence}.

\begin{thm}\label{intro:first}
Suppose $v \in \Gamma$ can not be written 
as $\gamma_1+\gamma_2$
for any $\gamma_i \in \cC$
with $\chi(\gamma_1, \gamma_2) \neq 0$.
%where $\chi$ is the Euler pairing on $\Gamma$. 
Then $\DT_H(v)$ is independent of $H$ and 
\begin{align*}
\DT_{H}(g_{\ast}v)=\DT_H(v). 
\end{align*}
for every autoequivalence $g \in \Aut(D^b(A))$.
\end{thm}
\vspace{5pt}

If $v \in \Gamma$ does not satisfy the assumption of Theorem~\ref{intro:first},
then $\DT_H(v)$ and $\DT_{H}(g_{\ast}v)$ are related by a wall-crossing formula.
The wall-crossing formula depends only on the derived equivalence $g$ and the possible ways in which $v$ can be written as a sum of two semihomogeneous classes.
The wall-crossing contributions are determined in Lemma~\ref{cor:C}.
In particular, the precise wall-crossing formula can be worked out explicitly in any concrete case.
An example of non-trivial wall-crossing is discussed in Theorem~\ref{intro:main}.

%
%Using the methods of this paper the precise wall-crossing formula can be worked out explicitly in any concrete case.\footnote{
%The wall-crossing formula depends only on the derived equivalence $g$ and the possible ways in which $v$ can be written as a sum of two semihomogeneous classes.
%The wall-crossing contributions are determined in Lemma~\ref{cor:C}.}
%An example of non-trivial wall-crossing is discussed in Theorem~\ref{intro:main}.

The assumption of Theorem~\ref{intro:first} is often cumbersome to check in practice.
We state a numerical criterion in its place. Consider the discriminant
\[ \Delta : H^{2\ast}(A,\BQ) \to \BQ, \]
that is the unique homogeneous degree $4$ polynomial function
which is invariant under the spin group and
is normalized by $\Delta(1 + \pt) = -1$.
Here $\pt \in H^6(A,\BZ)$ is the class of a point.
We refer to Appendix~\ref{Appendix_Spin_representations} for details and an explicit formula in case $A = E_1 \times E_2 \times E_3$.
We have the following. % criterion.

\begin{prop} 
\label{intro:prop_Delta}
Let $v \in \Gamma$. If $\Delta(v) \geq 0$, then $v$ satisfies the assumption of Theorem~\ref{intro:first}.
\end{prop}

Proposition~\ref{intro:prop_Delta} is in perfect agreement
with physical arguments by Sen on the behaviour of the partition function of $1/8$ BPS dyones under change of stability:
wall-crossing contributions can appear only for classes with negative discriminant, see \cite[Section~4]{Sen}.

\subsection{Principally polarized abelian threefolds of Picard rank one}
Let $(A, H)$ be a principally polarized abelian threefold with $\rho(A)=1$. 
By Mukai~\cite{Mu1} the group
$\SL_2(\mathbb{Z})$ acts on $D^b(A)$ (modulo shifts) by 
\begin{align}\label{SL2}
T=
\left( \begin{array}{cc}
1 & 1 \\
0 & 1 
\end{array}  \right)
\mapsto \, ( - ) \otimes \oO_X(H),
\quad 
S=
\left( \begin{array}{cc}
0 & -1 \\
1 & 0 
\end{array}  \right)
\mapsto \Phi_{\pP},
\end{align}
where $\Phi_{\pP}$ is the Fourier-Mukai transform with kernel the normalized Poincar\'e line bundle on $A \times A$. 
Moreover,
any autoequivalence acts
by an element in $\SL_2(\BZ)$ (moduli shifts, translation and twisting by degree $0$ line bundles).
%modulo shifts, translation by elements in $A$ and twisting by degree $0$ line bundles,
%any derived autoequivalence acts by an element in $\SL_2(\BZ)$.

The image of the Chern character map is
\begin{align}\label{Gamma}
\Gamma=\mathbb{Z} \oplus \mathbb{Z}[H] \oplus \mathbb{Z}[H^2/2] \oplus 
\mathbb{Z}[H^3/6]. 
\end{align}
%Explicit formulas for the induced action of $\mathrm{SL}_2(\BZ)$ on $\Gamma$
%can be found in Section~\ref{Subsection_Action_of_auto-equvalences_on_coh}.
Since the only semihomogeneous sheaves on $A$
are vector bundles\footnote{If $E$ is a semihomogeneous vector bundle, then 
$\ch(E) = r(E) \exp(c_1(E)/r(E))$
where $r(E)$ is the rank of $E$, see \cite{Mu4}.} or have $0$-dimensional support
the 
%$\SL_2(\mathbb{Z})$-invariant 
subset of semihomogeneous classes 
is
\begin{align}\notag
\cC =\{r(p^3, p^2q, pq^2, q^3) : 
(p, q, r) \in \mathbb{Z}^3, r\neq 0, \mathrm{gcd}(p, q)=1\}.  
\end{align}
For any
%\footnote{
%We always assume $\mathrm{gcd}(p,q) = 1$
%if we write $v=r(p^3, p^2q, pq^2, q^3)$
%for some $v \in \Gamma$.} 
$v=r(p^3, p^2q, pq^2, q^3) \in \cC$ define its slope by
\begin{align}\label{intro:theta}
\Theta(v) =\frac{q}{p} \in \mathbb{Q} \cup \{\infty\}
\end{align}
with the convention $\Theta(v)=\infty$ if $p=0$.
If $\gamma_1, \gamma_2 \in \Gamma$, 
then $\chi(\gamma_1, \gamma_2) \neq 0$ if and only if 
$\Theta(\gamma_1) \neq \Theta(\gamma_2)$. 
We have the following result. 

%Together with Theorem~\ref{intro:first} the following result
%completely determines the invariance of the reduced Donaldson--Thomas invariants of $(A,H)$
%under autoequivalences.

\begin{thm}\label{intro:main}
Suppose $v=\gamma_1+\gamma_2$ for some $\gamma_i \in \cC$
with $\Theta(\gamma_1)<\Theta(\gamma_2)$,
and let
\[ g = \begin{pmatrix} a&b \\ c&d \end{pmatrix} \in \mathrm{SL}_2(\BZ). \]
\begin{enumerate}
\item[(i)]
If $-\frac{d}{c} \notin [\Theta(\gamma_1), \Theta(\gamma_2))$
or $c=0$ then 
\[ \DT_H(g_{\ast}v)=\DT_H(v). \]
\item[(ii)]
If $-\frac{d}{c} \in [\Theta(\gamma_1), \Theta(\gamma_2))$
then
\[
\DT_H(v)-\DT_H(g_{\ast}v)
=
(-1)^{r_1 r_2 \alpha}
r_1 r_2 \alpha^9
\bigg( \sum_{\substack{ k_1 \ge 1 \\ k_1|r_1}} \frac{1}{k_1^2}
\bigg) \cdot
\bigg( \sum_{\substack{k_2 \ge 1 \\ k_2|r_2}} \frac{1}{k_2^2}
\bigg)
\]
where $\gamma_i=r_i(p_i^3, p_i^2 q_i, p_i q_i^2, q_i^3)$
%with $(p_i, q_i) = 1$
and $\alpha = p_1 q_2-p_2 q_1$.
\end{enumerate}
\end{thm}

\subsection{Curve counting} As before let $(A, H)$ be a  
principally polarized abelian threefold of Picard
rank $\rho(A)=1$. 
For any non-zero $(\beta, n) \in \mathbb{Z}^2$ define
\begin{align*}
\DT_{\beta, n}=\DT_H(1, 0, -\beta, -n). 
\end{align*}
The invariant $\DT_{\beta, n}$ enumerates
algebraic curves $C \subset A$ with $[C]=\beta H^2/2$ and $\chi(\oO_C)=n$ up to translation.

A conjecture for $\DT_{\beta, n}$ was proposed in \cite[Section~7.6]{BOPY} as follows.
Define the theta functions
\[
\theta_2(q) = \sum_{n \in \BZ} q^{(n+\frac{1}{2})^2}, \ \
\theta_3(q) = \sum_{n \in \BZ} q^{n^2}.
\]
Let $\aaa(n) \in \BZ$ be defined by the Fourier expansion
%be the Fourier coefficients
%of the following meromorphic $\Gamma_0(4)$ modular form of weight $-5/2$,
\begin{multline*}
\sum_{n} \aaa(n) q^n
=
\frac{-16}{\theta_2(q)^4 \theta_3(q)}
=
- q^{-1} + 2 - 8q^{3} + 12q^{4} - 39q^{7} + 56q^{8} + \ldots \ .
%- 152q^{11} + 208q^{12} - 513q^{15} + 684q^{16} - 1560q^{19} + 2032q^{20} - 4382q^{23} %+ 5616q^{24} - 11552q^{27} + 14592q^{28} - 28899q^{31} + 36088q^{32} - 69168q^{35} + 85500q^{36} + O(q^{39})
\end{multline*}
Let also 
$\mathsf{n}(\beta, k)= \sum_{\delta} \delta^2$
where $\delta$ runs over all positive divisors
of $k, \beta, \beta^2/k$ and $\beta^3/k^2$ if these numbers are integers,
and let $\mathsf{n}(\beta,k) = 0$ otherwise.

If $\beta < 0$, or $\beta = 0$ and $n < 0$ the invariant
$\DT_{\beta, n}$ vanishes since the moduli space is empty.
In all other cases we have the following. 

\begin{conj}[\cite{BOPY}] \label{intro:conj} \label{conj:DT}
Assume $\beta > 0$, or $\beta =0$ and $n > 0$. Then 
\begin{equation} \label{def:C}
\DT_{\beta, n}
=(-1)^{n}\sum_{\substack{k \ge 1 \\ k|n}}
\frac{1}{k} \mathsf{n}(\beta, k)\aaa \left(\frac{4\beta^3-n^2}{k^2}  \right)
\end{equation}
\end{conj}

We have the following corollary of Theorem~\ref{intro:main}.
\begin{cor} \label{intro:corr_evidence}
Let $(\beta, n) \in \mathbb{Z}^2$ be non-zero, and suppose $(c, d)$ is an integer solution of the equation $d^3-3\beta c^2d-nc^3=1$. Define
\begin{align*}
(\beta', n')=(d^2\beta+ncd+\beta^2 c^2, 6\beta^2 d^2 c +6c^2 d \beta n+n+2c^3 n^2-2c^3\beta^3). 
\end{align*}
If $4\beta^3-n^2 \ge 0$, then $\DT_{\beta, n} = \DT_{\beta', n'}$, and moreover $\DT_{\beta,n}$ satisfies Conjecture~\ref{conj:DT} if and only if $\DT_{\beta', n'}$ does.
\end{cor}

In Corollary \ref{intro:corr_evidence} the pairs $(\beta,n)$ and $(\beta',n')$ are related by a derived autoequivalence.
The discriminant specializes to $\Delta = 4 \beta^3 - n^2$. % see Section~\ref{Subsection_ppav_curve_counting} for a proof.

Corollary~\ref{intro:corr_evidence} yields evidence for Conjecture~\ref{conj:DT}.
In particular, calculations for primitive curve classes (which are easier) yield informations for imprimitive curve classes.
For example,
for $(\beta,n) = (1,1)$ and $(c,d) = (1,2)$ we obtain the non-trivial relation
\[ \DT_{7, 37} = \DT_{1,1} = 8 \]
where the last equality follows by a direct computation. % or by results in \cite{BOPY}.
%where the last equality can be verified directly.b

If $\Delta$ is negative, then $\DT_{\beta, n}$ and $\DT_{\beta', n'}$ differ by the wall-crossing contributions of Theorem~\ref{intro:main}.
We have checked in many cases (using a computer program) that the right hand side of Conjecture~\ref{conj:DT} satisfies the same wall-crossing behaviour.
This yields non-trivial evidence for Conjecture~\ref{conj:DT} also in the critical range where the discriminant is negative.
We refer to Section~\ref{Subsection_ppav_curve_counting} for further discussions and a proof of Corollary~\ref{intro:corr_evidence}.

The constraints obtained from Theorem~\ref{intro:first} are strongest for abelian threefolds with higher Picard number,
since these have a large group of derived autoequivalences.
The conjecture in \cite[Section~7.6]{BOPY} applies to curve counting invariants of arbitrary abelian threefolds.
It would be interesting to show the compatibility of the \cite{BOPY} conjecture with Theorem~\ref{intro:first} in general.
Another interesting direction is to use Theorem~\ref{intro:first} to extend the \cite{BOPY} conjecture to arbitrary primitive vectors $v \in \Gamma$.

\subsection{Idea of the proof of Theorem~\ref{intro:first}}
%The proof relies on the construction of reduced Donaldson--Thomas invariants and the theory of Bridgeland stability conditions.
Reduced Donaldson--Thomas invariants are defined by making the motivic Hall algebra and the integration map
equivariant with respect to the action of $\A := A \times \hat{A}$.\footnote{See also \cite{OS} for equivariant Hall algebras and a definition in a simpler case.}
The equivariant integration map (defined in Section~\ref{Subsection_Equivariant_integration_map}) takes values in the ring
\[ \BQ[\A] = \bigoplus_{\substack{B \subset \A \text{ connected} \\ \text{abelian subvarieties}}} \BQ \epsilon_{B}, \]
where the ring structure is defined in terms of the intersection of the subvarieties $B$.
For example, if $Z$ is a variety with $\A$-action and $Z_B \subset Z$ is the stratum of points whose stabilizers contain $B$ with finite index,
then its equivariant integral is the polynomial
\[ \mathbf{e}(Z) = \sum_{\substack{B \subset \A}} e([Z_B / (\A/B)] ) \epsilon_B. \]
Applying the integration map to moduli spaces of semistable sheaves (or certain linear combinations thereof)
yields the Donaldson--Thomas polynomial $\DTb_H(v) \in \BQ[\A]$.
% which is well-behaved under taking
%products and wall-crossing, and 
Its coefficient of $\epsilon_0$ is the reduced invariant $\DT_H(v)$.
Similarly %by \cite{PiYT} 
for every Bridgeland stability condition $\sigma \in \Stab(A)$
there is an invariant $\DTb_{\sigma}(v) \in \BQ[\A]$ counting $\sigma$-semistable objects of Chern character~$v$.
For every autoequivalence $g$ we have formally
\begin{equation} \DTb_{\sigma}(v) = \DTb_{g_{\ast} \sigma}( g_{\ast} v). \label{invariance}\end{equation}

%the definition relies in particular on the construction of a finite type moduli space of these objects in \cite{PiYT}.

In this paper we prove the following steps:
\begin{enumerate}
\item[(i)] Stability conditions on $A$ constructed by Maciocia--Piyaratne \cite{MaPi1, MaPi2} and Bayer--Macr\'i--Stellari \cite{BMS}
satisfy the full support property.
Hence they define a family in the stability manifold $\Stab(A)$.
In particular $\Stab(A) \neq \varnothing$.
The connected component $\Stab^{\circ}(A) \subset \Stab(A)$ which contains this family is called the main component.
\item[(ii)] $\Stab^{\circ}(A)$ is preserved by all autoequivalences.
\item[(iii)] (Gieseker chamber) For every $H$ and $v$, there exist a $\sigma \in \Stab^{\circ}(A)$ such that $\DTb_{\sigma}(v) = \DTb_H(v)$.
\item[(iv)] If $v$ can not be written as a sum of two semihomogeneous classes, then
all wall-crossing contributions vanish. In particular, $\DT_{\sigma}(v)$ is independent of $\sigma \in \Stab^{\circ}(A)$.
\end{enumerate}
We conclude
\[ 
\pushQED{\qed} 
\DTb_{H}(v) \overset{(iii)}{=} \DTb_{\sigma}(v) \overset{\eqref{invariance}}{=} \DTb_{g_{\ast} \sigma}(g_{\ast} v) \overset{(ii) + (iv)+(iii)}{=} \DTb_{H}(g_{\ast}v) \qedhere
\popQED
\]

\subsection{Plan of the paper}
In Section~\ref{section:reducedDT} we define the integration map for equivariant motivic hall algebras
and reduced Donaldson--Thomas invariants.
%generalizing earlier work of Gulbrandsen \cite{Gul} and \cite{OS}.
In Section~\ref{section:Bridgeland_stab_conditions} we prove the full support property
for certain Bridgeland stability conditions on abelian threefolds
and show the existence of a Gieseker chamber.
In Section~\ref{section:Wallcrossing_on_abelian_threefolds} we
define reduced Donaldson--Thomas invariants for Bridgeland semistable objects,
and discuss their wall-crossing behaviour. This leads to a proof of Theorem~\ref{intro:first}.
In Section~\ref{section:ppav} we specialize to principally polarized abelian threefolds and prove Theorem~\ref{intro:main}.
In Appendix~\ref{Appendix_Spin_representations} we discuss the discriminant function and spin representations.

\subsection{Conventions}
We always work over $\BC$ and all schemes are assumed to be locally of finite type. %Let $G$ be a subgroup of an abelian variety $A$.
Given an algebraic group $G$ we let $G^{\circ}$ denote the connected component of $G$ which contains the origin. % we say that $G^{\circ}$ is the (distinguished) connected component.
For a derived auto-equivence $g \in \Aut D^b(X)$ we let $g_{\ast}$ denote its induced action on cohomology.

\subsection{Acknowledgments}
The paper was started when the first author was visiting the last two authors at
Kavli IPMU  in March 2017, and  
completed
%a part of this work was done 
when the second author was visiting Kavli IPMU in Summer 2018.
We thank the institute for support and hospitality. 
We would also like to thank Davesh Maulik and Junliang Shen for discussions on abelian threefolds, and
Catharina Stroppel for helpful advice on spin representations.

G.~O.~was supported by the National Science Foundation Grant DMS-1440140 while in residence at MSRI, Berkeley.
D.~P.~was supported by World Premier International Research Center Initiative (WPI initiative), MEXT, Japan. 
Y.~T.~ is supported by World Premier International Research Center
Initiative (WPI initiative), MEXT, Japan, and Grant-in Aid for Scientific
Research grant (No. 26287002) from MEXT, Japan.

\section{Reduced Donaldson--Thomas invariants}\label{section:reducedDT}
\subsection{Overview}
Let $X$ be a smooth projective Calabi--Yau threefold equipped with an action of an abelian variety $A$.
The product
\[ \A = A \times \Pic^0(X) \]
acts on the moduli spaces of Gieseker semistable sheaves on $X$ by translation by elements in $A$ and tensor product with elements in $\Pic^0(X)$.
The goal of Section~\ref{section:reducedDT} is to define reduced (generalized) Donaldson--Thomas
invariants of $X$ which count $\A$-orbits of Gieseker semistable sheaves.
%The construction lifts ideas of \cite{OS} for actions by simple abelian varieties to arbitrary ones.
For abelian threefolds our definition generalizes work of Gulbrandsen \cite{Gul} and \cite{OS}.
However the definition is not special to abelian threefolds. A list of examples to keep in mind is the following:
\begin{itemize}
\item $X = A$ is an abelian threefold and $\A = A \times \widehat{A}$.
\item $X = S \times E$ with $S$ a K3 surface, $E$ an elliptic curve, and $\A = E \times \widehat{E}$.
\item $X = (S \times E)/G$ where $S$ is a symplectic surface, $E$ is an elliptic curve, and
$G$ is a finite group acting on $S$ by symplectic automorphisms, on $E$ by translation by torsion points,
and such that the induced diagonal action on $S \times E$ is free.
The $E$-action on $S \times E$ descends to an action on the quotient, and we can take $\A = E \times \widehat{E/G}$.
%\footnote{If $S$ is a K3 surface and $G = \BZ_n$ these $X$ are the CHL models. Their reduced Donaldson--Thomas theory is studied in \cite{BrOb-CHL}.}
\item $X$ is a Calabi--Yau threefold with $h^{1,0}(X)>0$, and $\A = \Pic^0(X)$.
\end{itemize}

In Section~\ref{Subsection_equivariant_Grothendieckgroup_of_var} and Section~\ref{subsection_eq_gr_stacks} we
discuss equivariant Grothendieck groups of varieties and stacks respectively.
This leads to the definition of the equivariant Hall algebra in Section~\ref{subsec:hall}.
In Section~\ref{Subsection_Equivariant_integration_map} we begin the construction of the equivariant integration map.

\subsection{Equivariant Grothendieck group of varieties} \label{Subsection_equivariant_Grothendieckgroup_of_var}
Let $A$ be an abelian variety. 
Following~\cite[Section~3]{OS}
the $A$-equivariant Grothendieck group of 
varieties $K_0^{A}(\mathrm{Var})$ is the free 
abelian group generated by the classes
\[ [X, a_X] \]
of a variety $X$ with an $A$-action 
$a_X \colon A \times X \to X$, modulo 
the equivariant 
scissor relations
\begin{align*}
[X, a_X]=[Z, a_X|_{Z}]+[U, a_X|_{U}]
\end{align*}
for every $A$-invariant 
closed subvariety $Z \subset X$ with $U=X \setminus Z$.
Taking products of varieties with the induced diagonal $A$-action
endows $K_0^{A}(\mathrm{Var})$ with the structure of a commutative ring with unit.

Consider the 
$\mathbb{Q}$-vector space
\begin{align*}
\mathbb{Q}[A]=\bigoplus_{B \subset A} \mathbb{Q} \epsilon_B
\end{align*}
where $B$ runs over all \emph{connected} abelian subvarieties of $A$.
We define a $\BQ$-linear ring structure
on $\mathbb{Q}[A]$ as follows.
If connected abelian subvarieties $B_1, B_2 \subset A$ intersect transversely, i.e. 
\begin{align*}
\mathrm{codim}(B_1 \cap B_2)= \mathrm{codim}B_1+\mathrm{codim}B_2,
\end{align*}
we set
\begin{align*}
\epsilon_{B_1} \cdot \epsilon_{B_2}=
\left| \frac{B_1 \cap B_2}{(B_1 \cap B_2)^{\circ}} \right|
\epsilon_{(B_1 \cap B_2)^{\circ}}.
\end{align*}
%where $(B_1 \cap B_2)^{\circ}$ is the connected 
%component of $B_1 \cap B_2$ which contains the origin. 
If $B_1, B_2$ are not transverse we set 
\[ \epsilon_{B_1} \cdot \epsilon_{B_2}=0. \]
%In particular, multiplication by $\epsilon_A$ is the identity.

\begin{lemma}
$(\mathbb{Q}[A], \cdot)$ is an
associative commutative algebra with unit $\epsilon_{A}$.
\end{lemma}
\begin{proof}
The key step is to prove associativity: Let $B_1, B_2, B_3 \subset A$
be connected. Then
$( \epsilon_{B_1} \cdot \epsilon_{B_2} ) \cdot \epsilon_{B_3}$ is
non-zero if and only if
\[ \mathrm{codim}(B_1 \cap B_2 \cap B_3) =
\mathrm{\codim}(B_1) + \mathrm{codim}(B_2) + \mathrm{codim}(B_3)\]
in which case we get
\[ ( \epsilon_{B_1} \cdot \epsilon_{B_2} ) \cdot \epsilon_{B_3}
= \left| \frac{ B_1 \cap B_2 \cap B_3 }{ (B_1 \cap B_2 \cap B_3)^{\circ}} \right|
\epsilon_{(B_1 \cap B_2 \cap B_3)^{\circ}}.
\]
In particular, the right hand side is invariant under permutation. % and hence the product is associative.
\end{proof}

Let $X$ be a variety with $A$-action $a_X$. For any abelian subvariety $B \subset A$ let
$X_B \subset X$
denote the (reduced) locally closed subscheme of points
whose stabilizer contain $B$ with finite index,
\begin{align*}
X_B = \{ x \in X : \Stab(x) \supset B, 
\, \left| \Stab(x)/B \right| < \infty\}. 
\end{align*}
The subscheme $X_B \subset X$ is $A$-invariant and 
the induced $A$-action on $X_B$ descends to an $A/B$-action
with finite stabilizers.
%\footnote{We use here that if an abelian variety $B$ acts on a variety $Y$ such that
%the stabilizer is $B$ at every point, then the action is trivial. This follows from Hilbert basis theorem.}
The quotient stack 
\begin{align*}
[X_B/(A/B)]
\end{align*}
is hence Deligne--Mumford and
its (topological) Euler characteristic
is well-defined as a rational number.

We define the $A$-reduced Euler characteristic of the class $[X,a_X]$ by
\begin{align*}
\e([X, a_X]) \cneq \sum_{B \subset A} e\big( [X_B/(A/B)] \big) \epsilon_B \, \in 
\mathbb{Q}[A]
\end{align*}
where the sum runs over all connected abelian subvarieties of $A$.

\begin{lem} The $\BQ$-linear map
\begin{align*}
\e \colon K_0^{A}(\mathrm{Var}) \to \mathbb{Q}[A],
\ [X, a_X] \mapsto \e([X, a_X])
\end{align*}
is a ring homomorphism.
\end{lem}
\begin{proof}
Since the $A$-reduced Euler characteristic respects the $A$-equivariant scissor relation, the map $\e$ is well-defined.
We need to show it is a ring homomorphism.
Let $X_1, X_2$ be varieties with $A$-actions. By a stratification argument
we may assume $X_i = (X_{i})_{B_i}$ for some connected abelian 
subvarieties $B_i \subset A$.
With respect to the diagonal $A$-action we have
\[ \mathrm{Stab}\big( (x_1, x_2) \big) = \mathrm{Stab}(x_1) \cap \mathrm{Stab}(x_2) \]
for all $(x_1, x_2) \in X_1 \times X_2$, and hence
%\[ (X \times Y) = (X \times Y)_{(B_1 \cap B_2)^{\circ}} \]
%and 
\[ \e( [X_1 \times X_2, a_{X_1 \times X_2} ] ) = c \, \epsilon_{(B_1 \cap B_2)^{\circ}} \]
for some $c \in \BQ$. We need to show
\[
c =
\left| \frac{B_1 \cap B_2}{(B_1 \cap B_2)^{\circ}} \right|
e\big( [X_1/ (A/B_1)] \big) e\big( [X_2 / (A/B_2)] \big) 
\]
if $B_1$ and $B_2$ are transverse, and $c=0$ otherwise.

Consider the
commutative diagram
of rows of exact sequences of abelian groups,
\[
\begin{tikzcd}
0 \arrow{r} & B_1 \cap B_2 \ar{d}  \ar{r} & A \ar{r} \ar{d}{\Delta} & A / (B_1 \cap B_2) \ar{r} \ar{d} & 0 \\
0 \ar{r} & B_1 \times B_2 \ar{r} & A \times A \ar{r} & A/B_1 \times A/B_2 \ar{r} & 0.
\end{tikzcd}
\]
Since the left hand square is fibered the induced morphism
\[ \alpha : ( B_1 \times B_2 ) / (B_1 \cap B_2) \to A \]
is injective, and we obtain the exact sequence
\[
0 \to A / (B_1 \cap B_2) \to
A/B_1 \times A/B_2 \to \mathrm{Coker}(\alpha) \to 0.
\]
The subvarieties $B_1$ and $B_2$ are transverse
if and only if the addition map
\[ B_1 \times B_2 \to A, (b_1, b_2) \mapsto b_1 + b_2 \]
is surjective, hence if and only if $\mathrm{Coker}(\alpha) = 0$.
If $B_1$ and $B_2$ are not transverse the quotient
\[ \left[ (X_1 \times X_2) \Big/ (A / (B_1 \cap B_2)^{\circ}) \right] \]
hence carries an action by the positive-dimensional
abelian variety $\mathrm{Coker}(\alpha)$ and therefore its Euler characteristic is zero; this implies $c=0$.
%\[ e\left( ( X_1 \times X_2 ) \Big/ (A / (B_1 \cap B_2)) \right) = 0 \]
%which implies $c=0$.
If $B_1$ and $B_2$ are transverse,
we get $A / (B_1 \cap B_2) = A/B_1 \times A/B_2$ and so
\begin{align*}
c & = e\left( \left[( X_1 \times X_2 ) \Big/ (A / (B_1 \cap B_2)^{\circ})\right] \right) \\
& = \left| \frac{B_1 \cap B_2}{(B_1 \cap B_2)^{\circ}} \right|
e\left(\left[ ( X_1 \times X_2 ) \Big/ (A / (B_1 \cap B_2))\right] \right) \\
%& = \left| \frac{B_1 \cap B_2}{(B_1 \cap B_2)^{\circ}} \right|
%e( (X_1 \times X_2) / (A/B_1 \times A/B_2) ) \\
& = \left| \frac{B_1 \cap B_2}{(B_1 \cap B_2)^{\circ}} \right|
e([X_1 / (A/B_1)]) e([X_2/ (A/B_2)]). \qedhere
\end{align*}
\end{proof}

\subsection{Preliminaries on stacks} \label{subsec:pre lim stacks}
We will follow Bridgeland \cite{BrH} for conventions on stacks.
In particular, all stacks are assumed to be algebraic and locally of finite type with affine geometric stabilizers.
Geometric bijections and Zariski fibrations of stacks are defined in \cite[Definition~3.1]{BrH} and \cite[Definition~3.3]{BrH}.
We also refer to \cite{MR3237451} for the definition of good moduli spaces and their properties.

For us an action of an algebraic group $G$ on a stack $\CX$ is a morphism 
\[ \rho \colon G \times \CX \to \CX \]
such that for all $\BC$-valued points $g_1, g_2 \in G$ there exists an isomorphism
\[ \rho_{g_1} \rho_{g_2} \cong \rho_{g_1 g_2}. \]
Hence we work with group actions on a stack in a very weak sense,
%In this sense, we work with the weakest possible sense of a group action, 
see also \cite{MR2125542} for a stronger definition.
If the stack $\CX$ has a good moduli space $X$, then by universal property the group action on $\CX$ induces a group action on $X$.
If $X$ is a reduced scheme (or algebraic space),
where both here are always assumed to be locally of finite type over $\BC$,
 %(always taken here of finite type over $\BC$),
then %since these are here always assumed to be locally of finite type over $\BC$,
a group action on $X$ induces a group action on the scheme in the usual sense (i.e.\ where we require that
the two possible maps $G \times G \times X \to X$ are equal).

In this paper we are only interested in taking the topological Euler characteristic,
hence when working with the good moduli space of a stack
we can always pass to the underlying reduced scheme.
Unless not stated otherwise, this will be done implicitly throughout.

\begin{rmk}
This technical discussion is necessary since for an abelian variety $A$
the group $G = A \times \hat{A}$ does \emph{not} act on the stack $M$
of coherent sheaves on $A$ in the sense of Romagny \cite{MR2125542} (see \cite[Sec.\ 3.6]{BONotes} for an example), and not even in the weaker sense where we only require the two maps $G \times G \times M \to M$ to be isomorphic (for example restrict the action to the point corresponding to the structure sheaf $\CO_A$).
This issue can be resolved by working with the group stack
$\mathrm{Auteq}^0(A)$ which is a $\BG_m$-gerbe over $A \times \hat{A}$ and defined as the open substack of the moduli stack $Coh(A \times A)$ corresponding to deformations of the diagonal.
Then $\mathrm{Auteq}^0(A)$ acts on the stack of coherent sheaves $M$ and induces an $A \times \hat{A}$ action
on each good moduli space without passing to the reduced subscheme.
However, to avoid unnecessary technicalities, we will work with the weaker definition of stacks defined above.
\end{rmk}

\subsection{Equivariant Grothendieck group of stacks} \label{subsection_eq_gr_stacks}
Let $A$ be an abelian variety, and
let $\mathcal{S}$ be an algebraic stack
equipped with an $A$-action $a_{\CS}$.

\begin{defn} \label{defn_grothringofstacks}
The $A$-equivariant relative Grothendieck group of stacks $K^A_0(\mathrm{St}/\CS)$ is defined to be the $\BQ$-vector space generated by the classes
\begin{equation*}
[\CX \xrightarrow{f} \CS, a_{\CX}],
\end{equation*}
where $\CX$ is an algebraic stack of finite type,
$a_{\CX}$ is an $A$-action on $\CX$,
and $f$ is an $A$-equivariant morphism,
modulo the following relations:

\begin{enumerate}
\item[(a)] For every pair of stacks $\CX_1$ and $\CX_2$ with $A$-actions $a_1$ and $a_2$ respectively a relation
\[
[\CX_1 \sqcup \CX_2 \xrightarrow{f_1 \sqcup f_2} \CS, a_{1}\sqcup a_2] = [\CX_1 \xrightarrow{f_1} \CS, a_1]+ [\CX_2 \xrightarrow{f_2} \CS, a_2]
\]
where $f_i$ ($i=1,2$) are $A$-equivariant.

\item[(b)]
For every commutative diagram 
\[
\begin{tikzcd}
\CX_1 \arrow{rr}{g} \arrow[swap]{dr}{f_1}& &\CX_2 \arrow{dl}{f_2}\\
& \CS & 
\end{tikzcd}    
\]
with all morphisms $A$-equivariant
and $g$ a geometric bijection
a relation
\[
[\CX_1 \xrightarrow{f_1} \CS, a_1] = [\CX_2 \xrightarrow{f_2} \CS, a_2].
\]
\item[(c)]
Let $\CX_1, \CX_2, \CY$ be stacks equipped with $A$-actions $a_1, a_2, a_Y$. % respectively.
%such that the stabilizer groups of $a_1, a_2, a_Y$ at all $\BC$ points have the same connected component, i.e.
%\[ \Stab_{a_1}(x_1)^{\circ} = \Stab_{a_2}(x_2)^{\circ} = \Stab_{a_Y}(y)^{\circ} \]
%for all $x_1 \in \CX_1(\BC), x_2 \in \CX_2(\BC), y \in \CY(\BC)$.
%where we let $G^{\circ}$ denote the connected subgroup of $G \subset A$ containing the zero.
Then for every pair of $A$-equivariant Zariski fibrations
\[
h_1: \CX_1 \rightarrow \CY, \ \quad h_2: \CX_2 \rightarrow \CY
\]
with the same fibers and for every $A$-equivariant morphism $\CY \xrightarrow{g} \CS$, a relation
\[
\pushQED{\qed}
[\CX_1 \xrightarrow{g\circ h_1} \CS ,a_1]
= [\CX_2 \xrightarrow{g\circ h_2} \CS ,a_2].
\qedhere \popQED
\]
\end{enumerate}
\end{defn}

\begin{rmk}
If $A$ is the the trivial group
Definition~\ref{defn_grothringofstacks} specializes to the
relative Gro\-then\-dieck group of stacks
defined by Bridgeland \cite[Definition~3.10]{BrH}.
In this case we will usually omit $A$ from the notation, and will write $K_0(\mathrm{St}/\CS)$.
We will follow the same convention throughout the section: the trivial abelian variety is omitted from the notation.
\end{rmk}
\begin{rmk} \label{Rmk_gsgfsg}
For any connected abelian subvariety $B \subset A$ the restriction of $A$-actions to $B$-actions
induces a morphism
\[ K^A_0(\mathrm{St}/\CS) \to K^B_0(\mathrm{St}/\CS). \]
In particular, if $B$ is the trivial abelian variety,
\[ \mathrm{Forg}: K_0^A(\mathrm{St}/\CS) \to K_0(\mathrm{St}/\CS), \]
is the map that forgets the equivariant structure.
\end{rmk}

\subsection{Non-equivariant Hall algebras}
Let $X$ be a Calabi--Yau threefold, i.e. a non-singular projective threefold with $K_X = 0$.
Let $\CM$ be the stack of coherent sheaves on $X$.
By \cite[4.2]{BrH} the Hall algebra of $X$ is the group
\[ H(X) := K_0( \mathrm{St} / \CM ) \]
together with the associative product $\ast$ defined by extension of sheaves.

Consider the polynomial ring
\begin{align*}
\Lambda=K_0(\mathrm{Var})[\mathbb{L}^{-1}, (1 + \BL + \cdots + \BL^n)^{-1}, n\ge 1]
\end{align*}
where $\mathbb{L} = [\BA^1] \in K_0(\mathrm{Var})$ is the class of the affine line.
The subalgebra of \emph{regular classes} is the $\Lambda$-submodule 
\begin{align*}
H_{\rm{reg}}(X) \subset H(X)
\end{align*}
generated by all classes $[Z \to \CM]$ where $Z$ is a variety.
In particular, $H_{\rm{reg}}(X)$ is closed under $\ast$-product.
The quotient
\[ H_{\rm{sc}}(X) = H_{\rm{reg}}(X) / (\BL - 1) H_{\rm{reg}}(X) \]
is called the semi-classical limit and is commutative with respect to $\ast$. The Poisson bracket
defined by
\begin{equation} \{ f,g \} := \frac{ f \ast g - g \ast f }{ \BL - 1 }, \quad f,g \in H_{\text{sc}}(X) \label{Poisson_bracket}\end{equation}
makes
$H_{\text{sc}}(X)$
a Poisson algebra with respect to 
$(\ast, \{-, -\})$. 

\subsection{Equivariant Hall algebras}\label{subsec:hall}
Let $X$ be a Calabi--Yau threefold equipped with the action of an abelian variety $A$.
The group
\[ \mathbf{A} := A \times \Pic^0(X) \]
acts on the stack of coherent sheaves $\CM$ on $X$ by
\[
(a, \CL) \cdot E =T_a^{\ast}E \otimes \CL
\quad \text{for all } a \in A, \CL \in \Pic^0(X), E \in \Coh(X).
\]
The $\A$-equivariant motivic Hall algebra is the group
\begin{align*}
H^{\A}(X) \cneq K_0^{\A}(\mathrm{St}/\CM).
\end{align*}
The product $\ast$ lifts canonically to an associative product on $H^\A(X)$ via the diagonal action, see \cite[Section~4.6]{OS}.
The forgetful morphism of Remark~\ref{Rmk_gsgfsg},
\[
\Forg : H^{\A}(X) \to H(X),
\]
is a ring homomorphism with respect to this product.

Define the subalgebra of regular classes by
\begin{equation} \label{defn_reg}
H^{\A}_{\mathrm{reg}}(X) := \Forg^{-1}( H_{\mathrm{reg}}(X) ).
\end{equation}
The semi-classical limit is the quotient
\[ H^{\A}_{\text{sc}}(X) = H_{\mathrm{reg}}^{\A}(X) / (\BL-1) H_{\mathrm{reg}}^{\A}(X). \]
By an argument parallel to \cite[Proposition~2]{OS} the algebra $H^{\A}_{\text{sc}}(X)$ is
commutative and the bracket $\{- ,-\}$ defined in \eqref{Poisson_bracket} lifts to a Poisson bracket on $H^{\A}_{\text{sc}}(X)$.\footnote{The condition (c) in Definition~\ref{defn_grothringofstacks} is used crucially here.}
Therefore $H_{\text{sc}}^{\A}(X)$ is a Poisson 
algebra with respect to $(\ast, \{-, -\})$.

\subsection{Gieseker stability}\label{subsec:Gieseker}
Let $H$ be a fixed polarization on $X$. 
For a sheaf $E \in \Coh(X)$, its \textit{Hilbert polynomial} is 
\begin{align*}
\chi(E \otimes \oO_X(mH))=a_d m^d+a_{d-1}m^{d-1}+\cdots
\end{align*}
where $a_i \in \mathbb{Q}$, 
$d=\dim \Supp(E)$ and $a_d$ is a positive rational number. 
The \textit{reduced Hilbert polynomial} is defined by
\begin{align*}
\overline{\chi}_H(E)
\cneq \frac{\chi(E \otimes \oO_X(mH))}{a_d} \in \mathbb{Q}[m].
\end{align*}
Let $\Gamma$ be the image of the Chern character map
\begin{align*}
\Gamma := \Imm\left( \ch \colon K(X) \to H^{2\ast}(X, \mathbb{Q})\right).
\end{align*}
Since $\overline{\chi}_H(E)$ only depends on the Chern character of 
$E$, there is a map 
$\overline{\chi}_H \colon \Gamma \to \mathbb{Q}[m]$
such that $\overline{\chi}_H(E)=\overline{\chi}_H(\ch(E))$.

The reduced Hilbert polynomial is used in the definition of Gieseker stability as follows. 
\begin{defi}
An object $E \in \Coh(X)$ is $H$-Gieseker (semi)stable 
if it is pure and for any non-zero subsheaf 
$F \subsetneq E$, we have 
\[ \overline{\chi}_H(F)(m) <(\le) \  \overline{\chi}_H(E)(m) \]
for $m\gg 0$. 
\end{defi}
Let $\Gamma_{+} \subset \Gamma$ be the set
of Chern characters of coherent sheaves,
\begin{align*}
\Gamma_{+} \cneq \Imm(\ch|_{\Coh(X)} \colon \Coh(X) \to \Gamma). 
\end{align*}
For any $v \in \Gamma_{+}$ let
\begin{align}\label{moduli:semistable}
\mM_H(v) \subset \mM
\end{align}
be the open substack of finite type parametrizing $H$-Gieseker semistable sheaves with 
Chern charecter $v$. For any fixed $\overline{\chi} \in \mathbb{Q}[m]$
consider the union
\begin{align*}
\mM_H(\overline{\chi})
=\coprod_{\overline{\chi}_H(v)=\overline{\chi}} \mM_H(v).
\end{align*}
The Hall algebra of semistable sheaves with reduced Hilbert polynomial $\overline{\chi}$ is defined by
\begin{align}\label{Hall:semistable}
H(X, \overline{\chi}) \cneq K_0(\mathrm{St}/\mM_H(\overline{\chi})). 
\end{align}
Since the category of $H$-Gieseker semistable 
sheaves with fixed reduced Hilbert polynomial is 
closed under extension, the natural inclusion map 
\begin{equation}
H(X, \overline{\chi}) \hookrightarrow H(X) \label{gdgsdfsd}
\end{equation}
is a ring homomorphism.
As before the Hall algebra $H(X, \overline{\chi})$
has a subalgebra of regular classes (the $\Lambda$-module generated by all $[Z \to \CM_H(\overline{\chi})]$ where $Z$ is a variety)
and a semi-classical limit. We have the natural inclusions\footnote{The subalgebra of regular classes could also be defined as the preimage of $H_{\mathrm{reg}}(X)$ under the inclusion \eqref{gdgsdfsd}, and similarly for the semi-classical limit.}
\begin{align}\label{regular:sc}
H_{\mathrm{reg}}(X, \overline{\chi}) \subset 
H_{\mathrm{reg}}(X), \  \ \ 
H_{\mathrm{sc}}(X, \overline{\chi}) \subset 
H_{\mathrm{sc}}(X).
\end{align}

Since \eqref{moduli:semistable} is $\A$-equivariant,
there exists an $\A$-equivariant version of (\ref{Hall:semistable}),
\[ H^{\A}(X, \overline{\chi}) \subset H^{\A}(X). \]
Similarly one has $\A$-equivariant versions of (\ref{regular:sc}),
\begin{align*}
H_{\mathrm{reg}}^{\A}(X, \overline{\chi}) \subset 
H_{\mathrm{reg}}^{\A}(X), \ \ \
H_{\mathrm{sc}}^{\A}(X, \overline{\chi}) \subset 
H_{\mathrm{sc}}^{\A}(X).
\end{align*}

\subsection{Poisson torus}
By the Riemann-Roch theorem, the Euler paring
\begin{align*}
\chi(E, F) \cneq \sum_{i \in \mathbb{Z}}
(-1)^i \dim \Ext^i(E, F), \quad E, F \in D^b(X)
\end{align*}
descends to a unique bilinear form 
\[ \chi \colon \Gamma \times \Gamma \to \Gamma \]
which satisfies $\chi(E, F)=\chi(\ch(E), \ch(F))$.
Consider the group
\[ C^{\A}(X) := \bigoplus_{v \in \Gamma} \BQ[\A] \cdot c_v. \]
An associative product $\ast$ and a Poisson bracket on $C^\A(X)$ are defined by
\begin{align*}
c_{v_1} \ast c_{v_2} & := (-1)^{\chi(v_1,v_2)} c_{v_1+v_2} \\
\{ c_{v_1}, c_{v_2} \} & := (-1)^{\chi(v_1,v_2)} \chi(v_1,v_2) c_{v_1+v_2}.
\end{align*}
Then $C^{\A}(X)$ is a Poisson algebra 
with respect to the above $(\ast, \{-, -\})$.

\subsection{Equivariant integration map: Overview} \label{Subsection_Equivariant_integration_map}
Recall from \cite{BrH} the integration map
\[ 
\CI : H_{\text{sc}}(X) \to C(X). %:=\bigoplus_{v \in \Gamma}\mathbb{Q} \cdot c_v.
\]
The map $\CI$ is a Poisson algebra homomorphism with respect to 
$(\ast, \{ - , - \})$
such that for every $Z \to \CM(v)$ with $Z$ a variety we have
\[ \CI([Z \stackrel{f}{\to} \CM(v)]) = e( Z, f^{\ast} \nu ) c_{v} = \left( \int_{Z} f^{\ast} \nu \ \mathrm{d}e \right) c_{v}. \]
Here $\nu : \CM \to \BZ$ is the Behrend function~\cite{Beh}
and $\mM(v) \subset \mM$ is the substack of sheaves with Chern character $v$.

For each $\overline{\chi} \in \mathbb{Q}[m]$, let $H_{\mathrm{sc}}(X, \overline{\chi})$, 
$H_{\mathrm{sc}}^{\A}(X, \overline{\chi})$
be the semi-classical limits of Hall algebras of semistable sheaves 
with reduced Hilbert polynomial $\overline{\chi}$ as defined in Section~\ref{subsec:Gieseker}.
The integration map $\iI$ restricts to the 
Poisson algebra homomorphism 
\begin{align*}
\iI \colon H_{\mathrm{sc}}(X, \overline{\chi}) \to C(X).
\end{align*}
The goal of the next section is to define an equivariant integration map
\[ \CI^\A : H^{\A}_{\text{sc}}(X, \overline{\chi}) \to C^{\A}(X) \]
which is a Poisson algebra homomorphism with respect to 
$(\ast, \{-, - \})$ such that
\begin{equation} \CI^\A([Z \xrightarrow{f} \CM_H(v), a]) =
\left(\sum_{B \subset \A}
(-1)^{\dim \A/B}  \epsilon_B
\int_{[Z_B/(\A/B)]}
f^{\ast} \nu \ \mathrm{d}e \right) c_v,
%= \e( [Z,a], f^{\ast} \nu ) c_{v}, 
\label{defegdf}
\end{equation}
for every $\A$-equivariant map $Z \xrightarrow{f} \CM_H(v)$, where $Z$ is a variety.

If an equivariant regular class $\alpha$ can be written ($\A$-equivariantly) as a $\Lambda$-linear combination of classes $[Z_i \to \CM,a_i]$ with $Z_i$ varieties, then we may define $\CI^\A(\alpha)$ directly using \eqref{defegdf}.
However, by our definition of regular classes this only holds after forgetting the equivariant structure. Hence we need to proceed with more caution. We take the following four steps:
\begin{enumerate}
\item[1.] Integrate regular elements non-equivariantly over the fibers of the map $p : \CM
_H({\overline{\chi}}) \to M_H(\overline{\chi})$, 
where $M_H(\overline{\chi})$ is the \emph{good moduli space} of $\CM_H({\overline{\chi}})$.
\item[2.] Show the constructible function obtained from (1.) is $\A$-equivariant. 
\item[3.] Integrate the constructible function of (1.) 
$\A$-equivariantly over $M_H(\overline{\chi})$ to get an element of $C^{\A}(X)$. 
\item[4.] Check the integration maps of (1.) and (3.) preserve the Poisson structures. 
It follows that $\CI^\A$ is a Poisson algebra homomorphism.
\end{enumerate}

\subsection{Equivariant integration map: Construction} \label{Subsection_Equiv_int_map}
\noindent 

\textbf{Step 1.} 
Let 
\begin{align*}
p \colon \mM_H(v) \to M_H(v)
\end{align*}
be the good moduli space of $\mM_H(v)$.
%i.e. 
%$M_H(v)$ is an algebraic space 
%satisfying that $p_{\ast}$ on coherent sheaves is exact
%and the induced morphism $\CO_{M_H(v)} \to p_{\ast} \CO_{\CM_H(v)}$ is an isomorphism.
which
parametrizes $S$-equivalence classes of 
$H$-Gieseker semistable sheaves with Chern character $v$. 
The existence of $M_H(v)$ as a projective scheme 
is well-known from 
the GIT construction of $\mM_H(v)$, 
see~\cite[Example~8.7]{MR3237451}.
We set
\begin{align}\label{map:goodmoduli}
p \colon \mM_{H}(\overline{\chi}) \to M_H(\overline{\chi})=\coprod_{\overline{\chi}_H(v)=\overline{\chi}}
M_H(v).
\end{align}
Until the end of this section, we fix $\overline{\chi}$
and only consider classes $v \in \Gamma$
satisfying $\overline{\chi}_H(v)=\overline{\chi}$.

Let $\Constr(M_H(\overline{\chi}))$ be the space of $\BQ$-valued constructible\footnote{A function $f : \CX \to \BQ$ is constructible,
if $f(\CX)$ is finite and for every $c \in f(\CX)$ the preimage $f^{-1}(c)$ is the union of a finite collection of finite type stacks.
In particular, $f : M \to \BQ$ constructible implies that $f|_{M_v}$ is non-zero only for finitely many $v \in \Gamma$,
where $M_v \subset M$ is the component of sheaves with Chern character $v$.} functions on $M_H(\overline{\chi})$.
Consider the map
\[ p_{\ast} : H_{\mathrm{reg}}(X, \overline{\chi}) \to \Constr(M_H(\overline{\chi})) \]
defined by integration over fibers as follows: If
\[ \alpha = \sum_i a_i [ Z_i \stackrel{f}{\to} \CM_H(v) ] \in H_{\mathrm{reg}}(X, \overline{\chi}) \]
for varieties $Z_i$ and $a_i \in \mathbb{Q}$,
then for every $x \in M_H(\overline{\chi})$ we let
\[ p_{\ast}(\alpha)(x) := \text{Coeff}_{c_v}\Big( \CI( \iota_{x \ast} \iota_x^{\ast} \alpha) \Big) = \sum_i a_i \int_{Z_i|_{\CM_x}} f^{\ast} \nu \, \mathrm{d}e, \]
where $\text{Coeff}_{c_v}(- )$ denotes the coefficient of $c_v$,
the map $\iota_x : \CM_x \to \CM_H({\overline{\chi}})$ is the inclusion of the fiber of 
the map (\ref{map:goodmoduli}) over $x \in M_H(\overline{\chi})$,
and we used the induced maps\footnote{Since $\iota_x$ is representable,
the composition $\iota_{x \ast} \iota_x^{\ast}$ preserves the subalgebra of regular classes.}
\[ \iota_x^{\ast} : H_{\mathrm{reg}}(X, \overline{\chi}) \to K_0(\St/\CM_x), \quad \iota_{x \ast} : 
K_0(\St/\CM_x) \to H(X, \overline{\chi}). \]

\vspace{7pt}
\noindent \textbf{Step 2.} 
The $\A$-action on the stack $\CM_H(\overline{\chi})$ 
descends to an $\A$-action on its good moduli space $M_H(\overline{\chi})$.\footnote{
As discussed in Section~\ref{subsec:pre lim stacks} we consider
the action here on the underlying reduced scheme.
}
Consider the subgroup of $\A$-invariant functions
\[ \Constr^\A(M_H(\overline{\chi})) \subset \Constr(M_H(\overline{\chi})). \]

\begin{lemma} The image of the composition 
\begin{align*}
H^\A_{\mathrm{reg}}(X, \overline{\chi}) \stackrel{\Forg}{\to} H_{\mathrm{reg}}(X, \overline{\chi}) \xrightarrow{p_{\ast}} \Constr(M_H(\overline{\chi}))
\end{align*}
lies in $\Constr^\A(M_H(\overline{\chi}))$. Hence we have the commatative diagram
\[
\begin{tikzcd}
H^\A_{\mathrm{reg}}(X, \overline{\chi}) \ar{d}{p^\A_\ast}\ar{r}{\Forg} &  H_{\mathrm{reg}}(X, \overline{\chi})
\ar{d}{p_{\ast}} \\
\Constr^\A(M_H(\overline{\chi})) \ar[hookrightarrow]{r} & \Constr(M_H(\overline{\chi})).
\end{tikzcd}
\]
with $p^\A_\ast = p_{\ast} \circ \Forg$.
\end{lemma}

\begin{proof}
Consider a regular equivariant class
\[ [ \CX \xrightarrow{f} \CM_H(v), a] \in H^{\A}_{\mathrm{reg}}(X, \overline{\chi}) \]
where $\CX$ is a stack, and let
\[ \phi = p_{\ast} \Forg( [ \CX \xrightarrow{f} \CM_{H}(v), a] ). \]
We need to show $\phi(a \cdot x) = \phi(x)$ for every $x \in M_H(\overline{\chi}) $ and $a \in A$.

Since the Behrend function is invariant under the $\A$ action,
by stratifying $\CX$ we may assume $f^{\ast} \nu$ is constant on $\CX$.
%and hence only contributes to a constant prefactor which we may ignore.
We let $\CX_x$ denote the fiber of $p \circ f : \CX \to M_H(\overline{\chi})$ over the
point $x \in M_H(\overline{\chi})$.
We need to compare the value of the integration map $\CI$ applied to
\[
[ \CX_x \to \CM_H(v) ],\,  [ \CX_{a \cdot x} \to \CM_H(v) ]
\in H_{\mathrm{reg}}(X, \overline{\chi}).
\]
Since $\CX$ carries an $\A$-action and $p \circ f$ is equivariant, translation by $a \in A$ yields an isomorphism of stacks
\[ t_a : \CX_x \xrightarrow{\cong} \CX_{a \cdot x}. \]
The claim now follows directly from the following Lemma.
\end{proof}

\begin{lemma} \label{lemma_independence} Let $[ \CY \xrightarrow{f} \CM_H(v) ] \in H_{\mathrm{reg}}(X, \overline{\chi})$ such that $f^{\ast} \nu$ is equal to a constant $k \in \BZ$.
Then the integral
\[ \CI( [ \CY \xrightarrow{f} \CM_H(v) ] ) \]
only depends on $k$, the class $v$ and the isomorphism class of the stack $\CY$.
\end{lemma}
\begin{proof}[Proof of Lemma \ref{lemma_independence}]
For a variety $Y$, let $P(Y)(u)$ be its virtual 
Poincar\'e polynomial. 
The stack $\yY$ admits a stratification whose strata 
is of the form $[Y_i/\GL_{n_i}(\mathbb{C})]$.
Then 
\begin{align*}
P(\CY)(u)
=\sum_i \frac{P(Y_i)(u)}{P(\GL_{n_i}(\mathbb{C}))(u)}
\in \mathbb{Q}(u)
\end{align*}
is independent of a stratification (see~\cite[Theorem~4.10]{Joy5}), and 
we have 
\[
\CI( [ \CY \xrightarrow{f} \CM_H(v) ] ) = k \lim_{u \to -1} P(\CY)(u) c_{v}
\]
where the limit on the right hand side exist since $\CY \to \CM_H(v)$ is regular.
The right hand side only depends on $\CY$ and $k$ and $v$,
and not on $f$.
\end{proof}

\vspace{7pt}
\noindent \textbf{Step 3.}
%Let $M_v \subset M$ be the component corresponding to sheaves with Chern character $v$.
Let $\phi : M_H(v) \to \BQ$ be a constructible $\A$-invariant function. Then there exists a stratification
\[ M_H(v) = \coprod_i Z_i \]
into $\A$-invariant subspaces $Z_i$ such that
\begin{itemize}
\item $Z_i$ is a variety,
\item the restriction $\phi|_{Z_i}$ is constant of value $a_i \in \BQ$,
\item there exists a connected subgroup $B_i \subset \A$ such that $(Z_i)_{B_i} = Z_i$.
\end{itemize}
Such a stratification can be constructed along the lines of \cite[2.4]{BrH} and \cite[3.3]{OS}.
We define an integration map
\[ J :  \Constr^\A(M_H(\overline{\chi})) \to C^\A(X) \]
by sending the constructible function $\phi$ to
\begin{align}\label{def:J} J(\phi)
= \left( \sum_i (-1)^{\dim(\A/B_i)} a_i e([Z_i/(A/B_i)] \epsilon_{B_i}) \right) c_{v}.
\end{align}
Since any two such stratifications have a common refinement, the map $J$ is well-defined.

\vspace{7pt}
\noindent \textbf{Step 4.} 
The direct sum map $\oplus \colon \CM_H(\overline{\chi}) \times \CM_H(\overline{\chi}) \to 
\CM_H(\overline{\chi})$
%By our assumption on the map $p \colon \mM \to M$, 
%taking the direct sum of sheaves 
descends to a map
\[ \oplus \colon M_H(\overline{\chi}) \times M_H(\overline{\chi}) \to M_H(\overline{\chi}). \]
Define an associative product and a Poisson bracket on $\Constr(M_H(\overline{\chi}))$ by
\begin{align*}
f \ast g & := \sum_{v_1, v_2} (-1)^{\chi(v_1, v_2)} \oplus_{\ast}( f_{v_1} \times g_{v_2} ), \\
\{ f, g \} & := \sum_{v_1, v_2} \chi(v_1, v_2) (-1)^{\chi(v_1, v_2)} \oplus_{\ast}( f_{v_1} \times g_{v_2} ),
\end{align*}
for all $f,g \in \Constr(M_H(\overline{\chi}))$, where $f_v = f|_{M_H(v)}$ and similar for $g$, and we let
\[ (f_{v_1} \times g_{v_2})(x_1, x_2) = f_{v_1}(x_1) g_{v_2}(x_2) \]
for all $x_1 \in M_{H}(v_1), x_2 \in M_{H}(v_2)$.
By a direct check $\mathrm{Constr}(M_H(\overline{\chi}))$
is a Poisson algebra with respect to $(\ast, \{-, -\})$.

Since taking direct sums is $\A$-equivariant, the operations $\ast$ and $\{ -,  - \}$ preserve the space of $\A$-invariant functions
and define a Poisson algebra structure on $\Constr^\A(M_H(\overline{\chi}))$. 

\begin{lemma} \label{Lemma_Poisson1} The map of integration along fibers
\[ p_{\ast} : H_{\mathrm{sc}}(X, \overline{\chi}) \to \Constr(M_H(\overline{\chi}))
\]
is a Poisson algebra homomorphism 
with respect to $(\ast, \{ - , - \})$.
The same holds for the equivariant map 
\[ p^\A_{\ast} : H^\A_{\mathrm{sc}}(X, \overline{\chi}) \to \Constr^\A(M_H(\overline{\chi})). \]
\end{lemma}

\begin{proof}
We first consider the non-equivariant case.
We need to show that for all $\alpha_1, \alpha_2 \in H_{\mathrm{sc}}(X, \overline{\chi})$ we have
\begin{align} \label{toprove} 
p_{\ast}( \alpha_1 \ast \alpha_2 ) & = p_{\ast}(\alpha_1) \ast p_{\ast}(\alpha_2), \\
p_{\ast}( \{ \alpha_1 ,  \alpha_2 \} ) & = \{ p_{\ast}(\alpha_1), p_{\ast}(\alpha_2) \},
\end{align}

Assume first that each $\alpha_i$ is supported over a point $x_i \in M_{H}(v_i)$, 
so in particular
\[ p_{\ast}(\alpha_i) = a_i \delta_{x_i}, \quad i=1,2, \]
where $a_i \in \BQ$ and we let $\delta_x$ is the characteristic function at the point $x$.
Then $\alpha_1 \ast \alpha_2$ is supported over the point $x = x_1 \oplus x_2$ and hence
\begin{align*}
p_{\ast}( \alpha_1 \ast \alpha_2 ) 
& = \text{Coeff}_{c_{v}} \left( \CI( \alpha_1 \ast \alpha_2 ) \right) \delta_x \\
& = \text{Coeff}_{c_{v}} \left( \CI( \alpha_1 ) \ast \CI( \alpha_2 ) \right) \delta_x \\
& = \left( a_1 a_2 (-1)^{\chi(v_1, v_2)} \right) \delta_x \\ 
& = p_{\ast}(\alpha_1) \ast p_{\ast}(\alpha_2),
\end{align*}
where we have set $v=v_1+v_2$. Similarly,
\begin{align*}
p_{\ast}( \{ \alpha_1, \alpha_2 \} ) 
& = \text{Coeff}_{c_v} \left( \CI( \{ \alpha_1, \alpha_2 \} ) \right) \delta_x \\
& = \text{Coeff}_{c_v} \left( \{ \CI(\alpha_1), \CI(\alpha_2) \} \right) \delta_x \\
& = \left( a_1 a_2 (-1)^{\chi(v_1, v_2)} \chi(v_1, v_2) \right) \delta_x \\
& = \{ p_{\ast}(\alpha_1) , p_{\ast}(\alpha_2) \}.
\end{align*}

For the general case let $\alpha_i = [X_i \to \mM_{H}(v_i)]$ where $X_i$ is a variety.
Let $x \in M_H(v_1+v_2)$ be a fixed point, and consider all possible decompositions
\[ x = x_{1j} \oplus x_{2j}, \quad j=1,\ldots, \ell \]
with $x_{ij} \in M_{H}(v_i)$ for $i=1,2$. Then, to compute the value of $p_{\ast}(\alpha_1 \ast \alpha_2)$ at $x$ we may replace $X_i$ by 
\[ \bigsqcup_{j=1}^{\ell} X_i |_{\CM_{x_{ij}}} \]
By bilinearity of both sides of \eqref{toprove} we may further assume that $\ell=1$,
or equivalently, that there is only a unique decomposition $x = x_{1} \oplus x_2$. 
But then we are in the case considered before and the claim follows.
%or equivalently, that there is only a single decomposition of $m = m_{1} + m_2$. But then
%the claim follows from the previously considered case. 
The argument for $\{ - , - \}$ is parallel.
This completes the non-equivariant case.

The equivariance case follows immediately:
we have $p_{\ast}^{\A} = p_{\ast} \circ \Forg$, and $\Forg$ and $p_{\ast}$ are both ring and Poisson algebra homomorphisms.
\end{proof}

\begin{lemma} \label{Lemma_Poisson2} The map
\[ J :  \Constr^\A(M_H(\overline{\chi})) \to C^\A(X) \]
is a Poisson algebra homomorphism. 
\end{lemma}
\begin{proof}
For every $i \in \{ 1, 2 \}$, let 
\[ X_i \subset M_{H}(v_i) \]
be an $\A$-invariant subspace such that $(X_i)_{B_i} = X_i$ for some connected subgroup $B_i \subset \A$.
We prove the claim for the $A$-invariant functions
\[ \delta_{X_i} \in \Constr^\A(M_H(\overline{\chi})). \]
The general case follows by a stratification argument.

By definition we have
\[
\delta_{X_1} \ast \delta_{X_2} = (-1)^{\chi(v_1, v_2)} \oplus_{\ast}( \delta_{X_1 \times X_2} )
\]
Applying $J$ yields
\begin{equation} \begin{aligned} \label{eq1}
J(\delta_{X_1} \ast \delta_{X_2})
& = (-1)^{\chi(v_1, v_2)+\dim(\A/B)} e\big( [X_1 \times X_2 / (\A/B)] \big) \epsilon_B c_{v_1+v_2} \\
& = (-1)^{\chi(v_1, v_2)+\dim(\A/B)} \e([X_1 \times X_2]) c_{v_1+v_2},
\end{aligned} \end{equation}
where $B = (B_1 \cap B_2)^{\circ}$ and $\e$ denotes the equivariant Euler characteristic.

On the other hand,
\begin{align*} J(\delta_{X_i}) & = (-1)^{\dim(\A/B_i)} e([X_i / (\A/B_i)] ) \epsilon_{B_i} c_{v_i} \\
& = (-1)^{\dim(\A/B_i)} \e(X_i ) c_{v_i}.
\end{align*}
By Section~\ref{Subsection_equivariant_Grothendieckgroup_of_var} we have
\[ \e([X_1 \times X_2]) = \e(X_1) \e(X_2). \]
Hence if $B_1$ and $B_2$ are not transverse, then \eqref{eq1} and 
$J(\delta_{X_1}) \ast J(\delta_{X_2})$ both vanish.
If $B_1$ and $B_2$ are transverse, then
\[ \dim(\A/B) = \codim(B) = \codim(B_1) + \codim(B_2) = \dim(\A/B_1) + \dim(\A/B_2) \]
which gives the desired equality:
\[
J(\delta_{X_1} \ast \delta_{X_2}) 
%= (-1)^{\chi(v_1, v_2)} (-1)^{\dim(\A/B_1)} (-1)^{\dim(\A/B_2)}\e(X_1) \e(X_2) c_{v_1+v_2}
= J(\delta_{X_1}) \ast J(\delta_{X_2}).
\]
The check that $J$ preserves the Poisson bracket is parallel.
\end{proof}

\begin{defn} \label{defn_equiv} The equivariant integration map is defined by
\[ \CI^{\A} := J \circ p^{\A}_{\ast} : H^\A_{\mathrm{sc}}(X, \overline{\chi}) \to \Constr^\CA(M_H(\overline{\chi})) \to C^{\A}(X). \]
\end{defn}

We have the following result.
\begin{thm}\label{thm:IA} The equivariant integration map $\CI^\A$
is a Poisson algebra homomorphism. Moreover, for
every $\A$-equivariant map 
$f : Z \to \CM_H(v)$, where $Z$ is a variety, we have
\begin{equation*} \CI^\A([Z \xrightarrow{f} \CM_H(v), a]) =
\left(\sum_{B \subset \A}
(-1)^{\dim (\A/B)}  \epsilon_B
\int_{[Z_B/(\A/B)]}
f^{\ast} \nu \ \mathrm{d}e \right) c_v,
%= \e( [Z,a], f^{\ast} \nu ) c_{v}, 
\end{equation*}
\end{thm}
\begin{proof}
The first claim follows from Lemma~\ref{Lemma_Poisson1} and~\ref{Lemma_Poisson2}.
For the second we may assume $Z$ is a $\A$-invariant subvariety of $\CM_H(v)$,
that $Z_B = Z$ for some connected abelian subvariety $B \subset \A$
and that the Behrend function is constant of value $k$ on $Z$.
Let $p_Z : Z \to Z' \subset M_H(v)$ 
be the restriction of $p : \CM_H(v) \to M_H(v)$ to $Z$. Then
\begin{align*} 
\CI^{\A}([Z \to \CM_H(v), a]) 
& = k (-1)^{\dim(\A/B)} \epsilon_B c_{v} \int_{[Z'/(\A/B)]} p_{Z\ast}(1) \ \mathrm{d}e \\
& = k (-1)^{\dim(\A/B)} \epsilon_B c_{v} \int_{[Z/(\A/B)]} 1 \ \mathrm{d}e. \qedhere
\end{align*}
\end{proof}

\subsection{Definition of Donaldson-Thomas invariants} \label{Subsection_definition_of_reduced_DT}
As above let $A$ be an abelian variety acting on a Calabi--Yau threefold $X$, and set $\A = A \times \Pic^0(X)$.
The stack of semistable sheaves (\ref{moduli:semistable}) defines 
an element 
\begin{align*}
\delta_{H}(v) \cneq [\mM_H(v) \subset \CM_H(\overline{\chi})]\in 
H^{\A}(X, \overline{\chi}). 
\end{align*}
Applying a formal logarithm defines the element
\begin{equation} \label{epsilon_guy}
\epsilon_{H}(v) \cneq 
\sum_{\begin{subarray}{c}
l\ge 1, v_1+\cdots+v_l=v \\
\overline{\chi}_H(v_i)=\overline{\chi}
\end{subarray}}
\frac{(-1)^{l-1}}{l}
\delta_H(v_1) \ast \cdots \ast \delta_H(v_l). 
\end{equation}
The following is the equivariant analog of a result of Joyce.

\begin{prop}
%\begin{align*}
$(\mathbb{L}-1)\epsilon_H(v) \in H_{\rm{reg}}^{\A}(X, \overline{\chi})$.
%\end{align*}
\end{prop}
\begin{proof} By Joyce \cite[Theorem~8.7]{Joy3} the element is regular after forgetting the equivariant structure. Hence it is regular by definition \eqref{defn_reg}.
\end{proof}

Define the class
\begin{align*}
\overline{\epsilon}_H(v) \cneq 
[(\mathbb{L}-1)\epsilon_H(v)] \in H_{\rm{sc}}^{\A}(X, \overline{\chi}).
\end{align*}

\begin{defi} The $\A$-reduced Donaldson--Thomas invariant of $X$ in class $v \in \Gamma_{+}$
is the unique element $\DTb_H(v) \in \mathbb{Q}[\A]$ such that
\[ \CI^{\A}(\overline{\epsilon}_H(v)) = \DTb_H(v) \cdot c_v. \]
\end{defi}

\begin{rmk}
We expect $\DTb_H(v) \in \BQ[\A]$ to be invariant under deformations of $X$ under which $v \in H^{\ast}(X,\BQ)$ remains algebraic.
If $v$ is primitive the deformation invariance property can be proven by constructing a slice of the $\A$-action,
see \cite{Gul} for abelian threefolds and \cite{O1_red} for $\mathrm{K3} \times E$.
\end{rmk}

It is convenient to define Donaldson-Thomas invariants for every $v \in \Gamma$ by the following convention:
\begin{itemize}
% \item If $v \in \Gamma_{+}$ let
% \[ \CI^{\A}(\overline{\epsilon}_H(v)) = \DTb_H(v) \cdot c_v. \]
\item If $v \in -\Gamma_{+}$ define
$\DTb_H(v) := \DTb_H(-v)$.
\item If $v \notin \pm \Gamma_{+}$ define $\DTb_H(v) := 0$.
\end{itemize}

For any $v \in \Gamma$ and connected abelian subvariety $B \subset \A$, 
we further define $\DT_{H}(v)_B \in \mathbb{Q}$ by the expansion
\begin{align*}
\DTb_H(v)=\sum_{B \subset \A} \DT_{H}(v)_B \cdot \epsilon_B. 
\end{align*}
Moreover we write $\DT_H(v) \cneq \DT_{H}(v)_{B=(0, 0)}$. 

\begin{rmk} Let $v \in \Gamma$. We expect that the stabilizer of an element $E \in \CM_H(v)$
only depends on its Chern character and not on its moduli.
In particular, for every $v \in \Gamma$ we expect to have $\DTb_H(v) = \DT_{H}(v)_B \epsilon_B$ for a $B$ determined by $v$.
Partial results in this direction were obtained by Gulbrandsen, see \cite[Proposition~3.5]{Gul}.
\end{rmk}

\section{Bridgeland stability conditions on abelian threefolds} \label{section:Bridgeland_stab_conditions}
\subsection{Review of stability conditions} \label{Subsection_Review_Stability_conds}
Let $X$ be a smooth projective variety,
and $D^b(X)$ its bounded derived category of coherent sheaves. 
Here we review Bridgeland stability conditions on $D^b(X)$. 
We fix a finitely generated free abelian group $\Lambda$,
and a group homomorphism 
$\cl \colon K(X) \to \Lambda$. 

\begin{defi}\emph{(\cite{Brs1})} \label{Definition_BSC}
A stability condition on $D^b(X)$ with respect to $(\Lambda, \cl)$
is a pair 
\begin{align}\notag
\sigma=(Z, \aA), \quad \aA \subset D^b(X)
\end{align}
where $Z \colon \Lambda \to \mathbb{C}$
is a group homomorphism and $\aA$ is the heart of a 
bounded t-structure
such that the following conditions hold: 
\begin{enumerate}
\renewcommand{\labelenumi}{(\roman{enumi})}
\item For any non-zero $E\in \aA$, we have 
\begin{align}\notag
Z(E) := Z(\cl(E))
\in \{ r e^{\pi i \phi} : r>0, \phi \in (0, 1] \}.
\end{align}
\item \emph{(Harder-Narasimhan property)}
For any $E\in \aA$, there is a filtration
\begin{align*}
0=E_0 \subset E_1 \subset \cdots \subset E_N
\end{align*}
in $\aA$ such that each subquotient $F_i=E_i/E_{i-1}$ is 
$Z$-semistable with 
$\arg Z(F_i)> \arg Z(F_{i+1})$. 
\item \emph{(Support property)}
There is a quadractic form $Q$ on 
$\Lambda$
such that $Q(\cl(E)) \ge 0$
for any $Z$-semistable object $E$
and $Q$ is negative definite on $\Ker(Z)$.
\end{enumerate}
\end{defi}
\vspace{4pt}

Here an object $E\in \aA$ is $Z$-\textit{(semi)stable} if we have
\begin{align*}
\arg Z(F) <(\le) \arg Z(E)
\end{align*}
in $(0, \pi]$ for any subobject $0\neq F \subsetneq E$.

For group 
homomorphisms $Z, Z' \colon \Lambda \to \mathbb{C}$, 
we write $Z \sim Z'$
if we have 
\begin{align*}
\Ree Z'=\lambda_1 \Ree Z+ \lambda_2 \Imm Z, \ 
\Imm Z'=\lambda_3 \Imm Z
\end{align*}
for some $\lambda_i \in \mathbb{R}$
with $\lambda_1, \lambda_3$ positive. 
Then $(Z, \aA)$ is a stability condition if and only 
if $(Z', \aA)$ is a stability condition, 
and $Z$-(semi)stable objects coincide with 
$Z'$-(semi)stable objects. 
In this case, we 
say that $(Z, \aA)$, $(Z', \aA)$
are equivalent and write
$(Z, \aA) \sim (Z', \aA)$.

Given a Bridgeland stability condition $\sigma = (Z, \aA)$
the category of $\sigma$-semistable objects
with phase $\phi \in \BR$ is defined
in case $\phi \in (0,1]$ by
\begin{align*}
\pP(\phi) \cneq 
\{ E \in \aA \, : \, E \mbox{ is } Z\mbox{-semistable with }
Z(E) \in \mathbb{R}_{>0} e^{\pi i \phi}\} \cup \{0\}. 
\end{align*}
and for general $\phi \in \BR$ by the condition
\[ \pP(\phi+1)=\pP(\phi)[1]. \]
The data of a stability condition $\sigma$
is equivalent to the data
\begin{align}\label{data}
(Z, \{\pP(\phi)\}_{\phi \in \mathbb{R}}), \ Z \colon \Lambda \to \mathbb{C}, \ \pP(\phi) \subset D^b(X)
\end{align}
satisfying some properties, see \cite[Section~5]{Brs1} for details.

Let $\Stab_{\Lambda}(X)$ be the set of stability conditions on $D^b(X)$ with respect to $(\Lambda, \cl)$. 
By~\cite{Brs1} there is a natural topology 
on $\Stab_{\Lambda}(X)$ such that the forgetful map
\begin{align*}
\Pi \colon 
\Stab_{\Lambda}(X) \to \Lambda_{\mathbb{C}}^{\vee},\ (Z, \aA) \mapsto Z
\end{align*}
is a local homeomorphism.

Let $\Gamma \subset H^{2\ast}(X, \mathbb{Q})$ be the image of the Chern character map.
We call the support property with respect to $(\Gamma, \ch)$ the \emph{full support property}.\footnote{
This will be used in the following way.
Suppose that $\sigma = (Z,\CA)$ is a stability condition with respect to $(\Lambda, \cl)$ and that $\cl : K(X) \to \Lambda$ factors through the Chern character map,
i.e. $\cl = \cl' \circ \ch$ for some $\cl' : \Gamma \to \Lambda$.
Then the pair $\sigma' = (Z \circ \cl', \CA)$ automatically satisfies conditions (i,ii) of Definition~\ref{Definition_BSC}, but not
necessarily the full support property (iii). Hence the stability condition $\sigma$ induces a stability condition with respect to $(\Gamma, \ch)$
if and only if $\sigma$ (or more precisely $\sigma'$) satisfies the full support property.}
The space of stability conditions with respect to $(\Gamma, \ch)$ is denoted by
\begin{align*}
\Stab(X) := \Stab_{\Gamma}(X). 
\end{align*}

Every (co-variant) autoequivalence $g \in \Aut(D^b(X))$
induces an action $g_{\ast} \in \Aut(\Gamma)$
which commutes with Chern character maps, see also Section \ref{Cohomological_FMT} for further details. 
Therefore $g$ also acts on $\Stab(X)$ by
\begin{equation} g_{\ast} (Z, \aA) := \big( g_{\ast}Z , g(\aA) \big), \label{Action_on_stab} \end{equation}
where $g_{\ast}Z(-) :=Z \circ g_{\ast}^{-1}(-)$.
The induced action on the manifold $\Stab(X)$ is a homeomorphism,
and the assignment $g \mapsto g_{\ast}$ defines a 
left $\Aut(D^b(X))$-action on $\Stab(X)$.

\subsection{Double tilting constructions}
Let $X$ be a smooth projective 3-fold, 
and let $B+i\omega \in \mathrm{NS}(X)_{\mathbb{C}}$
with $\omega$ ample. 
The $B$-twisted Chern character
of an object $E \in D^b(X)$ is
defined by
\begin{align*}
\ch^B(E) \cneq e^{-B} \ch(E) \in H^{2 \ast}(X, \mathbb{R}). 
\end{align*}
For any $E \in K(X)$ let
\begin{equation} \notag
\begin{aligned}
Z_{\omega, B}(E) &\cneq -\int_{X}e^{-i\omega} \ch^B(E) \\
&=\left(-\ch_3^B(E)+\frac{1}{2}\ch_1^B(E) \omega^2\right)+
i\left(\ch_2^B(E) \omega-\frac{1}{6}\ch_0^B(E) \omega^3  \right). 
\end{aligned}
\end{equation}
If $X$ is an abelian 3-fold, we have 
\begin{align}\label{Z_omegaB}
Z_{\omega, B}(E)=
-\chi(e^{B+i\omega}, \ch(E)).
\end{align}
The homomorphism
$Z_{\omega, B} \colon K(X) \to \mathbb{C}$
descends to a homomorphism
\[ Z_{\omega, B} \colon \Gamma \to \mathbb{C}. \]

In~\cite{BMT} a heart of a t-structure
$\aA_{\omega, B}$
was constructed as a candidate
for a Bridgeland stability condition
$(Z_{\omega, B}, \aA_{\omega, B})$. 
We review the construction. 
Consider the $B$-twisted $\omega$-slope function on $\Coh(X)$,
\begin{align*}
\mu_{\omega, B}(E)=
\frac{\mathrm{ch}_1^B(E) \cdot \omega^{2}}{\mathrm{rank}(E)}
\in \mathbb{R} \cup \{\infty\}.
\end{align*}
It defines the usual slope stability on $\Coh(X)$.
Define a torsion pair
$(\mathcal{T}_{\omega, B}, \mathcal{F}_{\omega, B})$
on $\Coh(X)$ by
\begin{align*}
&\mathcal{T}_{\omega, B}
=\langle E \in \mathrm{Coh}(X): 
E \mbox{ is } \mu_{\omega, B} \mbox{-semistable with }
\mu_{\omega, B}(E)> 0 \rangle, \\
&\mathcal{F}_{\omega, B}
=\langle E \in \mathrm{Coh}(X): 
E \mbox{ is } \mu_{\omega, B} \mbox{-semistable with }
\mu_{\omega, B}(E)\le 0 \rangle,
\end{align*}
where we let $\langle \ast \rangle$ denote the extension closure. 
Its tilt is the heart
\begin{align*}
\mathcal{B}_{\omega, B} =
\langle \mathcal{F}_{\omega, B}[1], \mathcal{T}_{\omega, B}
\rangle \subset D^b(X). 
\end{align*}
The slope function $\nu_{\omega, B}$ 
on $\bB_{\omega, B}$ is defined by
\begin{align*}
\nu_{\omega, B}(E)=\frac{\mathrm{Im}Z_{\omega, B}(E)}{\mathrm{ch}_1^B(E) \cdot 
\omega^2} \in \mathbb{R} \cup \{\infty\}. 
\end{align*}
It also defines the $\nu_{\omega, B}$-stability on 
$\bB_{\omega, B}$. 
Similarly to above, the torsion pair 
$(\mathcal{T}_{\omega, B}', \mathcal{F}_{\omega, B}')$
of $\mathcal{B}_{\omega, B}$ is defined by
\begin{align*}
&\mathcal{T}_{\omega, B}'
=\langle E \in \mathcal{B}_{\omega, B}: 
E \mbox{ is } \nu_{\omega, B}\mbox{-semistable with }
\nu_{\omega, B}(E)>0 \rangle, \\
&\mathcal{F}_{\omega, B}'
=\langle E \in \mathcal{B}_{\omega, B}: 
E \mbox{ is } \nu_{\omega, B}\mbox{-semistable with }
\nu_{\omega, B}(E)\le 0 \rangle.
\end{align*}
By tilting a second time we obtain the heart
\begin{align*}
\mathcal{A}_{\omega, B} = \langle \mathcal{F}_{\omega, B}'[1], 
\mathcal{T}_{\omega, B}' \rangle \subset D^b(X). 
\end{align*}
In~\cite{BMT} it was conjectured that the pairs
\[ \sigma_{\omega, B} := (Z_{\omega, B}, \aA_{\omega, B}) \]
are Bridgeland stability conditions. 

\subsection{Bogomolov-Gieseker inequalities}
In order to show that pairs $\sigma_{\omega, B}$
are stability conditions, and 
in particular satisfy the support property,
we need to investigate quadractic inequalities for semistable objects. 
First we recall quadractic inequalities for $\nu_{\omega, B}$-semistable
objects in $\bB_{\omega, B}$. 

Let $H$ be a fixed ample divisor on $X$
and consider the case $\omega=\alpha H$
for some $\alpha \in \mathbb{R}_{>0}$.
By~\cite{BMT}, there is a constant 
$C_H \ge 0$
such that for every effective divisor 
$D$ on $X$, we have
\begin{align*}
C_H(H^2 D)^2+(H^3)(HD^2) \ge 0. 
\end{align*}
If $X$ is an abelian 3-fold, we can take $C_H=0$.
Let us also take $B \in \mathrm{NS}(X)_{\mathbb{R}}$
and for any $E \in D^b(X)$ define
\begin{align*}
&\Delta(E) :=(\ch_1(E))^2-2\ch_0(E) \ch_2(E), \\
&\overline{\Delta}_{H, B}(E) :=
(H^2 \ch_1^B(E))^2 -2(H^3 \ch_0^B(E))(H\ch_2^B(E)). 
\end{align*}
By the Hodge index theorem
we have 
$\overline{\Delta}_{H, B}(E) \ge H^3 \cdot H\Delta(E)$
which is an equality when 
the Picard rank of $X$ is one. 

\begin{prop}\emph{(\cite{BMT})}\label{prop:Delta}
For any $\nu_{\omega, B}$-semistable 
object $E \in \bB_{\omega, B}$, 
where $\omega=\alpha H$
for an ample divisor $H$ and $\alpha \in \mathbb{R}_{>0}$,
we have 
the inequlaities
\begin{align*}
\overline{\Delta}_{H, B}(E) \ge 0, \qquad \text{and} \qquad 
H^3 \cdot H\Delta(E) +C_H(H^2 \ch_1^B(E))^2 \ge 0. 
\end{align*}
\end{prop}

For any $E \in D^b(X)$ define
\begin{align*}
\overline{\nabla}_{H, B}(E)=
12(H^2 \ch_1^B(E))^2-18 (H^3 \ch_0^B(E))(H\ch_2^B(E)).
\end{align*}
The following conjecture is proposed in~\cite{BMT, BMS}:
\begin{conj}\emph{(\cite{BMT, BMS}, \cite[Theorem~1.4]{PiYT})}\label{conj:BG}
For any $\nu_{\omega, B}$-semistable object $E \in \bB_{\omega, B}$,
where $\omega=\alpha H$ for an ample divisor $H$ and $\alpha \in \mathbb{R}_{>0}$,
we have
\begin{align*}
\alpha^2\overline{\Delta}_{H, B}(E)+\overline{\nabla}_{H, B}(E)
\ge 0. 
\end{align*}
\end{conj}

For fixed $(H, B)$, let $\Lambda_{H, B} \subset \mathbb{R}^4$
be the free abelian group of rank $4$ 
given by the image of the map
\begin{align*}
\cl \colon K(X) \to \mathbb{R}^4, \ 
E \mapsto (H^3 \ch_0^B(E), H^2 \ch_1^B(E), 
H \ch_2^B(E), \ch_3^B(E)).
\end{align*}
The following result  is proven in~\cite[Theorem~8.6]{BMS}
when $B$ is proportional to $H$, and 
the general case follows by a parallel argument.

\begin{prop}\label{prop:stabcond}\emph{(\cite[Theorem~8.6]{BMS})}
If Conjecture~\ref{conj:BG} holds for $X$
and some $\alpha \in \mathbb{R}_{>0}$, 
then we have 
\begin{align}\label{stab:ZA}
(Z_{\alpha H, B}^{a, b}, \aA_{\alpha H, B})
\in \Stab_{\Lambda_{H, B}}(X),
\end{align}
where $Z_{\alpha H, B}^{a, b}$ is defined by
\begin{align}\label{Zab}
Z_{\alpha H, B}^{a, b}=
\left( -\ch_3^B+bH \ch_2^B+aH^2 \ch_1^B \right)+
i \left(\alpha H \ch_2^B 
-\frac{\alpha^3}{6}H^3 \ch_0^B  \right) 
\end{align}
with $a, b \in \mathbb{R}$
satisfying 
\begin{align}\label{ineq:ab}
a>\frac{\alpha^2}{18}+\frac{\sqrt{3}}{6}\lvert b \rvert \alpha.
\end{align}
Moreover, there is an interval 
$I_{\alpha}^{a, b} \subset (\alpha^2, 18a)$
such that for all $K \in I_{\alpha}^{a, b}$, 
the quadratic form 
defined by
\begin{align*}
Q_K(-)=K \overline{\Delta}_{H, B}(-)
+\overline{\nabla}_{H, B}(-)
\end{align*}
establishes the support property for the 
stability condition (\ref{stab:ZA}). 
\end{prop}

Conjecture~\ref{conj:BG} is known to hold for abelian threefolds $A$ by~\cite{MaPi1, MaPi2, BMS, PiYT}.
Hence by Proposition~\ref{prop:stabcond} the pairs
\[ \sigma_{\omega, B}=(Z_{\alpha H, B}^{a=\alpha^2/2, b=0}, \aA_{\alpha H, B}), \ \omega = \alpha H \]
define Bridgeland stability conditions on $A$ with respect to $(\Lambda_{H, B}, \cl)$
and define points in $\Stab_{\Lambda_{H,B}}(A)$.
In the following subsections we show that the pairs (\ref{stab:ZA})
are stability conditions also with respect to $(\Gamma, \ch)$.
In particular, they form a family in $\Stab(A)$.

\subsection{Projection maps in cohomologies}
Let $X$ be an $n$-dimensional smooth projective variety, 
and $H \in \NS_{\mathbb{Q}}(X)$ an ample class. 
Let 
\begin{align*}
H^{2\ast}_{\mathrm{alg}}(X, \mathbb{Q})
\subset H^{2\ast}(X, \mathbb{Q})
\end{align*}
be the subspace spanned by algebraic classes. 
We fix some notation on the projection maps on 
$H^{2\ast}_{\mathrm{alg}}(X, \mathbb{Q})$.
For any $i$, we define 
$$
p_{H, i} : H^{2i}_{\mathrm{alg}}(X,\mathbb{Q}) \to H^{2i}_{\mathrm{alg}}(X,\mathbb{Q}), \ \ 
\gamma_i \mapsto \frac{\gamma_i \cdot H^{n-i}}{H^{n}} H^i.
$$
This gives us the map 
$$
p_{H} : H^{2*}_{\mathrm{alg}}(X,\mathbb{Q}) \to H^{2*}_{\mathrm{alg}}(X,\mathbb{Q}), \ \ 
(\gamma_0, \ldots, \gamma_n) \mapsto \left( p_{H, 0}(\gamma_0), \ldots , p_{H, n}(\gamma_n)\right).
$$ 
Also we define 
$$
p_{H, i}^{\perp} : H^{2i}_{\mathrm{alg}}(X,\mathbb{Q}) \to H^{2i}_{\mathrm{alg}}(X,\mathbb{Q}), \ \ 
\gamma_i \mapsto  \gamma_i - p_{H, i}(\gamma_i),
$$
and 
$$
p_{H}^{\perp} : H^{2*}_{\mathrm{alg}}(X,\mathbb{Q}) \to H^{2*}_{\mathrm{alg}}(X,\mathbb{Q}), \ \ 
(\gamma_0, \ldots, \gamma_n) \mapsto \left( p_{H, 0}^{\perp}(\gamma_0), \ldots , p_{H, n}^{\perp}(\gamma_n)\right).
$$ 
We define
\begin{align*}
& H^{2i}_{\mathrm{alg}}(X, \mathbb{Q})^{H, \parallel} = \imm \left( p_{H, i} \right), \quad
H^{2i}_{\mathrm{alg}}(X, \mathbb{Q})^{H, \perp} = \imm \left( p_{H, i}^{\perp}  \right),
\\
& H^{2*}_{\mathrm{alg}}(X, \mathbb{Q})^{H, \parallel} = \imm \left( p_{H} \right), \quad
H^{2*}_{\mathrm{alg}}(X, \mathbb{Q})^{H, \perp} = \imm \left( p_{H}^{\perp} \right). 
\end{align*}
So we have 
\begin{align*}
H^{2i}_{\mathrm{alg}}(X, \mathbb{Q}) = H^{2i}_{\mathrm{alg}}(X, \mathbb{Q})^{H, \parallel} \oplus H^{2i}_{\mathrm{alg}}(X, \mathbb{Q})^{H, \perp}, \\
H^{2*}_{\mathrm{alg}}(X, \mathbb{Q}) = H^{2*}_{\mathrm{alg}}(X, \mathbb{Q})^{H, \parallel} \oplus H^{2*}_{\mathrm{alg}}(X, \mathbb{Q})^{H, \perp}.
\end{align*}
By abuse of notation we will write $p_{H} $ for $p_{H, i} $, and $p_{H}^{\perp} $ for $p_{H, i}^{\perp}$.
We have 
\begin{equation}
\label{eqn:map-decom}
\id = p_H + p_H^{\perp}.
\end{equation}

We write
\begin{align*}
\ch_i^{H, \parallel}(E)  = p_H\left( \ch_i(E) \right), \
\ch_i^{H, \perp}(E) = p_H^{\perp} \left( \ch_i(E) \right).
\end{align*}
Then we have  $H^{n-i} \cdot \ch_i^{H, \parallel}(E) = H^{n-i} \cdot \ch_i(E) $, and $H^{n-i} 
\cdot \ch_i^{H, \perp}(E) =0$.
From the Hodge Index Theorem, we have
\begin{equation}
H^{n-2} \cdot \left( \ch_1^{H, \perp}(E)\right)^2 \le 0.
\end{equation}

\begin{rmk}
Let $\Lambda_H^{\parallel}$
be the image of 
the composition 
\begin{align*}
K(X) \stackrel{\ch}{\to} H^{2\ast}_{\mathrm{alg}}(X, \mathbb{Q})
\stackrel{p_H^{\parallel}}{\to}
H_{\mathrm{alg}}^{2\ast}(X, \mathbb{Q})^{H, \parallel}.
\end{align*}
If $B$ is proportional to $H$, then 
the support properties 
for $(\Lambda_{H, B}, \cl)$
and $(\Lambda_H^{\parallel}, p_H^{\parallel}\circ \ch)$
are equivalent. 
So in Proposition~\ref{prop:stabcond}, we obtain 
stability conditions in 
$\Stab_{\Lambda_H^{\parallel}}(X)$. 
\end{rmk}
We define 
$\Lambda_H^{\sharp}$, 
$\Lambda_H^{\flat}$
to be the images of maps
\begin{align}\notag
&\cl^{\sharp} \colon 
K(X) \to H^{2\ast}_{\mathrm{alg}}(X, \mathbb{Q}), \ 
E \mapsto (\ch_0(E), \ch_1(E), \ch_2^{H, \parallel}(E), \ch_3(E)), \\
\label{cl:flat} & \cl^{\flat} \colon 
K(X) \to H^{2\ast}_{\mathrm{alg}}(X, \mathbb{Q}), \ 
E \mapsto (\ch_0(E), \ch_1^{H, \parallel}(E), \ch_2(E), \ch_3(E))
\end{align}
respectively. 
In the next lemma, we 
observe that 
stability conditions in Proposition~\ref{prop:stabcond} 
satisfy the support property with respect to the 
$(\Lambda_H^{\sharp}, \cl^{\sharp})$.
\begin{lem}\label{lem:sharp}
In the situation of Proposition~\ref{prop:stabcond}, 
suppose that $B$ is proportional to $H$ and 
$C_H=0$.
Then 
the stability conditions (\ref{stab:ZA}) satisfy the 
support property with respect to $(\Lambda_H^{\sharp}, \cl^{\sharp})$.
For an interval $I_{\alpha}^{a, b} \subset (\alpha^2, 18a)$
and $K \in I_{\alpha}^{a, b}$, 
the quadractic form is 
\begin{align*}
Q_K^{\sharp}=
K \overline{\Delta}_{H, B}(-)+\overline{\nabla}_{H, B}(-)
+(K-\alpha^2)H^3 \cdot H(\ch_1^{H, \perp}(-))^2. 
\end{align*}
\end{lem}
\begin{proof}
The proof of~\cite[Lemma~8.8]{BMS}
is applied, by replacing the inequality 
$\overline{\Delta}_{H, B}(-)\ge 0$
for $\nu_{\omega, B}$-semistable objects 
with the inequality
(see Proposition~\ref{prop:Delta})
\begin{align*}
H^3 \cdot H\Delta(-)=\overline{\Delta}_{H, B}(-)
+H^3 \cdot H(\ch_1^{H, \perp}(-))^2 \ge 0. 
\end{align*}
Then the quadractic form 
$\alpha^2 \overline{\Delta}_{H, B}+\overline{\nabla}_{H, B}+
(K-\alpha^2)H^3 \cdot H\Delta$ gives the 
desired support property. 
\end{proof}

\subsection{Fourier-Mukai transforms and abelian 3-folds}
\label{Cohomological_FMT}
Our strategy for the proof of the full support property 
is to use Fourier-Mukai transforms. 
Let us quickly recall some of the important notions in Fourier-Mukai theory.
Further details can be found in \cite{Huybook}.

Let $X,Y$ be smooth projective varieties and let $p_i$, $i=1,2$ be the  projection
maps from $X \times Y$ to $X$ and $Y$, respectively.
The {\it Fourier-Mukai functor} $\Phi_{\eE}^{X \to Y}:
D^b(X) \to D^b(Y)$ with kernel $\eE \in D^b(X \times Y)$
is defined by
$$
\Phi_{\eE}^{X \to Y}(-) = 
\dR p_{2*} (\eE \stackrel{\textbf{L}}{\otimes} p_1^*(-)).
$$
When $\Phi_{\eE}^{X \to Y}$ is an equivalence of the derived categories,
usually it is called a {\it Fourier-Mukai transform}.
Any Fourier-Mukai  functor $\Phi_{\eE}^{X \to Y}: D^b(X) \to D^b(Y)$ induces a linear map
\begin{align*}
\Phi^{\mathrm{H}}_{\eE}  \colon  H^{2*}_{\mathrm{alg}}(X, \mathbb{Q})  \to
H^{2*}_{\mathrm{alg}}(Y, \mathbb{Q}).
\end{align*}
Here $H^{2\ast}_{\mathrm{alg}}(X, \mathbb{Q}) \subset H^{2\ast}(X, \mathbb{Q})$
is the subspace sppaned by algebraic classes. 
The above linear map 
is a linear isomorphism when $\Phi_{\eE}^{X \to Y}$ is a  Fourier-Mukai  transform.
The induced transform fits into the following commutative diagram, due to the Grothendieck-Riemann-Roch theorem.
\[
\begin{tikzcd}
D^b(X) \ar{r}{\Phi_{\eE}^{X \to Y}} \ar{d}{v_X(-)} &  D^b(Y) 
\ar{d}{v_Y(-)} \\
H^{2\ast}_{\mathrm{alg}}(X, \mathbb{Q}) \ar{r}{\Phi_{\eE}^{\mathrm{H}}} &
H^{2\ast}_{\mathrm{alg}}(Y, \mathbb{Q}).
\end{tikzcd}
\]
Here $v_X(-) = \ch(-) \sqrt{\text{td}_X}$ is the Mukai vector map. 
Note that for an abelian variety $X$, $\text{td}_X =1$.
Hence the Mukai vector $v(E)$ of $E \in D^b(X)$ is the same as its Chern character $\ch(E)$.

Let $X=A$ be an abelian variety, and $\widehat{A} =\Pic^0(A)$
its dual abelian variety. 
The \textit{Poincar\'e line bundle} $\pP$ on the product $A \times \widehat{A}$ is the
uniquely determined line bundle satisfying
(i) $\pP_{A \times \{\widehat{x}\}} \in \Pic(A)$ is represented by $\widehat{x} \in \widehat{A}$, and
(ii) $\pP_{\{e\} \times \widehat{A} } \cong \oO_{\widehat{A}}$.
In \cite{Mu1}, Mukai proved that the Fourier-Mukai functors 
\begin{align*}
\Phi^{A \to \widehat{A}}_{\pP}: D^b(A) \to D^b(\widehat{A}), \ 
\Phi_{\pP^\vee}^{\widehat{A} \to A} : D^b(\widehat{A}) \to D^b(A)  
\end{align*}
are equivalences of derived categories, 
i.e. Fourier-Mukai transforms. 
Moreover, he proved that   
$$
\Phi_{\pP^\vee}^{\widehat{A} \to A}  \circ   \Phi^{A \to \widehat{A}}_{\pP} \cong \id [-n], \  \
\Phi^{A \to \widehat{A}}_{\pP} \circ \Phi_{\pP^\vee}^{\widehat{A} \to A} \cong \id[-n],
$$
where $n$ is the dimension of $A$ and $\widehat{A}$.

Let $L$ be an ample line bundle on $A$. Its image under $\Phi_{\pP}^{A \to \widehat{A}}$ is
a semihomogeneous vector bundle\footnote{See Section \ref{semihomogeneous} for more details on semihomogeneous bundles.}
$\widehat{L}$
of rank $\chi(L) = c_1(L)^n/n!$,
$$
\Phi_{\pP}^{A \to \widehat{A}} (L) \cong \widehat{L}.
$$
Moreover, $-c_1(\widehat{L})$ is an ample divisor class on $\widehat{A}$. See~\cite{BL-polarization}  for further details. We have the following:

\begin{lem}[{\cite{BL-polarization}}]
\label{classicalcohomoFMT}
Let $H \in \NS_{\mathbb{Q}}(A)$ be an ample class on $A$.
Under the induced cohomological transform $\Phi_{\pP}^{\mathrm{H}}: H^{2*}_{\mathrm{alg}}(A, \mathbb{Q}) \to H^{2*}_{\mathrm{alg}}(\widehat{A},\mathbb{Q})$ of  $\Phi_{\pP}^{A \to \widehat{A}}$ we have
\begin{equation*}
\Phi_{\pP}^{\mathrm{H}}(e^{H}) = ({H^n}/{n!}) \, e^{-\widehat{H}}
\end{equation*}
for some ample class $\widehat{H} \in \NS_{\mathbb{Q}}(\widehat{A})$, satisfying 
\begin{equation*}
({H^n}/{n!}) ({\widehat{H}^n}/{n!}) =1.
\end{equation*}
Moreover, for each $0 \le i \le n$,  the induced cohomological transform gives rise to
an isomorphism
$\Phi_\pP^{\mathrm{H}} : H^{2i}_{\mathrm{alg}}(A, \mathbb{Q}) \to H^{2(n-i)}_{\mathrm{alg}}(\widehat{A},\mathbb{Q})$, satisfying
\begin{equation*}
\Phi_\pP^{\mathrm{H}}\left( \frac{H^i}{i!} \right) = \frac{(-1)^{n-i} H^n}{n! (n-i)!} \, 
\widehat{H}^{n-i}. 
\end{equation*}
\end{lem}
Let $H$, $\widehat{H}$ be as in Lemma~\ref{classicalcohomoFMT}. 
We write 
\begin{align}\label{notation:v}
v_i(-) = i! H^{n-i} \cdot \ch_{i}(-), \qquad
\widehat{v}_i(-)=i! \widehat{H}^{n-i} \cdot \ch_i(-).
\end{align}
For $\gamma=(\gamma_0, \ldots, \gamma_n) \in H^{2\ast}(A, \mathbb{Q})$, 
we also write $v_i(\gamma)=i! H^{n-i}\gamma_i$ and similarly for 
$\widehat{v}_i(-)$. 
The following is a particular case of \cite[Theorem~3.6]{PiyFMT}.
\begin{lem}
\label{prop:antidiagonal-rep-cohom-FMT}
We have the following equality for the induced cohomological transform $\Phi_{\pP}^{\mathrm{H}}: H^{2*}_{\mathrm{alg}}(A, \mathbb{Q}) \to H^{2*}_{\mathrm{alg}}(\widehat{A},\mathbb{Q})$:
$$
\widehat{v}_{i} \left(\Phi_{\pP}^{\mathrm{H}}(\gamma) \right) = 
\frac{(-1)^i n!}{H^n} \,   v_{n-i}(\gamma).
$$
\end{lem}

We also have the following corollary: 
\begin{cor}\label{cor:cohomological}
The induced cohomological transform $\Phi_{\pP}^{\mathrm{H}}: H^{2*}_{\mathrm{alg}}(A, \mathbb{Q}) \to H^{2*}_{\mathrm{alg}}(\widehat{A},\mathbb{Q})$ of  $\Phi_{\pP}^{A \to \widehat{A}}$ fits into the following diagrams:
\[
\begin{tikzcd}
H^{2i}_{\mathrm{alg}}(A, \mathbb{Q})  \ar{d}{p_{H, i}} \ar{r}{\Phi_{\pP}^{\mathrm{H}}}  &   H^{2(n-i)}_{\mathrm{alg}}(\widehat{A}, \mathbb{Q}) \ar{d}{p_{\widehat{H}, n-i}} \\
H^{2i}_{\mathrm{alg}}(A, \mathbb{Q}) \ar{r}{\Phi^{\mathrm{H}}_{\pP}}  &   H^{2(n-i)}_{\mathrm{alg}}(\widehat{A}, \mathbb{Q}),
\end{tikzcd}
\quad 
\begin{tikzcd}
H^{2i}_{\mathrm{alg}}(A, \mathbb{Q})  \ar{d}{p_{H, i}^{\perp}} \ar{r}{\Phi_{\pP}^{\mathrm{H}}}  &   H^{2(n-i)}_{\mathrm{alg}}(\widehat{A}, \mathbb{Q}) \ar{d}{p_{\widehat{H}, n-i}^{\perp}} \\
H^{2i}_{\mathrm{alg}}(A, \mathbb{Q}) \ar{r}{\Phi^{\mathrm{H}}_{\pP}}  &   H^{2(n-i)}_{\mathrm{alg}}(\widehat{A}, \mathbb{Q}).
\end{tikzcd}
\]
\end{cor}
\begin{proof}
The first diagram is a direct consequence of Lemma \ref{classicalcohomoFMT}. The second diagram follows from the relation
\eqref{eqn:map-decom}.
\end{proof}

In the case of $n=3$, 
the Fourier-Mukai transform $\Phi_{\pP}^{A \to \widehat{A}}$ with the Poincar\'e bundle as kernel
preserves double tilt hearts as follows:  
\begin{lem}\emph{(\cite[Theorem~5.3]{PiyFMT})}
\label{prop:ab-equiv-abelian-3}
Suppose that $A$ is an abelian 3-fold. 
Then for any $t \in \mathbb{R}_{>0}$, we have 
$$
\Phi_{\pP}^{A \to \widehat{A}} [1] \left( \aA_{\sqrt{3}tH/2, tH/2}\right) 
= \aA_{\sqrt{3}\widehat{H}/{2t}, -\widehat{H}/2t}
$$ 
where $\widehat{H} \in \NS_{\mathbb{Q}}(\widehat{A})$ is the induced ample
class as in Lemma~\ref{classicalcohomoFMT}.
\end{lem}

\subsection{(Semi)homogeneous sheaves}
\label{semihomogeneous}
We recall 
(semi)homogeneous sheaves on abelian varieties, and study
the effect of tensoring them to the stability. 
The arguments here will be also used in the proof of full support property. 

A vector bundle $E$ on an abelian variety $A$ is called \textit{homogeneous} if we have
$T_x^*E \cong E$ for all $x \in A$. 

\begin{prop}[\cite{Mu3}]
\label{prop:homogeneous-bundle-filtration}
A vector bundle $E$ on $A$ is homogeneous if and only if $E$
can be filtered by line bundles from $\Pic^0(A)$.
\end{prop}
For a coherent sheaf $E$ on $A$, 
we define
\begin{align}\label{defi:Phi(E)}
\Xi(E) :=\{(x, L) \in A \times \widehat{A} : T_x^{\ast}E \otimes L \cong E\}.   
\end{align}
By~\cite[Proposition~4.5]{Mu4}, we have 
$\dim \Xi(E) \le n$, where $n$ is the dimension of $A$.
A coherent sheaf $E$ on $A$
is \textit{semihomogeneous} if
$\dim \Xi(E)=n$.
If $E$ is a vector bundle, this is 
equivalent to that 
for every $x \in A$ there exists a flat line bundle $\pP_{A \times \{\widehat{x}\}}$ on $A$
such that $T_x^*E \cong E \otimes  \pP_{A \times \{\widehat{x}\}}$.
Also a coherent sheaf $E$ is called \textit{simple} if we have $\End_{A}(E) \cong \mathbb{C}$. 

\begin{lem}[{\cite[Theorem~5.8]{Mu3}}]
\label{prop:Mukai-semihomognoeus-properties}
Let $E$ be a simple vector bundle on an $n$-dimensional abelian variety $A$. Then the following conditions are equivalent:
\begin{enumerate}
\renewcommand{\labelenumi}{(\roman{enumi})}
\item $\dim H^1(A, \eE nd(E))=n$,
\item $E$ is semihomogeneous,
\item $\eE nd(E)$ is a homogeneous vector bundle.
\end{enumerate}
\end{lem}

\begin{lem}[{\cite{Mu3, OrA}}]
\label{prop:semihomo-numerical}
The following holds.
\begin{enumerate}
\renewcommand{\labelenumi}{(\roman{enumi})}
\item 
A rank $r$ simple semihomogeneous bundle $E$ has the Chern character 
\begin{equation*}
\label{semihomochern}
\ch(E) = r \cdot e^{c_1(E)/r}.
\end{equation*}
\item 
For any $D_{A} \in \mathrm{NS}_{\mathbb{Q}}(A)$, 
there  exists a simple semihomogeneous bundle $E$ on $A$ with 
$\ch(E) = r \cdot e^{D_{A}}$ for some $r \in \mathbb{Z}_{>0}$. 
\item 
Let  $E$ be a semihomogeneous bundle on $A$. Then $E$ is Gieseker semistable with respect to any ample line bundle $L$, and if $E$ is simple then it is slope stable with respect to $c_1(L)$.
\end{enumerate}
\end{lem}

Below we assume that $A$ is an abelian 3-fold. 
Let $\omega, B \in \mathrm{NS}_{\mathbb{Q}}(A)$
such that $\omega$ is an ample class. 
\begin{prop}
\label{prop:stab-under-tensoring}
Let $V$ be a simple semihomogeneous bundle on $A$ and let
$$D = \frac{c_1(V)}{\rk(V)}.$$
Then we have the following:
\begin{enumerate}
\renewcommand{\labelenumi}{(\roman{enumi})}
\item $E \in \Coh(A)$ is $\mu_{\omega, B}$-semistable if and only if $E \otimes V$ is $\mu_{\omega, B+D}$ semistable.
\item $E \in \bB_{\omega, B}$ is $\nu_{\omega, B}$-semistable 
if and only if $E \otimes V  \in \bB_{\omega, B+D}$ is 
$\nu_{\omega, B+D}$-semistable.
\item  $E \in \aA_{\omega, B}$ is $\sigma_{\omega, B}$-semistable
if and only if $E \otimes V  \in \aA_{\omega, B+D}$ is $\sigma_{H, B+D}$-semistable.
\end{enumerate}
\end{prop}
\begin{proof}
(i) \ This follows from the fact that slope semistability is preserved under tensoring by semistable vector bundles and 
from Lemma \ref{prop:semihomo-numerical} the simple semihomogeneous bundle  $V$ is 
slope stable. 

(ii) \ From part (i), we have $\bB_{\omega, B} \otimes V \subset \bB_{\omega, B+D}$; so 
$E \otimes V  \in \bB_{\omega, B+D}$. 
From Lemma \ref{prop:semihomo-numerical}, 
$$
\ch(V) = \rk(V) \cdot e^D,
$$
so $\ch^{B+ D}(E \otimes V) = \rk(V) \ch^B(E)$. Hence 
\begin{equation}
\label{eqn:tensor-tilt-slope-value}
\nu_{\omega, B +D} (E \otimes V) =  \nu_{\omega, B}(E).
\end{equation}

Suppose for a contradiction
$E \otimes V  \in \bB_{\omega, B+D}$ is not $\nu_{\omega, B+D}$ semistable; so the following destabilizing 
short exact sequence exists in $\bB_{\omega, B+D}$:
$$
0 \to P \to E \otimes V \to Q \to 0.
$$
By tensoring with the dual $V^{\vee}$ we get the following 
short exact sequence exists in $\bB_{\omega, B}$:
\begin{equation}
\label{ses:destab-tensor-B}
0 \to P \otimes V^{\vee} \to E \otimes \eE nd(V)    \to Q \otimes V^{\vee}  \to 0.
\end{equation}
From Lemma \ref{prop:Mukai-semihomognoeus-properties}, the bundle $\eE nd(V) = V \otimes V^{\vee}$ is a homogeneous bundle, and from Proposition \ref{prop:homogeneous-bundle-filtration} it can be filtered by line bundles $\{L_j\}$ from $\Pic^0(A)$. Therefore, 
$E \otimes \eE nd(V) \in \bB_{\omega, B}$ is filtered by $\nu_{\omega, B}$-semistable objects $\{ E \otimes L_j \}$ in  $\bB_{\omega, B}$; hence,
$E \otimes \eE nd(V) \in \bB_{\omega, B}$ is $\nu_{\omega, B}$-semistable.
However, according to \eqref{eqn:tensor-tilt-slope-value}, the short exact 
sequence \eqref{ses:destab-tensor-B} destabilizes $E \otimes \eE nd(V)$. This is the required contradiction. 

(iii) \ From part (ii), we have $\aA_{\omega, B} \otimes V \subset \aA_{\omega, B+D}$; so 
$E \otimes V  \in \aA_{\omega, B+D}$. Then the rest of the proof is  similar to part (ii).
\end{proof}

\subsection{Full support property via FM transforms on abelian 3-folds}
Let $A$ be an abelian 3-fold
and
$H \in \NS_{\mathbb{Q}}(A)$ be an ample class.
Let $v_i$ be the vectors as in (\ref{notation:v}), 
and consider the following form of central charge functions
\begin{align*}
W_{H ,t}^{p,q} = \left( - v_3  +   q v_2 + p v_0 \right)
+ i \left( v_2 - t v_1 \right)
\end{align*}
for $t, p, q \in \mathbb{R}$.

\begin{prop}\label{prop:ZW}
Let $t \ne 0$, and $a, b \in \mathbb{R}$. 
Then we have the following:
\begin{align*}
Z_{\sqrt{3}\lvert t \rvert H/2, t H/2}^{a,b} \thicksim W_{H ,t}^{p, q}
\end{align*}
for some $p,q \in \mathbb{R}$. 
Here 
$\alpha=\sqrt{3}|t|/2, a, b$ satisfy (\ref{ineq:ab}), that is 
$a >  (t^2/24) + (|t b|/4)$, if and only if $t,p,q$ satisfy
$t(t-q) < \frac{p}{t}<0$.
\end{prop} 
\begin{proof}
From the definition of $v_i$ and 
$\ch^{tH/2}(-) = e^{-t H/2}\ch(-)$, we have 
\begin{align*}
&H^3 \ch_0^{t H/2} = v_0, \ 
H^2 \ch_1^{t H/2} =   v_1 - t v_0/2, \
2 H \ch_2^{t H/2} = v_2 - t v_1 + t^2 v_0/4, \\ 
&6\ch_3^{t H/2}= v_3 -3 t v_2/2 + 3t^2 v_1/4 -t^3v_0/8.
\end{align*}

Now, by direct substitution one can check that
\begin{align*}
&Z_{\sqrt{3}\lvert t \rvert H/2, t H/2}^{a,b} \\
& = \left( -\ch_3^{t H/2} + b H \ch_2^{t H/2} +   a H^2  \ch_1^{t H/2} \right)
+  i  \frac{t}{2} \left( H \ch_2^{t H/2} - \frac{t^2}{8} H^3 \ch_0^{t H/2}  \right) \\
& = \frac{1}{6} \left( -  v_3  +   q v_2 + p v_0 +r(v_2-tv_1) \right) 
+ i \frac{t}{4} \left( v_2 - t v_1 \right) \\
& \thicksim W_{H ,t}^{p,q},
\end{align*}
where 
\begin{align*}
q  = \frac{3t}{4}+ \frac{6a}{t}, \ 
p= -3at + \frac{3bt^2}{4} + \frac{t^3}{8}, \ 
r  = \frac{3t}{8} + 3b - \frac{6a}{t}.
\end{align*}
By straightforward computation one can check that  
$|t|, a, b$ satisfy  
$a >  (t^2/24) + (|t b|/4)$, if and only if $t,p,q$ satisfy
$t(t-q) < \frac{p}{t}<0$.
\end{proof} 

Consequently, we get the following particular case of  
Proposition~\ref{prop:stabcond} and Lemma~\ref{lem:sharp}
in an alternative form.
\begin{prop}
\label{thm:stab-family-new-para}
Let the numbers $t, p, q\in \mathbb{R}$ satisfy
\begin{align}\label{eqn:stab-family-condition}
t\neq 0, \ 
t(t-q)<\frac{p}{t}<0. 
\end{align}
Then the pair 
$$
\left( W_{H ,t}^{p,q}, \aA_{\sqrt{3}\lvert t\rvert H/2, tH/2}\right)
$$
defines a Bridgeland stability condition on $A$
with respect to $(\Lambda_H^{\sharp}, \cl^{\sharp})$. 
\end{prop}

Let us write %the Fourier-Mukai transforms:
\begin{align*}
\Psi : = \Phi_{\pP}^{A \to \widehat{A}} [1] : D^b(A)  \to D^b(\widehat{A}), \
\widehat{\Psi} : = \Phi_{\pP^\vee}^{\widehat{A} \to A} [2] : D^b(\widehat{A})  \to D^b(A).
\end{align*} 
Then $\widehat{\Psi}$ is the quasi inverse of $\Psi$, and $\Psi$ is the  quasi inverse of 
$\widehat{\Psi}$. 
%Also we write the induced cohomological  transforms of $\Psi$ and  $\widehat{\Psi}$ by $\Psi^H$ and $ \widehat{\Psi}^H$ respectively.
Recall  that $\Psi_{\ast}   W_{H, t}^{p,q} : K(\widehat{A}) \to \mathbb{C}$ is the function defined by 
\begin{align*}
\Psi_{\ast}   W_{H, t}^{p,q}(-) = W_{H, t}^{p,q}(\widehat{\Psi}(-)). 
\end{align*}
Let $\widehat{H}$ be the induced 
ample divisor on $\widehat{A}$ as in Lemma~\ref{classicalcohomoFMT}, 
and $\widehat{v}_i$ be the vectors as in (\ref{notation:v}). 
\begin{prop}\label{prop:W}
Let $t, p ,q \in \mathbb{R}$ such that $t>0$.
We have  
$$
\Psi_{\ast}   W_{H, t}^{p,q}  \thicksim  W_{\widehat{H}, t'}^{p',q'} 
$$
for some $t', p', q' \in \mathbb{R}$ defined by
\begin{align}
\label{eqn:parameter-under-FMT}
t'= -\frac{1}{t} < 0, \ p' = - \frac{1}{p}, \  q' = \frac{t q}{p}.
\end{align}
Moreover, if $\{t> 0,p,q\}$ satisfies (\ref{eqn:stab-family-condition}), then $\{ t'<0, p',q'\}$ also satisfies (\ref{eqn:stab-family-condition}).
\end{prop}
\begin{proof}
From Lemma  \ref{classicalcohomoFMT}, we have
$$
v_{i}(\widehat{\Psi}(-)) = (-1)^{i} \frac{H^3}{6} \ \widehat{v}_{3-i}(-).
$$
Hence 
\begin{align*}
&\Imm \left( \Psi_{\ast}   W_{H, t}^{p,q} \right)  
= \frac{H^3}{6} \left(  \widehat{v}_1 + t  \widehat{v}_2  \right)
=\frac{H^3 t}{6} \Imm W_{\widehat{H}, t'}^{p', q'}, 
\\
&\Ree \left( \Psi_{\ast}   W_{H, t}^{p,q}  \right)  
= \frac{H^3}{6}\left( \widehat{v}_0 + q \widehat{v}_1 + p \widehat{v}_3 \right)
=\frac{H^3}{6}\left(-p \cdot \Ree W_{\widehat{H}, t'}^{p', q'}+tq
\cdot \Imm W_{\widehat{H}, t'}^{p', q'}\right).
\end{align*}
Therefore the first claim holds. 
By direct computation  one can check that if 
$\{t> 0,p, q,r\}$ satisfies \eqref{eqn:stab-family-condition}, then we have
\begin{equation*}
t'(t' - q') < \frac{p'}{t'}<0.
\end{equation*} 
That is \eqref{eqn:stab-family-condition} holds for $\{ t'<0, p',q'\}$.
\end{proof}

For $t \in \mathbb{R}_{>0}$, by Lemma~\ref{prop:ab-equiv-abelian-3}, 
Proposition~\ref{prop:W} and Proposition~\ref{prop:ZW}, 
we have the following:
\begin{lem}\label{prop:equiv-new-parameters-stab}
Let $t>0, p ,q \in \mathbb{R}$ satisfy \eqref{eqn:stab-family-condition}.
Then we have the following equivalence of Bridgeland stability conditions:
$$
\Psi_{\ast}  \left( W_{H, t}^{p,q},   \aA_{\sqrt{3}tH/2, tH/2} \right) \thicksim  
\left( W_{\widehat{H}, t'}^{p',q'} ,  \aA_{-\sqrt{3}t'\widehat{H}/2, t'H/2} \right)
$$
for $t'<0, p' ,q' \in \mathbb{R}$ defined as in \eqref{eqn:parameter-under-FMT} satisfying \eqref{eqn:stab-family-condition}.
\end{lem} 

Consequently we prove the following:
\begin{lem}\label{lem:conseqnce}
\label{lem:full-support-property-parallel-case}
If $t>0, p ,q \in \mathbb{R}$ satisfy \eqref{eqn:stab-family-condition},
then the Bridgeland stability condition defined by the pair
\begin{align}\label{pair:W}
\left( W_{H, t}^{p,q},   \aA_{\sqrt{3}tH/2, tH/2} \right)
\end{align}
satisfies the full support property, i.e. 
it is an element 
of $\Stab(A)$.
\end{lem}
\begin{proof}
From Lemma~\ref{lem:sharp}, 
there exists a quadratic form, say
$Q_1$, which establishes the support property for 
the stability condition (\ref{pair:W})
with respect to 
$(\Lambda_H^{\sharp}, \cl^{\sharp})$. 
Choose $t'<0, p' ,q' \in \mathbb{R}$ as in Lemma \ref{prop:equiv-new-parameters-stab}.
Now from Lemma~\ref{lem:sharp}, 
there exists a quadratic form, say $Q_2$, which establishes the support property for 
the stability condition 
$\left(W_{\widehat{H}, t'}^{p',q'} ,  \aA_{-\sqrt{3}t'\widehat{H}/2, t'H/2}\right)$
with respect to 
$(\Lambda_{\widehat{H}^{\sharp}}, \cl^{\sharp})$. 
Hence, from Lemma \ref{prop:equiv-new-parameters-stab}
and Corollary~\ref{cor:cohomological}, the quadratic form
$Q_2(\Psi(-))$ establishes the support property for 
the stability condition (\ref{pair:W})
with respect to 
$(\Lambda_H^{\flat}, \cl^{\flat})$
defined in (\ref{cl:flat}). 
Therefore, 
the quadratic form 
\begin{equation}
Q(-) = Q_1(-) + \lambda Q_2 (\Psi(-)), \ \ \text{for any } \lambda \in \mathbb{R}_{>0}
\end{equation}
establishes the support property for 
the stability condition (\ref{pair:W})
with respect to 
$(\Gamma, \ch)$, that is the full support property. 
\end{proof}

\begin{thm}\label{thm:fullsupport}
Let $B \in \NS_{\mathbb{Q}}(A)$,
$\alpha=\sqrt{3}t/2$ for some 
$t \in \mathbb{Q}_{>0}$ and $a,b  \in \mathbb{R}$ satisfying
(\ref{ineq:ab}). 
Then 
the stability condition
$(Z_{\alpha H, B}^{a, b}, \aA_{\alpha H, B})$
in Proposition~\ref{prop:stabcond}
satisfies the full support property.
\end{thm}
\begin{proof}
Let us fix a slope semistable semihomogeneous bundle $V$ on $A$ such that
\begin{equation}
\frac{c_1(V)}{\rk(V)}= -B+\frac{t}{2}H.
\end{equation}
From Lemma \ref{prop:semihomo-numerical}, $\ch(V) = \rk(V) \cdot e^{(-B+tH/2)}$.
Let $E$ be a 
$(Z_{\alpha H, B}^{a, b}, \aA_{\alpha H, B})$-semistable 
object. 
By Proposition \ref{prop:stab-under-tensoring}, 
$E \otimes V$
is a
($Z_{\sqrt{3}tH/2, tH/2}^{a, b}, \aA_{\sqrt{3}tH/2, tH/2})$-semistable object. 
Let $Q$ be the quadractic form on $\Gamma$
which establishes the full support property for (\ref{pair:W}),
which exists
by Theorem~\ref{lem:conseqnce}. 
Since $\ch(E \otimes V) = \rk(V)\cdot \ch^{B-tH/2}(E)$, 
the quadractic form $Q(e^{-B+tH/2}(-))$
establish the support property for
$(Z_{\alpha H, B}^{a, b}, \aA_{\alpha H, B})$. 
\end{proof}

Consequently, we arrive at the following, which is the main result of Section~\ref{section:Bridgeland_stab_conditions}.
It implies in particular the existence of stability conditions on $A$ with respect to $(\Gamma, \ch)$,
or equivalently that $\Stab(A) \neq \varnothing$.

\begin{thm} \label{thm_full_support}
There is a continuous family of Bridgeland stability conditions in $\Stab(A)$, parameterized by the set
$$
(\omega , B, a, b) \in \Amp_{\mathbb{R}}(A) \times  \NS_{\mathbb{R}}(A)
\times \mathbb{R}  \times \mathbb{R},   \ \ a >  \frac{1}{18}  + \frac{\sqrt{3}}{6}|b|
$$
via
$$
(\omega , B, a, b) \mapsto \left( Z_{\omega, B}^{a, b}, \aA_{\omega, B} \right).
$$
In particular, 
there is a continuous embedding
$\Amp_{\mathbb{C}}(A) \to \Stab(A)$
given by 
$B+i\omega \mapsto \sigma_{B, \omega}$. 
The action of $\Aut(D^b(A))$ on $\Stab(A)$
preserves the connected component $\Stab^{\circ}(A)$
which contains the image of the above map.
\end{thm}
\begin{proof}
The first statement is similar to the proof of \cite[Proposition~8.10]{BMS}, 
using Theorem~\ref{thm:fullsupport}.
Below, we give a proof of the second statement. 
Let $F$ be a derived autoequivalence of $A$. If the Fourier-Mukai kernel of 
$F$ is a vector bundle (up to a shift) then the claim is a direct consequence of~\cite[Theorem~1.1]{PiyFMT}.
Suppose that the Fourier-Mukai 
kernel of $F$ is not a vector bundle up to a shift. 
By a theorem of Orlov~\cite[Proposition~9.53]{Huybook}, 
the kernel of an auto-equivlance between two abelian varieties is represented by a sheaf up to shift.
Therefore for
a derived equivalence defined by $F' = \Phi_{\pP}^{A \to \widehat{A}} \circ \otimes \oO_A(nH) \circ F$, 
where $H$ is ample and $n$ is sufficiently large, the Fourier-Mukai kernel of $F'$
is a vector bundle up to a shift. Again from \cite[Theorem~1.1]{PiyFMT}, $F'$
takes $\Stab^0(A)$ to $\Stab^0(\widehat{A})$. 
Since 
$\Phi_{\pP}^{A \to \widehat{A}}$ and $\otimes \oO_A(nH)$ preserve connected
components $\Stab^0(A)$, 
$\Stab^0(\widehat{A})$, the equivalence $F$
also preserves $\Stab^0(A)$.
\end{proof}

\subsection{Standard slice}
In what follows, we focus on 
some subspace of $\Stab(A)$
and find stability conditions on it 
where semistable objects 
coincide with Gieseker semistable sheaves. 

We fix an ample divisor $H$ and consider $B+i\omega$ written as 
\begin{align*}
\omega=\alpha H, \ 
B=\beta H, \ \alpha\in \mathbb{R}_{>0}, 
\beta \in \mathbb{R}.
\end{align*}
We write  
$\sigma_{\alpha H, \beta H}=
(Z_{\alpha H, \beta H}, \aA_{\alpha H, \beta H})$
as $\sigma_{\alpha, \beta}=(Z_{\alpha, \beta}, \aA_{\alpha, \beta})$
and so on. 
Recall that we considered the surjective map
\begin{align}\label{isom:gamma}
\Gamma_{\mathbb{Q}} \twoheadrightarrow \mathbb{Q}^4, \ 
\ch_i \mapsto v_i=i! H^{3-i}\ch_i.
\end{align}
For $\beta \in \mathbb{R}$ let $(v_0^{\beta}, v_1^{\beta}, v_2^{\beta}, v_3^{\beta}) \in \mathbb{R}^4$
be the vector corresponding to $v(\ch^{\beta H})$,
\begin{align*}
&v_0^{\beta}=v_0, \ v_1^{\beta}=v_1-\beta v_0, \ 
v_2^{\beta}=v_2-2\beta v_1+\beta^2 v_0, \\ 
&v_3^{\beta}=v_3-3\beta v_2+3\beta^2 v_1-\beta^3 v_0. 
\end{align*}

Consider the subspace
\[ \Stab_H(A) \subset \Stab(A) \]
of stability conditions $(Z, \aA)$ such that $Z$ factors through the map (\ref{isom:gamma}). Let
\[ \Stab_H^{\circ}(A) \subset \Stab_H(A) \]
denote the component which contains the elements $\sigma_{\alpha, \beta}$ (the component exists by Theorem~\ref{thm_full_support}).
The space $\Stab^{\circ}_H(A)$ is completely described in~\cite{BMS} as follows. 
Let $\mathfrak{B} \subset \mathbb{R}^4$ be the 
open subset given by
\begin{align*}
\mathfrak{B}=\left\{(\alpha, \beta, a, b) \in \mathbb{R}^4 : 
\alpha>0, a>\frac{\alpha^2}{18}+\frac{\sqrt{3}}{6}\lvert b \rvert \alpha
\right\}. 
\end{align*}
For $(\alpha, \beta, a, b) \in \mathfrak{B}$, 
the central charge
$Z_{\alpha, \beta}^{a, b}:=Z_{\alpha H, \beta H}^{a, b}$
in (\ref{Zab}) is written as 
\begin{align*}
Z_{\alpha, \beta}^{a, b}=\frac{1}{6}\left(-v_3^{\beta}+3bv_2^{\beta}+6a v_1^{\beta}
+i\alpha \left(3v_2^{\beta}-\alpha^2 v_0^{\beta}  \right) \right).
\end{align*}

\begin{thm}\emph{(\cite{BMS})}\label{thm:BMS}
We have the continous embedding
\begin{align}\label{emb:B}
\mathfrak{B} \to \Stab^{\circ}_H(A), \ (\alpha, \beta, a, b) \mapsto 
\sigma_{\alpha, \beta}^{a, b}:=(Z_{\alpha, \beta}^{a, b}, \aA_{\alpha, \beta})
\end{align}
whose image gives a slice of the 
$\widetilde{\GL}_2^{+}(\mathbb{R})$-action on $\Stab_H^{\circ}(A)$. 
\end{thm}

The upper-half plane $\BH \subset \mathbb{C}$ is embedded into
$\mathfrak{B}$ by 
\begin{align*}
\beta+i\alpha \mapsto (\alpha, \beta, \alpha^2/2, b=0)
\end{align*}
and
its image under the embedding (\ref{emb:B}) is
$\sigma_{\alpha, \beta}=\sigma_{\alpha, \beta}^{a=\alpha^2/2, b=0}$. 

\subsection{Gieseker chamber}
We keep the notation from the previous subsection. 
Let 
$\overline{\Gamma}_{+} \subset \Gamma$ be the 
subset of $v \in \Gamma$
such that either 
\begin{align*}
v_0>0, \mbox{ or } v_1>0=v_0, \mbox{ or }
v_2>0=v_1=v_0, \mbox{ or }
v_3>0 =v_2=v_1=v_0. 
\end{align*}
The set $\overline{\Gamma}_{+}$ contains $\Gamma_{+}$, 
the set of Chern characters of coherent sheaves.

We first consider $\nu_{\alpha, \beta}$-semistable objects in $\bB_{\alpha, \beta}$.
For $v \in \overline{\Gamma}_{+}$, 
by the same arguments as in~\cite[Theorem~3.1]{Maci}, 
we can describe the 
wall and chamber structure for
$\nu_{\alpha, \beta}$-semistable objects on $\bB_{\alpha, \beta}$
with Chern character $v$
on the $(\alpha, \beta)$-plane:
\begin{align*}
\hH=\{\beta+i\alpha: \alpha \in \mathbb{R}_{>0}, \beta \in \mathbb{R}\}. 
\end{align*}
The  
walls are (after rescaling $\alpha$ by $\sqrt{3}\alpha$)
finite nested semi-circles: 
each wall is a semi-circle 
contained in 
$\beta<v_1/v_0$
(where $v_1/v_0=\infty$ for $v_0=0$), 
whose center lies on the $\beta$-axis,
and for any two walls one of them is contained in the interior of the other.

When $(\alpha, \beta)$ lies in the outer of every wall, 
the $\nu_{\alpha, \beta}$-semistable objects are 
described in terms of stability conditions on sheaves. 
For this purpose, we introduce the following notion, 
which lies between slope stability and Gieseker stability: 
\begin{defi}
For a smooth projective 3-fold $X$ and an ample divisor $H$ on it, 
a coherent sheaf $E \in \Coh(X)$ is called $\nu_{H}$-semistable 
if it is pure and for any subsheaf $0 \subsetneq F \subset E$, we have
\begin{align*}
\overline{\chi}_H^{\dag}(F)(m) \le \overline{\chi}_H^{\dag}(E)(m) 
\end{align*}
for $m\gg 0$. 
Here for a polynomial $p(m)$ in $m$
we let $p^{\dag}(m) = p(m) - p(0)$.
\end{defi}
In the case of $X=A$ we have the following.
\begin{lem}\label{lem:nuGie}
\begin{enumerate}
\renewcommand{\labelenumi}{(\roman{enumi})}
\item 
A torsion free sheaf
$E \in \Coh(A)$
is $\nu_H$-semistable if and only if for any subsheaf $F \subset E$, we have
\begin{align*}
\frac{v_1(F)}{v_0(F)} \le \frac{v_1(E)}{v_0(E)}, \ \mbox{ and } \
\frac{v_2(F)}{v_0(F)}\le \frac{v_2(E)}{v_0(E)} \ \mbox{ if } \
\frac{v_1(F)}{v_0(F)}=\frac{v_1(E)}{v_0(E)}. 
\end{align*}
In particular, it is slope semistable. 
\item 
A $\nu_H$-semistable torsion free sheaf $E \in \Coh(A)$ is 
Gieseker-semistable if and only if for any 
$\nu_H$-semistable subsheaf $F \subset E$
with the same $(v_1/v_0, v_2/v_0)$, we have 
\begin{align*}
\frac{v_3(F)}{v_0(F)} \le \frac{v_3(E)}{v_0(E)}. 
\end{align*} 
\item The same statements of (i), (ii) hold 
after replacing $v_i$ with $v_i^{\beta}$ for any $\beta \in \mathbb{R}$. 
\end{enumerate}
\end{lem}

\begin{lem}\label{lem:nuH}
For $v \in \overline{\Gamma}_{+}$, let $(\alpha, \beta) \in \hH$
lies in the outer of every wall with respect to the 
$\nu_{\alpha, \beta}$-stability with Chern character $v$. 
Then for $E \in D^b(A)$ with $\ch(E)=v$, it is 
a $\nu_{\alpha, \beta}$-semistable object in $\bB_{\alpha, \beta}$
if and only if it is $\nu_H$-semistable coherent sheaf. 
\end{lem}
\begin{proof}
The proof is similar to the surface case, for example 
see~\cite[Theorem~1,2, Lemma~2.6]{MR3217639}. 
\end{proof}

For any $v \in \overline{\Gamma}_{+}$ with $(v_0, v_1) \neq (0, 0)$
the curve $\nu_{\alpha, \beta}(v)=0$, i.e.
\begin{align*}
v_0 \beta^2-\frac{v_0}{3}\alpha^2
-2v_1 \beta+v_2=0
\end{align*}
intersects each wall at the top of the semi-circle. 
We define
\begin{align*}
\sS_{v} \subset \hH
\end{align*}
to be the intersection 
of the outer of every wall and 
the region $\nu_{\alpha, \beta}(v)>0$. 
If $(v_0, v_1)=0$, then there is no wall with respect to the 
$\nu_{\alpha, \beta}$-stability, and $\nu_{\alpha, \beta}(v)=\infty$,
so we set $\sS_v=\hH$. 
In any case for fixed $\alpha>0$, we have
$(\alpha, \beta) \in \sS_v$ for $\beta \ll 0$. 

The following proposition proves the existence of a Gieseker chamber on $\Stab^{\circ}(A)$.
\begin{prop}\label{lem:DTchamber}
For any $(\alpha, \beta) \in \sS_v$, 
there exists $s(\alpha, \beta)>0$ such that 
for any $s > s(\alpha, \beta)$
the following holds: an object $E \in D^b(A)$ with 
$\ch(E)=v$ is a $Z_{\alpha, \beta}^{a=s, b=0}$-semistable 
object in $\aA_{\alpha, \beta}$ if and only if 
it is a $H$-Gieseker semistable sheaf. 
\end{prop}

\begin{proof}
For $t\ge 0$, consider the central charge
\begin{align*}
W_t=(1+t) \alpha^2 v_1^{\beta}-3t v_3^{\beta}
+i\alpha \left( 3v_2^{\beta}-\alpha^2 v_0^{\beta} \right). 
\end{align*}
For all $t>0$ we have
\begin{align*}
W_t \sim Z_{\alpha, \beta}^{a=s, b=0}, \ 
s=\frac{1+t}{18t}\alpha^2 >\frac{\alpha^2}{18}.
\end{align*}
Hence by Theorem~\ref{thm:BMS} the pair $(W_t, \aA_{\alpha, \beta})$ is a Bridgeland stability condition for any $t>0$.
These stability conditions 
degenerate to the very weak stability condition
$(W_0, \aA_{\alpha, \beta})$ at $t=0$, see \cite[Section~3.4]{PiYT}. 

Let $\dD_v \subset D^b(A)$ be the set of objects with 
Chern character $v$. 
By the definition of $\sS_v$, we have 
$\Imm W_0(E)>0$
for any $E \in \dD_v$. 
Therefore by~\cite[Lemma~2.19]{PiYT}, 
we have
\begin{align}\notag
&\{ E \in \dD_v : 
E \mbox{ is } W_0\mbox{-semistable in } \aA_{\alpha, \beta}\} \\
\label{set:1}
& =\{E \in \dD_v : E \mbox{ is } \nu_{\alpha, \beta}
\mbox{-semistable in } \bB_{\alpha, \beta}\}.
\end{align}
By Lemma~\ref{lem:nuH} and the definition of $\sS_v$,
(\ref{set:1}) coincides with 
\begin{align}\label{set:2}
\{E \in \dD_v : E \mbox{ is } \nu_H\mbox{-semistable in }\Coh(A)\}. 
\end{align}

On the other hand by~\cite[Proposition~2.27]{PiYT}, 
for $0<t\ll 1$ we have 
\begin{align*}
&\{ E \in \dD_v : 
E \mbox{ is } W_t\mbox{-semistable in } \aA_{\alpha, \beta}\} \\
&=\left\{E \in \dD_v : \begin{array}{c}
E \mbox{ is } \xi\mbox{-semistable among } W_0
\mbox{-semistable } \\
\mbox{ objects in }
\aA_{\alpha, \beta} \mbox{ with }
\arg W_0(-)=\arg W_0(v)
\end{array}
\right\},
\end{align*}
where $\xi$ is the slope function given by
\begin{align*}
\xi=\frac{3v_3^{\beta}-
\alpha^2 v_1^{\beta}}{3v_2^{\beta}-\alpha^2 v_0^{\beta}}. 
\end{align*}
By Lemma~\ref{lem:nuGie}
and (\ref{set:1}), (\ref{set:2}),  
for $v_0>0$ the last set of objects 
is the set of $H$-Gieseker semistable sheaves $E \in \Coh(A)$
with $\ch(E)=v$. 
Since $s=(1+t)\alpha^2/18t$ goes to $\infty$ for $t\to +0$, this implies the Lemma in case $v_0 > 0$.
The case $v_0=0$ is similar. 
\end{proof}

\section{Wallcrossing on abelian threefolds} \label{section:Wallcrossing_on_abelian_threefolds}
Let $A$ be an abelian threefold and let $\widehat{A} = \Pic^0(A)$ be its dual. We set
\[ \A = A \times \widehat{A}. \]
Let also $H \in \Pic(A)$ be a fixed ample class.

\subsection{Reduced DT invariants for Bridgeland semistable objects} \label{Section_reduced_DT_Bridgeland}
In Section~\ref{Subsection_definition_of_reduced_DT}, we defined 
$\A$-reduced Donaldson-Thomas invariants
\begin{align*}
\DTb_H(v) \in \mathbb{Q}[\A]
\end{align*}
counting $H$-Gieseker semistable sheaves on $A$.
Here we define reduced Donaldson--Thomas invariants counting Bridgeland semistable objects on $A$.
The construction is completely parallel to above and we will be brief.

Let $\sigma \in \Stab^{\circ}(A)$ be a Bridgeland stability condition which satisfies the 
full support property, 
and let $v \in \Gamma$.
We consider the moduli stack
\begin{align}\label{stack:B}
\mM_{\sigma}(v, \phi)
\end{align}
of $\sigma$-semistable objects $E \in D^b(A)$ with $\ch(E)=v$ and phase $\phi \in \BR$. 
By~\cite{PiYT}, the stack (\ref{stack:B}) is an algebraic stack of finite type.
Moreover by \cite{AHLH} we have that 
the stack (\ref{stack:B}) admits a good moduli space
\begin{align*}
p \colon  \mM_{\sigma}(v, \phi) \to M_{\sigma}(v, \phi)
\end{align*}
for a separated algebraic space $M_{\sigma}(v, \phi)$
of finite type. We set
\begin{align*}
p \colon 
\mM_{\sigma}(\phi):=\coprod_{v} \mM_{\sigma}(v, \phi)
\to M_{\sigma}(\phi) :=\coprod_{v} M_{\sigma}(v, \phi)
\end{align*}

By the argument in \cite[Proof of Theorem 5.6]{PiYT}
we may assume that
%If $\sigma$ is not defined over $\BQ$, by the wall-chamber structure
%one can perturb $\sigma$ and may assume that the central charge of $\sigma$ is defined over $\mathbb{Q}$,
%see~. This yields the definition of reduced Donaldson--Thomas invariants also in general.
%
%First let us assume that the central charge of 
$\sigma$ is defined over $\mathbb{Q}$. 
Let $\phi \in \mathbb{R}$ be fixed, 
and let $\pP(\phi)$ be the category of $\sigma$-semistable 
objects with phase $\phi$. 
Then there exist a noetherian heart 
\begin{align*}
\aA =\pP((\psi-1, \psi])\subset D^b(X)
\end{align*}
for some $\psi \in \mathbb{R}$ with $\phi \in (\psi-1, \psi]$.
The heart $\aA$ 
is closed under the $\A$-action, 
since the $\A$ action leaves all the Chern characters invariant.
Then by~\cite[Corollary~4.21]{PiYT} the stack $\oO bj(\aA)$ of objects in $\aA$
is an algebraic stack locally of finite type 
with $\A$-action. As in Section~\ref{subsec:hall}
consider the $\A$-equivariant motivic Hall algebra with respect to the heart $\CA$,
\begin{align*}
H^{\A}(\aA)
=K_0^{\A}(\mathrm{St}/\oO bj(\aA)).
\end{align*}
Then similarly to Section~\ref{subsec:Gieseker}, we have the subalgebra
\begin{align*}
H^{\A}(\aA, \phi) := K_0^{\A}(\mathrm{St}/\mM_{\sigma}(\phi)) \subset
H^{\A}(\aA).
\end{align*}
We define $H_{\rm{reg}}^{\A}(\aA, \phi)$, $H_{\rm{sc}}^{\A}(\aA, \phi)$ and the integration map
\begin{align}\label{integrate:Bstab}
\CI^\A \colon H_{\rm{sc}}^{\A}(\aA, \phi) \stackrel{p_{\ast}^{\A}}{\to}
\mathrm{Constr}^{\A}(M_{\sigma}(\phi)) \stackrel{J}{\to} C^{\A}(X)
\end{align}
as in Section~\ref{Subsection_Equivariant_integration_map}.
The stack (\ref{stack:B}) defines the element
\begin{align*}
\delta_{\sigma}(v, \phi)\cneq [\mM_{\sigma}(v, \phi) \subset \mM_{\sigma}(\phi)]
\in H^{\A}(\aA, \phi).
\end{align*}
Using the result of Joyce, the logarithm
\begin{align}\label{def:log}
\epsilon_{\sigma}(v, \phi) \cneq
\sum_{l\ge 1, v_1+\cdots+v_l=v}
\frac{(-1)^{l-1}}{l}
\delta_{\sigma}(v_1, \phi) \ast \cdots \ast \delta_{\sigma}(v_l, \phi)
\end{align}
yields the regular element $(\mathbb{L}-1)\epsilon_{\sigma}(v, \phi)$ which in turn defines
\begin{align*}
\overline{\epsilon}_{\sigma}(v, \phi) \cneq [(\mathbb{L}-1)\epsilon_{\sigma}(v, \phi)]
\in H_{\rm{sc}}^{\A}(\aA, \phi). 
\end{align*}
We define the $\A$-reduced Donaldson--Thomas invariant
$\DTb_{\sigma}(v, \phi) \in \mathbb{Q}[\A]$ by
\begin{align*}
\CI^{\A}(\overline{\epsilon}_{\sigma}(v, \phi))=\DTb_{\sigma}(v, \phi) \cdot c_v. 
\end{align*}

Since $\DTb_{\sigma}(v, \phi)=\DTb_{\sigma}(v, \phi+1)$ the following convention makes sense.

\begin{defi}
For all $\sigma=(Z, \aA) \in \Stab(A)$ and $v \in \Gamma$ define 
\begin{align*}
\DTb_{\sigma}(v) \cneq \left\{ 
\begin{array}{ll}
\DTb_{\sigma}(v, \phi), & \mbox{ if }Z(v) \in \mathbb{R}_{>0} e^{\pi i \phi} \mbox{ for some }
\phi \in \mathbb{R} \\
0, & \mbox{ if }Z(v)=0. 
\end{array}
\right. 
\end{align*}
\end{defi}

For any connected abelian subvariety $B \subset \A$, we define $\DT_{\sigma}(v)_B \in \mathbb{Q}$ by 
\begin{align*}
\DTb_{\sigma}(v)=\sum_{B \subset \A}\DT_{\sigma}(v)_B \cdot \epsilon_B. 
\end{align*}
As before we usually write $\DT_{\sigma}(v) \cneq \DT_{\sigma}(v)_{B=(0, 0)}$. 

We have the following comparision result.
\begin{prop}\label{prop:DTH}
For any $v \in \Gamma$ and ample divisor $H$ on $A$, there exists a $\sigma \in \Stab^{\circ}(A)$ such that 
$\DTb_{\sigma}(v)=\DTb_{H}(v)$. 
\end{prop}
\begin{proof} This follows from Proposition~\ref{lem:DTchamber} and since $\DTb_{\sigma}(v)=\DTb_{\sigma}(-v)$ by convention. \end{proof}

\subsection{Comparison under change of stability conditions}\label{subsec:compare}
The integration map $\CI^{\A}$ defined in Section~\ref{Section_reduced_DT_Bridgeland}
depended on a choice of stability condition.
We check the definition is well-behaved under
change of stability condition.

Consider a pair of stability conditions
\[ \sigma=(Z, \aA), \sigma'=(Z', \aA') \in \Stab^{\circ}(A). \]
Let $v \in \Gamma$ be fixed and let $\phi, \phi' \in \mathbb{R}$ be phases such that $Z(v) \in \mathbb{R}_{>0}e^{\pi i\phi}$ and
$Z'(v) \in \mathbb{R}_{>0}e^{\pi i\phi'}$.
We assume that there is an open embedding of stacks
\begin{equation}
\iota \colon \mM_{\sigma'}(v, \phi') \subset\mM_{\sigma}(v, \phi). \label{Incluopen}
\end{equation}
%Let us fix $\sigma=(Z, \aA) \in \Stab^{\circ}(A)$, 
%$v \in \Gamma$ and $\phi \in \mathbb{R}$
%satisfying $Z(v) \in \mathbb{R}_{>0}e^{\pi i\phi}$. 
%Let $\sigma'=(Z', \aA') \in \Stab^{\circ}(A)$
%be another stability condition 
%such that 
%$Z'(v) \in \mathbb{R}_{>0}e^{\pi i\phi'}$
%and we have the embedding of stacks
%\begin{align*}
%\iota \colon 
%\mM_{\sigma'}(v, \phi') \subset\mM_{\sigma}(v, \phi).
%\end{align*}

The inclusion $\iota$ induces the map
\begin{align*}
\iota_{\ast} \colon K_0^{\A}(\mathrm{St}/\mM_{\sigma'}(v, \phi'))
\to K_0^{\A}(\mathrm{St}/\mM_{\sigma}(v, \phi)).
\end{align*}
Recall also from Section~\ref{Section_reduced_DT_Bridgeland} the integration maps
\[ \CI^{\A} : K_{0, \mathrm{reg}}^{\A}(\mathrm{St}/\mM_{\sigma}(v, \phi)) \to \BQ[\A]c_v, \ \ \CI^{\prime \A} : 
K_{0, \mathrm{reg}}^{\A}(\mathrm{St}/\mM_{\sigma'}(v, \phi')) \to \BQ[\A]c_v. \]
obtained from the stability conditions $\sigma$ and $\sigma'$ respectively.
Here $\mathrm{reg}$ stands for regular elements. 
%The following lemma shows that these integration maps are compatible with respect to $\iota_{\ast}$.
%In other words the integration map does not depend on a choice of stability condition.

\begin{prop} \label{prop:compare} We have $\CI^{\A} = \CI^{\prime \A} \circ \iota_{\ast}$. 
In particular,
\begin{align*}
\iI^{\A}\left((\mathbb{L}-1)\iota_{\ast}\epsilon_{\sigma'}(v, \phi')\right)
=\DT_{\sigma'}(v, \phi') \cdot c_v.
\end{align*}
\end{prop}
\begin{proof}
By the universal property of good moduli spaces, 
we have the commutative diagram
\[\label{dia:goodmoduli}
\begin{tikzcd}
\mM_{\sigma'}(v, \phi') \arrow{r}{\iota} \arrow{d}{p'} & \mM_{\sigma}(v, \phi) \arrow{d}{p}\\
M_{\sigma'}(v, \phi') \arrow{r}{\tau}  & M_{\sigma}(v, \phi),
\end{tikzcd}
\]
where the left arrow is the good moduli space for 
$\mM_{\sigma'}(v, \phi')$. 
Then it is enough to show that the following 
diagram is commutative
\[
\begin{tikzcd}
K_{0, \mathrm{reg}}^{\A}(\mathrm{St}/\mM_{\sigma'}(v, \phi')) \arrow{r}{\iota_{\ast}} \arrow{d}{p_{\ast}^{'\A}} & 
K_{0, \mathrm{reg}}^{\A}(\mathrm{St}/\mM_{\sigma}(v, \phi)) \arrow{d}{p_{\ast}^{\A}}\\
\mathrm{Constr}^{\A}(M_{\sigma'}(v, \phi')) \arrow{r}{\tau_{\ast}} \arrow{d}{J'} & 
\mathrm{Constr}^{\A}(M_{\sigma}(v, \phi)) \arrow{d}{J} \\
C^{\A}(X) \arrow{r}{\id} & C^{\A}(X).
\end{tikzcd}
\]
Here 
$J$, $J'$ are defined as in (\ref{def:J}), 
and $\tau_{\ast}$ is defined as follows:
for any $\A$-invariant subspace $Z \subset M_{\sigma'}(v, \phi')$
and $x \in M_{\sigma}(v, \phi)$ let
$\tau_{\ast}(1_Z)(x)=e(\tau^{-1}(x) \cap Z)$. 

The upper diagram is commutative since 
both $p_{\ast}^{\A} \circ \iota_{\ast}$
and $\tau_{\ast} \circ p_{\ast}^{'\A}$
compute the Behrend function weighted Euler numbers of fibers to the map to $M_{\sigma}(v, \phi)$,
and the Behrend weights agree since \eqref{Incluopen} is an open embedding.
To show that the lower diagram is commutative, 
by the definition of equivariant Euler number it 
is enough to show that the map $\tau$ preserves the 
connected component of the stabilizer 
groups of $\A$-actions, i.e. for any 
$x \in M_{\sigma'}(v, \phi')$, the induced map 
$\Stab(x)^{\circ} \to \Stab(\tau(x))^{\circ}$ is an isomorphism.
By the diagram on good moduli spaces and since the open immersion $\iota$
preserves the connected component of the stabilizer group,
it is enough to show the following lemma. 
\end{proof}
\begin{lem}
For any $x \in M_{\sigma}(v, \phi)$, 
the connected component of the stablizer
$B=\Stab(x)^{\circ}$ acts trivially on 
the geometric points of $p^{-1}(x) \subset \mM_{\sigma}(v, \phi)$.
\end{lem}
\begin{proof}
For a fixed 
$x \in M_{\sigma}(v, \phi)$, there is a finite number of 
$B$-fixed
$\sigma$-stable objects $E_1, \ldots, E_n$
with phase $\phi$
such that any point in $p^{-1}(x)$
corresponds to iterated extensions of $E_1, \ldots, E_n$.
By the induction argument, 
it is enough to prove the following: for any
$\sigma$-semistable objects $P$, $Q$ fixed by $B$ and with phase $\phi$, and for any extension 
\[ 0 \to P \to R \to Q \to 0 \]
we have $g(R) \cong R$ for any $g \in B$. 

The last claim is proved as follows. For $g \in B$, let 
\begin{align*}
a_g \colon g(P) \stackrel{\cong}{\to} P, \ b_g \colon g(Q) \stackrel{\cong}{\to} Q
\end{align*}
be isomorphisms. For $u \in \Ext^1(Q, P)$, we set
\begin{align*}g(u)' = b_g \circ g(u) \circ a_g^{-1} \in \Ext^1(Q, P),
\end{align*}
where $g(u) \in \Ext^1(g(Q), g(P))$ is the extension induced by
the $B$-action. The assignment $g \mapsto (u \mapsto g(u)')$
is well-defined up to choices of $a_g$, $b_g$, so defines a map
\begin{align*}
B \to \GL(\Ext^1(Q, P))/\Aut(Q) \times \Aut(P).
\end{align*}
The target is an affine variety and $B$ is an abelian variety, so the image must be an identity. 
This gives the proof of the above claim. 
\end{proof}

\subsection{Reduced DT invariants for semihomogeneous sheaves}
Recall the subset of semihomogeneous sheaves $\cC \subset \Gamma$ defined in \eqref{def:C0}. 
Since the stabilizer $B \subset \A$ of every non-zero coherent sheaf on $A$ is at most 3-dimensional \cite[Proposition~4.5]{Mu4},
and the sheaf is semihomogeneous if and only if $\dim(B) = 3$, we have the following.
\begin{lem}\label{lem:B} Let $v \in \Gamma$ and let $B \subset \A$ be a connected abelian subvariety.
\begin{enumerate}
\item[(a)] If $\dim B>3$, then $\DT_{H}(v)_B=0$.
\item[(b)] If $\dim B=3$ and $\DT_{H}(v)_B\neq 0$ then $v \in \cC$. 
\end{enumerate}
\end{lem}
We have the following generalization of Lemma~\ref{lem:B}.
\begin{lem}\label{lem:B2}
Let $\sigma \in \Stab^{\circ}(A)$. Let $v \in \Gamma$ and let $B \subset \A$ be connected.
\begin{enumerate}
\item[(a)] If $\dim B > 3$, then $\DT_{\sigma}(v)_B=0$.
\item[(b)] If $\dim B = 3$ and $\DT_{\sigma}(v)_B\neq 0$, then $v \in \cC$. 
\end{enumerate}
\end{lem}
The above lemma follows immediately from the following: 
\begin{lem} \label{Lemma_24t6rgs}
For every $E \in D^b(A)$, 
let $\Xi(E) \subset \A$ be as in (\ref{def:Phi(E)}).
Then 
we have $\dim \Xi(E) \le 3$.
If $\dim \Xi(E)=3$, we have 
$\ch(E) \in \cC$. 
\end{lem}
\begin{proof}
For every $E \in D^b(A)$ with $F_i=\hH^i(E)$, we have
\begin{align*}
\Xi(E) \subset \bigcap_{i\in \mathbb{Z}} \Xi(F_i)
\end{align*}
and $\dim \Xi(F_i) \le 3$ 
by~\cite[Proposition~4.5]{Mu4}. 
Suppose that $\dim \Xi(E)=3$. 
Then 
$\dim \Xi(E)=\dim \Xi(F_i)=3$ for any $i \in \mathbb{Z}$
such that $F_i \neq 0$. 
In particular each $F_i$ is a semihomogeneous sheaf. 
It is enough to show that 
$\ch(F_i)$ is proportional to $\ch(F_j)$ 
for each pair $(i, j)$. 

First suppose that each $F_i$ is a vector bundle. 
Then $\ch(F_i)$ is written as $r(F_i) e^{c_1(F_i)/r(F_i)}$. 
Let $\Xi^{\circ}(-) \subset \Xi(-)$ be the connected component 
which contains $(0, 0)$. 
Then 
we have $\Xi^{\circ}(E)=\Xi^{\circ}(F_i)$
for any $i \in \mathbb{Z}$ with $F_i \neq 0$. 
By~\cite[Theorem~4.9~(3)]{Mu4}, 
the subabelian variety $\Xi^{\circ}(F_i) \subset \A$ 
determines $c_1(F_i)/r(F_i)$.
Therefore for each $(i, j)$, 
we have $c_1(F_i)/r(F_i)=c_1(F_j)/r(F_j)$, and 
$\ch(F_i)$, $\ch(F_j)$ are proportional.

When $F_i$ is not a vector bundle, we can apply 
a Fourier-Mukai transform $\Phi_{\pP}^{A \to \widehat{A}} 
\circ \otimes \oO_A(mH)$
for $m\gg 0$ 
and use Theorem~\ref{thm:Orlov} below 
to reduce to the case that every $F_i$ is a vector bundle. 
\end{proof}

%Here we have used the following Theorem by Orlov~\cite{OrA}:
\begin{thm}\emph{(\cite{OrA})}\label{thm:Orlov}
There is a map
\begin{align}\label{Or:isom}
\Aut(D^b(A)) \to \Aut(A \times \widehat{A}), \ g \mapsto g_{\ast}
\end{align}
such that $g_{\ast}\Xi(E)=\Xi(g(E))$ for any $E \in D^b(A)$. 
\end{thm}

\subsection{Independence of stability conditions}\label{subsec:independence}
We show the absence of walls in good cases.

\begin{thm}\label{prop:wall1}
Suppose that $v \in \Gamma$ is not written as
$\gamma_1+\gamma_2$ for some $\gamma_i \in \cC$ with 
$\chi(\gamma_1, \gamma_2) \neq 0$.
Then for any $\sigma, \sigma' \in \Stab^{\circ}(A)$ we have
\[
\DTb_{\sigma}(v) = \DTb_{\sigma'}(v)
\]
\end{thm}

\begin{proof}
We prove that for any $\sigma=(Z, \aA) \in \Stab^{\circ}(A)$ there is an open 
neighborhood $\sigma \in U \subset \Stab^{\circ}(A)$
with $\DTb_{\sigma}(v)=\DTb_{\sigma'}(v)$
for any $\sigma' \in U$. 

Suppose that $Z(v)=0$. 
Then there is no $\sigma$-semistable object
$E$ with $\ch(E)=v$. 
By the wall and chamber structure on $\Stab^{\circ}(A)$, 
there is an open neighborhood $\sigma \in U \subset \Stab^{\circ}(A)$
such that for any $\sigma' \in U$
there is no $\sigma'$-semistable object $E$ with $\ch(E)=v$. 
It follows $\DTb_{\sigma}(v)=\DTb_{\sigma'}(v)=0$. 

Hence we may assume that $Z(v) \neq 0$.
Let $\phi \in \mathbb{R}$ such that 
$Z(v) \in \mathbb{R}_{>0} e^{\pi i \phi}$. 
For an open neighborhood $\sigma \in U \subset \Stab^{\circ}(A)$, 
we take $\sigma' =(Z', \aA') \in U$.
For $\psi \in \mathbb{R}$, let 
$\pP(\psi)$, $\pP'(\psi)$ be
the $\sigma$, $\sigma'$-semistable 
objects with phase $\psi$. 
By shrinking $U$ and applying a $\mathbb{C}$-action on $\Stab^{\circ}(A)$ if necessary, 
we can assume that 
\begin{align*}
\pP(\phi) \subset \pP'((\phi-\varepsilon, \phi+\varepsilon)) \subset \aA
\end{align*}
for some $0< \varepsilon \ll 1$, where the right hand side is the extension 
closure of objects in $\pP'(\psi)$
with $\psi \in (\phi-\varepsilon, \phi+\varepsilon)$. 

We then have the following identity in $H^{\A}(\aA, \phi)$,
\begin{align}\label{wc:delta}
\delta_{\sigma}(v, \phi)=\sum_{\begin{subarray}{c}
l\ge 1, \gamma_1+\cdots+\gamma_l=v \\
Z(\gamma_i) \in \mathbb{R}_{>0} e^{\pi i \phi}, \\ 
Z'(\gamma_i) \in \mathbb{R}_{>0} e^{\pi i \phi_i}, \\
\phi_1>\cdots>\phi_l,  \phi_i \in (\phi-\varepsilon, \phi+\varepsilon).
\end{subarray}
}
\delta_{\sigma'}(\gamma_1, \phi_1) \ast \cdots \ast \delta_{\sigma'}(\gamma_l, \phi_l). 
\end{align}
Here 
$\mM_{\sigma'}(\gamma_i, \phi_i)$
is an open substack of $\mM_{\sigma}(\gamma_i, \phi)$,
and\footnote{More precisely $\delta_{\sigma'}(\gamma_i, \phi_i)$
is the push-forward under 
the open embedding $\mM_{\sigma'}(\gamma_i, \phi_i) \subset \mM_{\sigma}(\gamma_i, \phi)$
as in Section~\ref{subsec:compare}, and 
we have omitted the notation of the push-forward.} 
\begin{align*}
\delta_{\sigma'}(\gamma_i, \phi_i)
=[\mM_{\sigma'}(\gamma_i, \phi_i) \subset \mM_{\sigma}(\gamma_i, \phi)]
\in H^{\A}_{\mathrm{sc}}(\aA, \phi).
\end{align*}
By substituting (\ref{def:log}) and multiplying $(\mathbb{L}-1)$, 
we obtain an identity in $H_{\rm{sc}}^{\A}(\aA, \phi)$ of the form
\begin{multline} \label{dgsdfgsf}
\overline{\epsilon}_{\sigma}(v, \phi) = 
\overline{\epsilon}_{\sigma'}(v, \phi')
+ \sum_{\gamma_1+\gamma_2=v}a_{\gamma_1, \gamma_2}\{ \overline{\epsilon}_{\sigma'}(\gamma_1, \phi_1), \overline{\epsilon}_{\sigma'}(\gamma_2, \phi_2)\} \\
+ \sum_{\gamma_1+\gamma_2+\gamma_3=v} a_{\gamma_1, \gamma_2, \gamma_3} \{ \{ \overline{\epsilon}_{\sigma'}(\gamma_1, \phi_1), \overline{\epsilon}_{\sigma'}(\gamma_2, \phi_2)\}, 
\overline{\epsilon}_{\sigma'}(\gamma_3, \phi_3)\}+\cdots 
\end{multline}
for some $a_{\gamma_1, \cdots, \gamma_l} \in \mathbb{Q}$. 

We apply the equivariant integration map $\CI^{\A}$ to \eqref{dgsdfgsf}.
%For every $i$, 
If we write
\[ \CI^{\A}(\overline{\epsilon}_{\sigma'}(\gamma_i, \phi_i)) = \sum_{k} b_{k} \epsilon_{B_{k}} c_{\gamma_i} \]
for some $b_{k} \in \BQ$ and $B_{k} \subset \A$, then by Lemma~\ref{lem:B2} the $B_{k}$ are of codimension $\geq 3$.
By the definition of $\mathbb{Q}[\A]$ 
it follows that only linear or quadratic terms in the $\overline{\epsilon}_{\sigma'}(\gamma_i, \phi_i)$ contribute when applying $\CI^{\A}$ to \eqref{dgsdfgsf}.
Moreover using Proposition~\ref{prop:compare}, the contribution of the quadratic term is
\begin{align*}
&\sum_{\gamma_1+\gamma_2=v}\iI^{\A}(a_{\gamma_1, \gamma_2}\{  \overline{\epsilon}_{\sigma'}(\gamma_1, \phi_1), \overline{\epsilon}_{\sigma'}(\gamma_2, \phi_2) \}) \\
&=\sum_{\gamma_1+\gamma_2=v}a_{\gamma_1, \gamma_2}\sum_{\begin{subarray}{c}
B_i \subset \A, \\
i=1, 2
\end{subarray}}
(-1)^{\chi(\gamma_1, \gamma_2)}\chi(\gamma_1, \gamma_2)
\DT_{\sigma'}(\gamma_1)_{B_1} \DT_{\sigma'}(\gamma_2)_{B_2}
\left| B_1 \cap B_2\right|
\epsilon_{(0, 0)}. 
\end{align*}
where the sum is over connected abelian subvarieties $B_i \subset \A$ of dimension $3$ such that $B_1$ and $B_2$ are transversal. 
By Lemma~\ref{lem:B2} and its proof, the above is 
non-zero only if $v=\gamma_1+\gamma_2$ 
with
$\gamma_i \in \cC$ such that
$\chi(\gamma_1, \gamma_2) \neq 0$. 
\end{proof}

The proof of Theorem~\ref{prop:wall1} also shows the following: 
\begin{cor}\label{cor:B}
For every $v \in \Gamma$ and every positive-dimensional connected abelian subvariety 
$B \subset \A$ we have 
\[ \DT_{\sigma}(v)_{B}=\DT_{H}(v)_{B} \]
for all $\sigma \in \Stab^{\circ}(A)$ and ample divisors $H$.
% Moreover for any $g \in \Aut(D^b(A))$, we have 
% \[ \DT_{H, g_{\ast}B}(g_{\ast}v)=\DT_{H, B}(v), \]
% where $g_{\ast} \in \Aut(A \times \widehat{A})$ is the morphism (\ref{Or:isom}). 
\end{cor}

Combining Theorem~\ref{prop:wall1} and Proposition~\ref{prop:DTH} yields the following.
\begin{cor} \label{cor:DTH=DTsigma}Under the assumptions of Proposition~\ref{prop:wall1},
\begin{equation}\label{DT:hs}
\DTb_{\sigma}(v)=\DTb_H(v). 
\end{equation}
for all $\sigma \in \Stab^{\circ}(A)$ and ample divisors $H$.
In particular, $\DTb_{\sigma}(v)$, $\DTb_{H}(v)$ are independent of $\sigma$ and $H$. 
\end{cor}

We have now all ingredients for the proof of Theorem~\ref{intro:first}.
\begin{proof}[Proof of Theorem~\ref{intro:first}]
Suppose that $v \in \Gamma$ is not written as $\gamma_1+\gamma_2$
for some $\gamma_i \in \cC$ with $\chi(\gamma_1, \gamma_2) \neq 0$.
Let $g \in \Aut D^b(A)$ be a derived autoequivalence and let $\sigma \in \Stab^\circ(A)$
be a stability condition which lies in the Gieseker chamber with respect to $v$.
By Corollary~\ref{cor:DTH=DTsigma} and an application of $g$ we have
\[
\DTb_H(v)=\DTb_{\sigma}(v)=\DTb_{g_{\ast}\sigma}(g_{\ast}v)
\]
By Theorem~\ref{thm_full_support} the stability condition $g_{\ast} \sigma$ lies in the component $\Stab^{\circ}(A)$.
Hence again by Corollary~\ref{cor:DTH=DTsigma},
\[ \DTb_{g_{\ast}\sigma}(g_{\ast}v) = \DTb_H(g_{\ast}v). \qedhere \]
\end{proof}

\subsection{The discriminant}
From Appendix~\ref{Appendix_Spin_representations} recall the discriminant
\[ \Delta : H^{2\ast}(A, \BZ) \to \BZ. \]
%normalized such that $\Delta(1 + \pt) = -1$, where $\pt \in H^6(A, \BZ)$ is the class of a point.
By construction $\Delta$ is invariant under all derived autoequivalences of $A$.
The following lemma directly implies Proposition~\ref{intro:prop_Delta}.
\begin{lemma} \label{Lemma_Discrri} Let $v \in \Gamma$.
\begin{enumerate}
\item If $\Delta(v) > 0$, then $v$ is not of the form $\gamma_1 + \gamma_2$ with $\gamma_i \in \CC$.
\item If $\Delta(v) = 0$ and $v = \gamma_1 + \gamma_2$ with $\gamma_i \in \CC$, then $\chi(\gamma_1, \gamma_2) = 0$.
\end{enumerate}
Hence, if $\Delta(v) \geq 0$ then $v$ satisfies the assumption of Proposition \ref{prop:wall1}.
\end{lemma}

\begin{proof}
By Theorem~\ref{thm:Orlov} the set $\CC$ is preserved by derived autoequivalences.
Therefore as in the proof of Lemma \ref{Lemma_24t6rgs} we may assume $\gamma_i = r_i e^{c_1/r_i}$
for some $c_1 \in H^2(A)$ and $r_i \in \BZ$.
The claims follow now from Theorem~\ref{Thm_Discriminant}.
\end{proof}

\subsection{Reduced DT invariants for semihomogeneous sheaves II}
We calculate the Donaldson-Thommas invariants of semihomogeneous sheaves.
\begin{lemma} \label{cor:C} Let $v \in \CC$. Then
\[ \DTb_{\sigma}(v) = \left(\sum_{k\ge 1, k|v}\frac{1}{k^2} \right) \epsilon_{B} \]
for some three-dimensional $B \subset \A$ determined by $v$.
\end{lemma}

\begin{proof}
By~\cite[Proposition~4.11]{Mu4}, there exists another abelian variety $A'$ and a 
equivalence $F \colon D^b(A) \stackrel{\sim}{\to} D^b(A')$
such that 
\[ F_{\ast}v=(0, 0, 0, r) \]
for some $r\ge 1$. Hence 
$\Delta(v) = \Delta(F_{\ast}v) = 0$. 
Using Lemma~\ref{Lemma_Discrri} and Corollary~\ref{cor:B} we conclude
\[
\DTb_{\sigma}(v) = \DTb_{F_{\ast} \sigma}(0,0,0,r) = \DTb_H(0,0,0,r)
=
\left(\sum_{k\ge 1, k|r} \frac{1}{k^2} \right) \epsilon_{\{ 0 \} \times \widehat{A}}. 
\]
where the last equality is \cite[Proposition~6]{OS}.
\end{proof}

\section{Principally polarized abelian threefolds} \label{section:ppav}
\subsection{Setup}
%In this section, we prove Theorem~\ref{intro:main}. 
Let $(A, H)$ be a principally polarized abelian 3-fold of Picard rank $\rho(A)=1$.
We identify $A$ with its dual $\widehat{A}$ via the isomorphism 
\begin{align*}
A \stackrel{\cong}{\to} \widehat{A}, \ 
x \mapsto T_x^{\ast}\oO_A(H) \otimes \oO_A(-H). 
\end{align*}

We also identify elements in $\Gamma$ with vectors $(v_0, v_1, v_2, v_3) \in \BZ^4$ via the isomorphism
\begin{align}\label{isom:Zgamma} 
\mathbb{Z}^4 \stackrel{\cong}{\to} \Gamma, \ \ 
(v_0, v_1, v_2, v_3) \mapsto 
(v_0, v_1[H], v_2[H^2/2], v_3[H^3/6]). 
\end{align}
%We write an element of $\Gamma$ as a vector $(v_0, v_1, v_2, v_3)$
%in $\mathbb{Z}^4$
%by the above isomorphism (\ref{isom:Zgamma}).
Under this identification the Euler pairing $\chi$ on $\Gamma$ is
\begin{align}\label{Euler}
\chi\big((v_0, v_1, v_2, v_3), (v_0', v_1', v_2', v_3')\big)
=v_0 v_3'-3v_1 v_2'+3v_2 v_1'-v_3 v_0'. 
\end{align}
The discriminant defined in Appendix~\ref{Appendix_Spin_representations} takes the form
\begin{align}\label{Delta}
\Delta(v_0, v_1, v_2, v_3) = -4(v_0 v_2^3+v_1^3 v_3)-v_0^2 v_3^2+3v_1^2 v_2^2 +6v_0 v_1 v_2 v_3. 
\end{align}

\subsection{Action of autoequivalences on cohomology} \label{Subsection_Action_of_auto-equvalences_on_coh}
Recall that the group $\SL_2(\mathbb{Z})$ is generated by the elements
\begin{align*}
T=
\left( \begin{array}{cc}
1 & 1 \\
0 & 1 
\end{array}  \right), \ 
S=
\left( \begin{array}{cc}
0 & -1 \\
1 & 0 
\end{array}  \right)
\end{align*}
with relations 
$S^2=(TS)^3$ and $S^4=1$. 
Let $\widetilde{\SL}_2(\mathbb{Z})$ be the 
group generated by 
$\widetilde{S}$, $\widetilde{T}$ 
with the relation 
$\widetilde{S}^2=(\widetilde{T} \widetilde{S})^3$. 
There is an exact sequence of groups
\begin{align}\label{SL:ext}
1 \to \mathbb{Z} \stackrel{i}{\to} \widetilde{\SL}_2(\mathbb{Z}) 
\stackrel{j}{\to} 
\SL_2(\mathbb{Z}) \to 1
\end{align}
where the map $i$ sends $1$ to 
$\widetilde{S}^4$ 
and $j$ sends $\widetilde{S}$, $\widetilde{T}$
to $S$, $T$ respectively. 

By a result of Mukai~\cite{Mu1} there is a group homomorphism 
\begin{align}\label{hom:SL}
\widetilde{\SL}_2(\mathbb{Z}) \to \Aut(D^b(A))
\end{align}
sending $\widetilde{S}$, $\widetilde{T}$
to $\Phi_{\pP}:=\Phi_{\pP}^{A \to A}$ and $\otimes \oO_A(H)$ respectively. 
Because $\Phi_{\pP}^4=[-6]$ acts on $\Gamma$ trivially, 
(\ref{hom:SL}) descends to a homomorhism
%representation of $\SL_2(\mathbb{Z})$: 
\begin{equation} \label{dsdghaa}
\SL_2(\mathbb{Z}) \to \Aut(\Gamma). 
\end{equation}
In terms of the generators $(S, T)$ this representation is given by
\begin{align}\label{intro:ST}
T
\mapsto \left(
\begin{array}{cccc}
1 & 0 & 0 & 0 \\
1 & 1 & 0 & 0 \\
1 & 2 & 1 & 0 \\
1 & 3 & 3 & 1 \\
\end{array}
\right), \quad
S
\mapsto 
\left(
\begin{array}{cccc}
0 & 0 & 0 & 1 \\
0 & 0 & -1 & 0 \\
0 & 1 & 0 & 0 \\
-1 & 0 & 0 & 0 \\
\end{array}
\right). 
\end{align}
%in (\ref{intro:ST}). 
For $g \in \SL_2(\mathbb{Z})$, we let 
$g_{\ast} \in \Aut(\Gamma)$ denote the induced 
isomorphism.

We can interpret the action \eqref{dsdghaa}
as a $\SL_2(\mathbb{Z})$-action on 
two variable homogeneous polynomials as follows. 
Identify elements in $\Gamma$ with certain cubic homogeneous polynomials in two variables via the map
\begin{align}\label{poly}
(v_0, v_1, v_2, v_3) \mapsto 
v_0x^3+3v_1x^2 y+3v_2 xy^2+v_3 y^3. 
\end{align}
The group $\SL_2(\mathbb{Z})$ acts on the homogeneous cubic polynomials in $(x, y)$ by 
the transformation 
\begin{align}\label{trans}
g_{\ast} \colon (x, y) \mapsto (dx+by, cx+ay)
\end{align}
where $g = \binom{a\ b}{c\ d} \in \SL_2(\BZ)$.
This action coincides with $g_{\ast} \in \Aut(\Gamma)$ under the identification \eqref{poly}.\footnote{
The identification \eqref{poly} also gives motivation to call $\Delta(v)$ the discriminant, since it coincides
with the discriminant of the cubic polynomial on the right hand side of (\ref{poly}).}

\subsection{Action of autoequivalences on stability conditions}
We next describe the action of 
$\widetilde{\SL}_2(\mathbb{Z})$ on $\Stab^{\circ}(A)$.
%as defined by \eqref{Action_on_stab}. 
%
Let $\BH \subset \mathbb{C}$ be the upper half plane. 
By Theorem~\ref{thm:BMS}, we have the embedding
\begin{align}\label{emb:H}
\BH \to \Stab^{\circ}(A), \ 
\tau=\beta+i\alpha \mapsto \sigma_{\tau} :=\sigma_{\alpha, \beta}
=\sigma_{\alpha H, \beta H}.
\end{align}
The group $\mathrm{SL}_2(\BZ)$ acts on the upper half plane $\BH$ by
\[ \tau \mapsto g \cdot \tau = \frac{a \tau + b}{c \tau + d} \]
for all $\binom{a\ b}{c\ d} \in \mathrm{SL}_2(\BZ)$ and $\tau \in \BH$.
The following Lemma shows that, modulo the $\widetilde{\GL}_2^{+}(\mathbb{R})$ action, these two actions coincide.
\begin{lem}\label{lem:modular}
For any $g \in \SL_2(\mathbb{Z})$ with lift $\widetilde{g} \in \widetilde{\SL}_2(\mathbb{Z})$ and for any $\tau \in \BH$,
there exists a unique $\xi \in \mathbb{C} \subset \widetilde{\GL}_2^{+}(\mathbb{R})$ such that
\[ \widetilde{g}_{\ast} \sigma_{\tau} = \sigma_{g \tau} \cdot \xi. \]
\end{lem}

\begin{proof}
By Theorem~\ref{thm:BMS} and since $\Aut D^b(A)$ preserves the main component of the stability manifold, we have 
\[ \widetilde{g}_{\ast}\sigma_{\tau} = \sigma' \cdot \xi \]
for some $\sigma' \in \mathfrak{B}$ and $\xi \in \widetilde{\GL}_2^{+}(\mathbb{R})$. 
Therefore it is enough to show that the central charge of $\widetilde{g}_{\ast}\sigma_{\tau}$ is of the desired form.

By \eqref{Z_omegaB} the central charge of $\sigma_{\tau}$ is written as  
\begin{align*}
Z_{\tau}(v)=-\chi(e^{\tau H}, v)
\end{align*}
for all $v \in \Gamma$. Hence the central charge of $\widetilde{g}_{\ast}\sigma_{\alpha, \beta}$ is
\begin{equation} \label{5ydfgdf}
Z_{\tau}(g_{\ast}^{-1}v) =-\chi(e^{\tau H}, g_{\ast}^{-1}v) \\
= -\chi(g_{\ast}e^{ \tau H}, v).
\end{equation}
Under the correspondence (\ref{poly}) we have $e^{\tau H} = (x+\tau y)^3$ which implies
\[ g_{\ast} e^{\tau H} = (c \tau + d)^3 (x + (g\tau)y)^3 = (c \tau + d)^3 e^{(g \tau)H}. \]
Inserting back into \eqref{5ydfgdf} the Lemma follows.
\end{proof}

\subsection{Wall and chamber structure}
We consider classes $v \in \Gamma$ which can be written as $\gamma_1+\gamma_2$ 
for some $\gamma_i \in \cC$ such that $\chi(\gamma_1, \gamma_2) \neq 0$. 
Since $(A,H)$ is principally polarized we have
\begin{align*}
\cC =\{r(p^3, p^2q, pq^2, q^3) : 
(p, q, r) \in \mathbb{Z}^3, r\neq 0, \mathrm{gcd}(p, q)=1\}.  
\end{align*}
Hence $v$ can be written as 
\begin{align}\label{v123}
v=\gamma_1+\gamma_2, \ \gamma_i=r_i(p_i^3, p_i^2 q_i, p_i q_i^2, q_i^3)
\in \cC, \ \Theta(\gamma_1)<\Theta(\gamma_2),
\end{align}
where $\Theta(\gamma_i) = q_i/p_i$.
\begin{lem}\label{lem:v12}
If $v$ is written as in (\ref{v123}), then $\gamma_1$, $\gamma_2$
are uniquely determined from $v$. 
% Moreover $v=(v_0, v_1, v_2, v_3)$ satisfies 
% $\Delta(v)<0$, where $\Delta(v)$ is given by (\ref{Delta}). 
\end{lem}
\begin{proof}
Each $\gamma_i \in \cC$ is either written as
$u_i(1, \theta_i, \theta_i^2, \theta_i^3)$ 
for some $u_i \in \mathbb{Z}$ and $\theta_i \in \mathbb{Q}$, or 
proportional to $(0, 0, 0, 1)$. 
If $\gamma_2$ is proportional to $(0, 0, 0, 1)$, then the lemma holds. 
Therefore it is enough to show that, for fixed $v=(v_0, v_1, v_2, v_3)$, the 
equation
\begin{align}\label{eq:v}
v_j=u_1 \theta_1^j+u_2 \theta_2^j, \ \theta_1<\theta_2, \ 0\le j\le 3
\end{align}
has at most one solution of $(u_1, u_2, \theta_1, \theta_2)$. 
The equations (\ref{eq:v}) for $j=0, 1$ give
\begin{align}\label{eq:u}
u_1=\frac{v_1-v_0 \theta_2}{\theta_1-\theta_2}, \ 
u_2=\frac{v_0 \theta_1-v_1}{\theta_1-\theta_2}. 
\end{align}
By substituting this into (\ref{eq:v}) for $j=3, 4$, we obtain 
\begin{equation}
\begin{aligned}\label{eq:v2}
&v_1(\theta_1+\theta_2)-v_0 \theta_1 \theta_2=v_2, \\
&v_1(\theta_1^2+\theta_1 \theta_2+\theta_2^2)-v_0\theta_1 \theta_2(\theta_1+\theta_2)=v_3
\end{aligned}
\end{equation}
respectively. 
By substituting $v_0 \theta_1 \theta_2=v_1(\theta_1+\theta_2)-v_2$ into 
the second, we obtain 
\begin{align}\label{eq:v3}
v_2(\theta_1+\theta_2)-v_1 \theta_1 \theta_2=v_3. 
\end{align}
On the other hand if (\ref{eq:v}) has a solution, we have 
\begin{align*}
v_1^2 -v_0 v_2=-u_1 u_2(\theta_1-\theta_2)^2 \neq 0. 
\end{align*}
Therefore (\ref{eq:v2}), (\ref{eq:v3}) 
give 
\begin{align*}
\theta_1+\theta_2=\frac{v_1 v_2-v_0 v_3}{v_1^2-v_0 v_2}, \
\theta_1 \theta_2=\frac{v_2^2-v_1 v_3}{v_1^2-v_0 v_2}. 
\end{align*}
The number of $(\theta_1, \theta_2) \in \mathbb{Q}^2$ with 
$\theta_1<\theta_2$ satisfying the above equation 
is at most one, 
and $(u_1, u_2)$ is determined by 
$(\theta_1, \theta_2)$. 
\end{proof}

If $v$ is written as (\ref{v123}), by 
Lemma~\ref{lem:v12} and 
the proof of Proposition~\ref{prop:wall1}
the only possible wall in $\Stab^{\circ}(A)$ 
where $\DT_{\sigma}(v)$ can jump is
\begin{align*}
\wW_v \cneq \{(Z, \aA) \in \Stab^{\circ}(A) : 
Z(\gamma_2) \in \mathbb{R}_{>0} Z(\gamma_1)\}. 
\end{align*}

\begin{lem}\label{beta0}
For a fixed $\alpha>0$, there is 
$\beta_0 \in \mathbb{R}$ such that 
if $\beta<\beta_0$, then the 
image of the map 
\begin{align*}
\mathbb{R}_{\ge 1/2} \to \Stab^{\circ}(A), \ 
s \mapsto \sigma_{\alpha, \beta}^{a=s\alpha^2, b=0}
\end{align*}
does not intersect with $\wW_{v}$. 
\end{lem}
\begin{proof}
As in the proof of Lemma~\ref{lem:v12}, 
suppose either $\gamma_i = u_i(1, \theta_i, \theta_i^2, \theta_i^3)$
for some $u_i \in \mathbb{Z}$ and $\theta_i =\Theta(\gamma_i) \in \mathbb{Q}$, 
or $\gamma_2$ is proportional to $(0, 0, 0, 1)$. 
We have
\begin{align*}
&Z_{\alpha, \beta}^{a=s\alpha^2, b=0}(u_i(1, \theta_i, \theta_i^2, \theta_i^3)) \\
&=u_i \left\{-(\theta_i-\beta)^3+6s\alpha^2(\theta_i-\beta)
+\sqrt{-1}\left( 3\alpha(\theta_i-\beta)^2-\alpha^3   \right)   \right\}. 
\end{align*}

First suppose that 
$\gamma_2$ is proportional to $(0, 0, 0, 1)$. 
If
$\sigma_{\alpha, \beta}^{a=s\alpha^2, b=0}$ 
lies in $\wW_{v}$, then we have
$Z_{\alpha, \beta}^{a=s\alpha^2, b=0}(\gamma_1) \in \mathbb{R}$, hence
\begin{align*}
3(\theta_1-\beta)^2-\alpha^2=0. 
\end{align*}
Hence the lemma holds by setting 
$\beta_0=\theta_1-\alpha/\sqrt{3}$. 

Next suppose that $\gamma_2$ is not proportional to $(0, 0, 0, 1)$. 
If $\sigma_{\alpha, \beta}^{a=s\alpha^2, b=0}$ 
lies in $\wW_{v}$, we have
\begin{align*}
\frac{(\theta_1-\beta)^3-6s \alpha^2(\theta_1-\beta)}
{3\alpha(\theta_1-\beta)^2 -\alpha^3}
=
\frac{(\theta_2-\beta)^3-6s \alpha^2(\theta_2-\beta)}
{3\alpha(\theta_2-\beta)^2 -\alpha^3}. 
\end{align*}
By setting $\theta=\theta_2-\theta_1$, $\overline{\beta}=\beta-\theta_1$
and simplifying, we obtain
\begin{align*}
3\overline{\beta}^2(\overline{\beta}-\theta)^2+6s \alpha^4+
3(6s-1)\alpha^2 \overline{\beta}^2
+3\theta(1-6s)\alpha^2 \overline{\beta}
-\theta^2 \alpha^2=0. 
\end{align*}
Since $\theta>0$, the above equation gives 
\begin{align*}
3(1-6s)\overline{\beta}-\theta \le 0. 
\end{align*}
Using $s\ge 1/2$, we obtain 
$\beta \ge 7\theta_1/6-\theta_2/6$. 
Hence the lemma follows by setting 
$\beta_0=7\theta_1/6-\theta_2/6$. 
\end{proof}

\begin{cor}
For any fixed $\alpha>0$, we have
\begin{align*}
\DT_{\sigma_{\alpha, \beta}}(v)=\DT_H(v), \ \beta \ll 0. 
\end{align*}
\end{cor}
\begin{proof}
If $(\alpha, \beta) \in \sS_v$
and $s\gg 0$, then by Proposition~\ref{lem:DTchamber} we have
\begin{align*}
\DT_{\sigma}(v)=\DT_H(v), \ 
\sigma=\sigma_{\alpha, \beta}^{a=s \alpha^2, b=0}.
\end{align*}
By Lemma~\ref{beta0}, for $\beta <\beta_0$, 
the wall $\wW_v$ does not intersect 
with a path from 
$\sigma_{\alpha, \beta}=\sigma_{\alpha, \beta}^{a=\alpha^2/2}$
to $\sigma_{\alpha, \beta}^{a=s\alpha^2, b=0}$, $s\gg 0$. 
Therefore $\DT_{\sigma_{\alpha, \beta}}(v)=\DT_H(v)$. 
\end{proof}

We next describe the wall $\wW_v$ on the 
$(\alpha, \beta)$-plane, 
i.e. $\BH \cap \wW_v$
where $\BH=\{\beta+i\alpha \in \mathbb{C} : \alpha>0\}$ is embedded into $\Stab^{\circ}(A)$
via the map (\ref{emb:H}). 
\begin{lem}\label{lem:compute:wall}
Suppose that 
$\gamma_i \in \cC$ is written as 
$\gamma_i=u_i(1, \theta_i, \theta_i^2, \theta_i^3)$, 
$0\neq u_i \in \mathbb{Z}$, $\theta_i \in \mathbb{Q}$
with $\theta_1<\theta_2$.
Then $\BH \cap \wW_{v}$ is
\begin{equation} \label{walls}
\left(\alpha \pm \frac{\sqrt{3}}{6}(\theta_1-\theta_2) \right)^2+\left(\beta-\frac{\theta_1+\theta_2}{2}\right)^2
=\frac{1}{3}(\theta_1-\theta_2)^2, \  \mp u_1/u_2 >0, \\
\end{equation}
If $\gamma_1=u_1(1, \theta_1, \theta_1^2, \theta_1^3)$ 
and $\gamma_2=(0, 0, 0, u_2)$, 
then $\BH \cap \wW_v$ is 
\begin{align*}
\beta=\pm\frac{\sqrt{3}}{3}\alpha+\theta_1, 
\ \pm u_1/u_2>0.
\end{align*}
\end{lem}

\begin{proof}
If $\gamma_i=u_i(1, \theta_i, \theta_i^2, \theta_i^3)$, 
we have 
\begin{align*}
Z_{\alpha, \beta}(\gamma_i)=u_i(\beta-\theta_i+i\alpha)^3. 
\end{align*}
Then $\BH \cap \wW_{v}$ is 
\begin{align}\label{BHW}
\frac{(\beta-\theta_1+i\alpha)^3}
{(\beta-\theta_2+i \alpha)^3} \in \left\{\begin{array}{ll}
\mathbb{R}_{>0}, & u_1/u_2>0 \\
\mathbb{R}_{<0}, & u_1/u_2<0. 
\end{array}
\right.
\end{align}
Since we have
\begin{align*}
\frac{(\beta-\theta_1+i\alpha)}
{(\beta-\theta_2+i \alpha)}
=\frac{1}{\alpha^2+(\beta-\theta_2)^2}
\left\{\alpha^2+(\beta-\theta_1)(\beta-\theta_2)+i\alpha(\theta_1-\theta_2)  \right\}
\end{align*}
and its imaginary part is negative, 
the condition (\ref{BHW}) is equivalent to 
\begin{align*}
\frac{\alpha(\theta_1-\theta_2)}{\alpha^2+(\beta-\theta_1)(\beta-\theta_2)}
=\pm \sqrt{3}, \ \pm u_1/u_2>0. 
\end{align*}
By simplifying, we obtain
the desired equation (\ref{walls}). 
The latter case is similar. 
\end{proof}

The walls (\ref{walls}) are circles which intersects with the $\beta$-axis at $\beta=\theta_1, \theta_2$, see Figure~\ref{Figure_Circles}.

\begin{figure}[h]
\centering
\includegraphics[width=\textwidth]{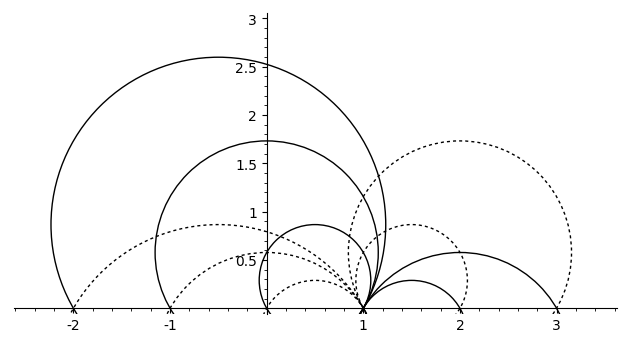}
\caption{The walls $\CW_v$ of type \eqref{walls} for $\theta_1=1$ and $\theta_2 \in \{ -2, -1, 0, 1/2, 3/2, 2, 3 \}$.
The circles are drawn dotted/solid depending on $u_1/u_2 \gtrless 0$.
%For any given $v$ only a single wall is active.
}
\label{Figure_Circles}
\end{figure}

\subsection{Proof of Theorem~\ref{intro:main}} 
Suppose $v \in \Gamma$ is written as
\begin{align}\label{v12}
v=\gamma_1+\gamma_2, \ \gamma_i=r_i(p_i^3, p_i^2 q_i, p_i q_i^2, q_i^3)
\in \cC
\end{align}
with $\Theta(\gamma_1)<\Theta(\gamma_2)$ and let
\[ g = \begin{pmatrix} a&b \\ c&d \end{pmatrix} \in \mathrm{SL}_2(\BZ). \]

\noindent
\textbf{Case 1.} $-\frac{d}{c} \notin [ \Theta(\gamma_1), \Theta(\gamma_2) )$ or $c=0$.

We take $\sigma_{\alpha, \beta}$ with $\beta \ll 0$. 
Then we have 
\begin{align*}
\DT_{H}(g_{\ast}v)&=\DT_{\sigma_{\alpha, \beta}}(g_{\ast}v)
=\DT_{g_{\ast}^{-1}\sigma_{\alpha, \beta}}(v)=\DT_{\sigma_{\alpha', \beta'}}(v),
\end{align*}
where by Lemma~\ref{lem:modular}, we have
\begin{align}\label{asymptotic}
\beta'+i\alpha'=\frac{d(\beta+i\alpha)-b}{-c(\beta+i\alpha)+a}. 
\end{align}
For $\beta \to -\infty$ we get
\begin{align*}
\beta'+i\alpha' \to 
\left\{  \begin{array}{cc}
-d/c+0, & c\neq 0, \\
-\infty, & c=0. 
\end{array} \right. 
\end{align*}
Therefore
there exists a path in $\BH$ 
which connects $(\alpha, \beta)$, $\beta \ll 0$
and $(\alpha', \beta')$, 
and does not intersect with $\BH \cap \wW_v$. 
We conclude
\begin{align*}
\DT_{\sigma_{\alpha', \beta'}}(v)=
\DT_{\sigma_{\alpha, \beta}}(v)=\DT_H(v)
\end{align*}
as desired.
%\end{proof}

\noindent
\textbf{Case 2.} $-\frac{d}{c} \in [ \Theta(\gamma_1), \Theta(\gamma_2) )$.

With the notation and argument of Step 1 it is enough to compute the right hand side of 
\begin{align*}
\DT_H(v)-\DT_H(g_{\ast}v)=
\DT_{\sigma_{\alpha, \beta}}(v)-\DT_{\sigma_{\alpha', \beta'}}(v)
\end{align*}
for $\beta \ll 0$. 
By the asymptotic behavior (\ref{asymptotic}), 
$(\alpha', \beta')$ lies
inside (resp.~RHS) of the wall $\BH \cap \wW_v$
if $\Theta(\gamma_2)<\infty$ (resp.~$\Theta(\gamma_2)=\infty$). 
Let $(\alpha_0, \beta_0)$ lies on the wall $\BH \cap \wW_v$
and take 
$\sigma_0=\sigma_{\alpha_{0}, \beta_0}$.
Let $\sigma_{\pm}$ be small 
deformations of $\sigma_0$ such that 
their central charges $Z_{\pm}$
satisfy 
\begin{align*}
\arg Z_{+}(\gamma_1)>\arg Z_{+}(\gamma_2), \ \arg Z_{-}(\gamma_1)<\arg Z_{-}(\gamma_2).
\end{align*}
From the computations in Lemma~\ref{lem:compute:wall}, 
if $\Theta(\gamma_2)<\infty$ (resp.~$\Theta(\gamma_2)=\infty$) then
$\sigma_{+}$ lies in the outer (resp.~LHS) of the wall 
$\BH \cap \wW_v$
and $\sigma_-$ lies inside (resp.~RHS) of it. 
Therefore we have 
\begin{align*}
\DT_{\sigma_{\alpha, \beta}}(v)=\DT_{\sigma_+}(v), \ 
\DT_{\sigma_{\alpha', \beta'}}(v)=\DT_{\sigma_{-}}(v). 
\end{align*}
On the other hand, the equation (\ref{wc:delta}) yields
\begin{align*}
\delta_{\sigma_0}(v, \phi)&=\delta_{\sigma_+}(v, \phi_{+})+
\delta_{\sigma_+}(\gamma_1, \phi_1) \ast \delta_{\sigma_+}(\gamma_2, \phi_2)+\cdots \\
&=\delta_{\sigma_-}(v, \phi_{-})+
\delta_{\sigma_-}(\gamma_2, \phi_2') \ast \delta_{\sigma_-}(\gamma_1, \phi_1')+\cdots. 
\end{align*}
From this we obtain
\begin{align*}
\overline{\epsilon}_{\sigma_+}(v, \phi_+)-\overline{\epsilon}_{\sigma_-}(v, \phi_-)
=-\{\overline{\epsilon}_{\sigma_0}(\gamma_1, \phi), 
\overline{\epsilon}_{\sigma_0}(\gamma_2, \phi)\}
+\cdots. 
\end{align*}
By the proof of Proposition~\ref{prop:wall1}, 
after applying $\iI^{\A}$, 
only the first term on the right
%$-\{\overline{\epsilon}_{\sigma_0}(\gamma_1), \overline{\epsilon}_{\sigma_0}(\gamma_2)\}$
contributes to the difference 
$\DT_{\sigma_+}(v)-\DT_{\sigma_{-}}(v)$. 
Since we have $\chi(\gamma_1, \gamma_2)=r_1 r_2(p_1 q_2-p_2 q_1)^3$
and using Corollary~\ref{cor:C} we find
\begin{align*}
\DT_{\sigma_+}(v)-\DT_{\sigma_-}(v) = &
(-1)^{r_1 r_2(p_1 q_2-p_2 q_1)}r_1 r_2(p_1 q_2-p_2 q_1)^3 \\
& \cdot\left( \sum_{k_1 \ge 1, k_1|r_1} \frac{1}{k_1^2} \right)
\left( \sum_{k_2 \ge 1, k_2|r_2} \frac{1}{k_2^2} \right)
\lvert B_1 \cap B_2 \rvert,
\end{align*}
where $B_i \subset \Xi(E_i)$ 
is the connected component which contains $(0, 0)$ 
for 
a semihomogeneous sheaf $E_i$ with 
Chern character 
$\pm \gamma_i \in \Gamma_{+}$. 

By~\cite[Theorem~4.9~(i)]{Mu4}, 
we have $B_i=\Xi(F_i)$ 
for a Jordan-Holder factor of $E_i$, 
whose Chern character is 
$\overline{\gamma}_i=\pm(p_i^3, p_i^2 q_i, p_i q_i^2, q_i^3) \in \Gamma_+$. 
By~\cite[Theorem~4.9~(ii)]{Mu4}, 
we hence obtain 
\begin{align*}
\lvert B_1 \cap B_2 \rvert=\chi(\overline{\gamma}_1, \overline{\gamma}_2)^2
=(p_1 q_2-p_2 q_1)^6. 
\end{align*}
Therefore the result follows. 
\qed

\subsection{Curve counting invariants} \label{Subsection_ppav_curve_counting}
For any $(\beta, n) \in \mathbb{Z}^2$ consider the rank one reduced 
Donaldson--Thomas invariant
\begin{align*}
\DT_{\beta, n}=\DT_H(1, 0, -\beta, -n). 
\end{align*}
We want to study the behaviour of $\DT_{\beta, n}$ under Fourier-Mukai transforms.

The following Lemma gives a strong constraint when two such rank $1$ classes can be related by a Fourier-Mukai transform.
\begin{lem}\label{lem:gcd}
Let $(\beta,n) \in \BZ^2$ and suppose that
\begin{align}\label{g:beta}
g(1, 0, -\beta, -n)=(1, 0, -\beta', -n')
\end{align}
for some $(\beta',n') \in \BZ^2$ and $g \in \SL_2(\mathbb{Z})$. Then there is 
$(c, d) \in \mathbb{Z}^2$
satisfying 
\begin{align*}
d^3-3\beta c^2 d-nc^3=1
\end{align*}
such that we have 
\begin{align}\label{class:beta}
(\beta', n')=(d^2\beta+ncd+\beta^2 c^2, 6\beta^2 d^2 c +6c^2 d \beta n+n+2c^3 n^2-2c^3\beta^3). 
\end{align}
\end{lem}
\begin{proof}
Let $g = \binom{a\ b}{c\ d} \in \SL_2(\mathbb{Z})$. The condition (\ref{g:beta}) gives 
\begin{align*}
(dx+by)^3-3\beta(dx+by)(cx+ay)^2-n(cx+ay)^3=x^3-3\beta' x y^2-n' y^3. 
\end{align*}
We obtain the equations
\begin{align}\label{int:eq1}
&d^3-3\beta c^2 d-nc^3=1, \\
\notag&b d^2-\beta(2acd+bc^2)-nac^2=0, \\
\notag&\beta'=\beta(a^2 d+2abc)-b^2 d+a^2 cn, \\
\notag&n'=a^3 n +3\beta a^2 d-b^3. 
\end{align}
Since $ad-bc=1$, comparing with the first equation of (\ref{int:eq1}) gives
\begin{align}\label{ab:m}
a=d^2-3\beta c^2+mc, \ 
b=nc^2+md
\end{align}
for some $m \in \mathbb{Z}$. 
By substituting this into the second equations of (\ref{g:beta}), 
we obtain $m=2\beta c$. 
By substituting (\ref{ab:m})
into the third and fourth equation of (\ref{int:eq1}), 
and simplifying, we obtain (\ref{class:beta}). 
\end{proof}

Let $C_{\beta, n} \in \mathbb{Q}$ be the conjectural value of $\DT_{\beta,n}$ defined by the right hand side of \eqref{def:C}.
By Lemma~\ref{lem:gcd} and an elementary check we have
\[ C_{\beta, n}=C_{\beta', n'} \]
whenever $(\beta,n)$ and $(\beta',n')$ are related as in \eqref{g:beta}.
We therefore obtain the following evidence for Conjecture~\ref{intro:conj}.
\begin{cor}\label{cor:rkone}
If $4\beta^3-n^2\ge 0$ and $(\beta',n')$ is as in \eqref{g:beta}, then
\[ \DT_{\beta, n}=\DT_{\beta', n'}. \]
In particular, $\DT_{\beta, n}=C_{\beta, n}$ if and only if $\DT_{\beta', n'}=C_{\beta', n'}$.
\end{cor}
\begin{proof}
Since $\Delta(1, 0, -\beta, -n)=4\beta^3-n^2$ this follows from Theorem~\ref{intro:first} and Proposition~\ref{intro:prop_Delta}. 
\end{proof}

Suppose that $(\beta, n) \in \mathbb{Z}^2$ satisfies
\begin{align*}
(1, 0, -\beta, -n)=\gamma_1+\gamma_2, \ \gamma_i \in \cC, \ 
\Theta(\gamma_1)<\Theta(\gamma_2). 
\end{align*}
We address the following question: 
\begin{conj}\label{quest}
Suppose that $\beta \neq 0$ or $n>0$. 
For any integer solution $(c, d)$ of 
$d^3 -3\beta c^2 d-nc^3=1$, we have 
\begin{align*}
-\frac{d}{c} \notin ( \Theta(\gamma_1), \Theta(\gamma_2)). 
\end{align*}
\end{conj}
\begin{example}\label{rmk:beta0}
If $\beta=0$ and $n>0$, then we have 
\begin{align*}
(1, 0, 0, -n)=\gamma_1+\gamma_2, \ 
\gamma_1=(1, 0, 0, 0), \gamma_2=-n(0, 0, 0, 1)
\end{align*}
and $\Theta(\gamma_1)=0$, $\Theta(\gamma_2)=\infty$. 
In this case, 
for any 
integer solution $(c, d)$ of 
$d^3-nc^3=1$ we have
$-d/c \notin (0, \infty)$. 
Moreover $-d/c=0$ only if 
$n=1$ and $(c, d)=(-1, 0)$. 
In this case, $(\beta', n')$ given by (\ref{class:beta}) is 
$(0, -1)$. 
\end{example}
We have the following lemma:
\begin{lem}\label{conj:equiv}
Conjecture~\ref{quest} is equivalent to the following: 
for $\beta \neq 0$ or $n>0$ and an integer solution $(c, d)$
of $d^3 -3\beta c^2 d-nc^3=1$, if we have 
\begin{align}\label{cond:cd}
-\frac{d}{c} \in [ \Theta(\gamma_1), \Theta(\gamma_2)) 
\end{align}
then $\beta'=0$ and $n'\le 0$. 
Here $(\beta', n')$ is given by (\ref{class:beta}). 
\end{lem}
\begin{proof}
By Example~\ref{rmk:beta0}, we may assume that $\beta \neq 0$. 
By writing $\theta_i=\Theta(\gamma_i)$, the computation 
in Lemma~\ref{lem:v12}
shows
\begin{align}\label{theta:12}
\theta_1+\theta_2=\frac{n}{\beta}, \ \theta_1 \theta_2=\beta. 
\end{align}
It follows that
\begin{align}\label{product:c}
\left(\theta_1+\frac{d}{c} \right) \left(\theta_2+\frac{d}{c} \right)=
\frac{1}{c^2} \cdot \frac{\beta'}{\beta}. 
\end{align}
Suppose that Conjecture~\ref{quest} is true. 
Then the condition (\ref{cond:cd}) implies 
$-d/c=\theta_1$, hence 
$\beta'=0$ follows. 
Suppose by a contradiction that $n'>0$. 
Note that 
\begin{align}\label{g-1}
g^{-1}(1, 0, 0, -n')=(1, 0, -\beta, -n).
\end{align}
We write
\begin{align*}
g^{-1}=\left( \begin{array}{cc}
a' & b' \\
c' & d' 
\end{array}\right)
=
\left( \begin{array}{cc}
d & -b \\
-c & a 
\end{array}\right). 
\end{align*}
Then the condition (\ref{g-1}) implies 
$(d')^3-n'(c')^2=1$, and 
the condition (\ref{cond:cd})
implies that $-d'/c' \in [0, \infty)$
(see Remark~\ref{rmk:beta0}).  
By Remark~\ref{rmk:beta0}, this implies that 
$d'=a=0$. By (\ref{ab:m}), we have 
$a=d^2-\beta c^2=0$, 
thus $\beta=\theta_1^2$ follows. 
By (\ref{theta:12}),
we have $\theta_1^2=\theta_1 \theta_2$. 
Since $\theta_1 \neq \theta_2$, we have 
$\theta_1=0$ and $\beta=0$, a contradiction. 

The converse statement follows from (\ref{product:c}).  
\end{proof}

\begin{rmk}
If Conjecture~\ref{quest} is false, then by Lemma~\ref{conj:equiv}
we have $\DT_{\beta, n} \neq \DT_{\beta', n'}$
for $\beta' \neq 0$ or $n'>0$, 
while $C_{\beta, n}=C_{\beta', n'}$.
So
either $(\beta, n)$ or $(\beta', n')$ would 
give a counter-example to
Conjecture~\ref{intro:conj}. 
\end{rmk}
Theorem~\ref{intro:main} (i) and 
Lemma~\ref{conj:equiv} immediately implies
the following: 
\begin{cor}
For $\beta \neq 0$ or $n>0$, suppose that 
Conjecture~\ref{quest} is true. 
Then for any integer solution $(c, d)$
of $d^3-3\beta c^2 d-nc^3=1$
with 
either $\beta'\neq 0$ or $n'>0$, 
we have 
$\DT_{\beta, n}=C_{\beta, n}$ if and 
only if $\DT_{\beta', n'}=C_{\beta', n'}$
holds. 
\end{cor}
By Example~\ref{rmk:beta0}, we can 
apply the above corollary 
for $\beta=0$ and $n>0$.
Since $\DT_{0, n}=C_{0, n}$ holds 
by~\cite{MR3421659}, 
we obtain the following: 
\begin{cor}
For $n>0$ and any integer solution 
$(c, d)$ of 
$d^3 -nc^3=1$,
except $n=1$ and $(c, d)=(-1, 0)$, we have 
\begin{align}\notag
\DT_{cdn, n+2c^3 n^2}=(-1)^{n-1} \frac{1}{n}\sum_{k\ge 1, k|n}k^2. 
\end{align}
\end{cor}

\appendix
\section{Spin representations and the discriminant}
\label{Appendix_Spin_representations}
Let $U$ be a $\BQ$-vector space with basis $x_1, \ldots, x_n$.
The algebra of endomorphisms of the exterior algebra
$\bigwedge^{\bullet} U$ is the exterior algebra generated
by multiplication by $x_i$ and differentiation (i.e. interior product) $\partial/\partial x_i$:
\[ \End_{\BQ}\Big( \bigwedge\nolimits^{\bullet} U \Big)
=
\bigwedge\nolimits^{\bullet} \left\langle x_1 \wedge\ , \ldots, x_n \wedge\ , \frac{\partial}{\partial x_1} , \ldots, \frac{\partial}{\partial x_n} \right\rangle. \]
The Lie subalgebra of $\End_{\BQ}(\wedge^{\bullet} U)$ 
%(with respect to the standard commutator)
generated by
\begin{equation} x_i \wedge x_j,\quad x_i \wedge \frac{\partial}{\partial x_j} - \frac{1}{2} \delta_{ij}, \quad \frac{\partial^2}{\partial x_i \partial x_j},
\quad 1 \leq i < j \leq n \label{generators} \end{equation}
is isomorphic to $\mathfrak{so}(2n)$, and the induced action of
$\mathfrak{so}(2n)$ on $\bigwedge^{\bullet} U$ is called the \emph{spin representation}.
%is a Lie subalgebra (with respect to the standard commutator) isomorphic to $\mathfrak{so}(2n)$.
This Lie algebra action integrates to a representation of the spin group $\mathrm{Spin}(2n)$.
%The induced action of $\mathfrak{so}(2n)$ on $\bigwedge^{\bullet} U$ integrates to a representation of $\mathrm{Spin}(2n)$, called the \emph{spin representation}.
%

The action by $\mathfrak{so}(2n)$ preserves the decomposition
\[ \bigwedge\nolimits^{\bullet} U = \bigwedge\nolimits^{\text{even}} U \, \oplus \, \bigwedge\nolimits^{\text{odd}} U. \]
where $\bigwedge\nolimits^{\text{even/odd}} U$ is the subspace spanned by all even/odd wedge products.
The induced action of the spin group on $\bigwedge\nolimits^{\text{even/odd}} U$ is irreducible and called the
even/odd half-spin representation.

There exist a unique (up to scalar) invariant
bilinear form $\beta$ on $\bigwedge\nolimits^{\text{even}}U$.
If $n$ is even, we normalize $\beta$ by $\beta(1, \prod_{i=1}^{n} x_i) = 1$.

\begin{rmk}
If $A$ is an abelian variety of dimension $g$, then $H^1(A,\BQ)$ is of dimension $2g$ and
\[ H^{\ast}(A,\BQ) = \bigwedge\nolimits^{\bullet} H^1(A,\BQ). \]
The action of the group of derived autoequivalences on $H^{\ast}(A, \BQ)$
factors through the spin representation of $\mathrm{Spin}(4g)$, see \cite[Section~3]{Mu4}.
Every function on $H^{\ast}(A,\BQ)$ invariant under $\mathrm{Spin(4g)}$ is therefore invariant
under all autoequivalences.
For instance the invariant bilinear form $\beta$ is the Euler pairing:
\[ \forall E,F \in \Coh(A)\colon \ \ \chi(E,F) = \beta( \ch(E) , \ch(F) ). \]
\end{rmk}

% $\mathrm{Spin}(2n)$ preserves this decomposition.
% We call the even (odd) summand the even (odd) half-spin representation of $\mathfrak{so}(2n)$.
% The restriction of the $\mathrm{Spin}(2n)$ to the even (odd) summand is irre
% and its restriction to each summand is irreducible and called the even and odd half-spin representation respectively.
%
%The subspace $\bigwedge^{\text{even}}U$ spanned by even wedge products is preserved by the spin group
%The representation on $\bigwedge^{\text{even}}U$ is irreducible and called the half-spin representation.
%Integrating it yields the half-spin representation of $\mathrm{Spin}(2n)$.
%The induced action on $\bigwedge^{\text{even}}U$ is irreducible and called the even half-spin representation of $\mathrm{so}(2n)$.

\begin{thm}\label{Thm_Discriminant} Assume $\dim(U)=6$.
\begin{enumerate}
\item[a)] There exist a unique homogeneous degree $4$ polynomial function
\[ \Delta : \bigwedge\nolimits^{\text{even}}U \to \BQ \]
which is invariant under the action of $\mathrm{Spin}(12)$. We normalize $\Delta$ by
$\Delta(1 + \prod_{i=1}^{6} x_i) = -1$.
\item[b)] We have $\Delta(e^{\omega}) = 0$ for all $\omega \in \bigwedge^{2}U$.
\item[c)] For all $r_1, r_2 \in \BZ$ and $\omega_1, \omega_2 \in \bigwedge^{2}U$ we have
\[
\Delta(r_1 e^{\omega_1} + r_2 e^{\omega_2}) = -\beta(r_1 e^{\omega_1}, r_2 e^{\omega_2})^2.
\]
%where $\beta$ is the unique invariant bilinear form on $\bigwedge\nolimits^{\text{even}}U$,
%and $\beta$ and 
%where $\Delta$ is normalized by
%$\Delta(1 + \prod_{i=1}^{6} x_i) = -1$.
%and $\beta(1, \prod_{i=1}^{6} x_i) = 1$.
%  
%  If $\Delta$ is normalized such that $\Delta(1 + \prod_{i=1}^{6} x_i) = 1$ then
\end{enumerate}
\end{thm}

\begin{rmk}
Let $A = E_1 \times E_2 \times E_3$ where $E_1, E_2, E_3$ are very general elliptic curves.
The subalgebra of algebraic classes $\Gamma \subset H^{\ast}(A,\BQ)$ is generated by
\[ L_i = \pi_i^{\ast} [\pt_{i}] \in H^2(A, \BZ), \ \ i=1,2,3 \]
where $\pt_{i} \in H^2(E_i)$ is the point class and $\pi_i : A \to E_i$ is the projection. 
% Let also
% \[ [E_1] = L_2 L_3, \quad [E_2] = L_1 L_3, \quad \quad [E_3] = L_1 L_2. \]
If
\[ \gamma = (r, b_1 L_1 + b_2 L_2 + b_3 L_3, d_1 L_2 L_3 + d_2 L_1L_3 + d_3 L_1L_2, n) \in \Gamma \]
is a general element, then the discriminant of $\gamma$ is
\begin{align*}
\Delta(\gamma) & = -n^2 r^2 - 4 (r d_1 d_2 d_3 + b_1 b_2 b_3 n) \\
& - (b_1^2 d_1^2 + b_2^2 d_2^2 + b_3^2 d_3^2) \\
& + 2 b_1 b_2 d_1 d_2 + 2 b_1 b_3 d_1 d_3 + 2 b_2 b_3 d_2 d_3 \\
& + 2 r n ( b_1 d_1 + b_2 d_2 + b_3 d_3 ).
\end{align*}
\end{rmk}

\begin{proof}[Proof of Theorem~\ref{Thm_Discriminant}]
Let $V = \bigwedge\nolimits^{\text{even}}U$. By a calculation in \cite{sage} the tensor product $V^{\otimes 4}$ 
contains $4$ copies of the trivial representation.\footnote{
The submission line to SAGE is:
%$$\mathrm{chi}= \text{WeylCharacterRing}("\mathrm{D6}")(1/2,1/2,1/2,1/2,1/2,1/2); \mathrm{chi}\text{\^}4$$
\begin{quote}
\verb|chi=WeylCharacterRing("D6")(1/2,1/2,1/2,1/2,1/2,1/2); chi^4|
\end{quote}
with result \verb|chi^4 = 4*D6(0,0,0,0,0,0) + (...)|.
%with result $chi^4 = 4*D6(0,0,0,0,0,0) + (...)|$.
}
%\footnote{
%The submission line to SAGE is: 
%chi=WeylCharacterRing("D6")(1/2,1/2,1/2,1/2,1/2,1/2); chi^4
%}
Three of them arise from $\beta \otimes \beta$ by permuting factors, and hence are not $S_4$ invariant.
This shows uniqueness. We prove existence. Consider a general element
\[ \gamma = \sum_{\substack{I \subset \{ 1,2,3,4,5,6 \} \\ |I| \text{ even }}} a_I x_I \]
where $a_I \in \BQ$ and $x_I = \prod_{i \in I} x_i$. 
We make the ansatz
%\begin{multline*} 
\begin{equation} \Delta(\gamma) = \sum_{I=(I_1, I_2, I_3, I_4)} c_I a_{I_1} a_{I_2} a_{I_3} a_{I_4} 
%\, = \, a^2 a_{012345}^2 + \ldots 
\label{Delta_formula} \end{equation}
%= a^2 a_{012345}^2 + 2 a a_{01} a_{2345} a_{012345} + \ldots 
%\end{multline*}
for some $c_I \in \BQ$, where the $I_j$ run over all even subsets of $\{ 1, \ldots, 6 \}$ such that every $1 \leq i \leq 6$ appears in the subsets $I_j, j=1,2,3,4$ exactly twice.
A computer calculation\footnote{
The code for this computation is available on the first author's webpage.}
shows that there exist unique (up to scaling) $c_I$ such that $\Delta$ is invariant under the generators \eqref{generators}.
This proves (a).

Multiplication by $\omega \in \wedge^2 U$ is an element of the Lie algebra $\mathfrak{so}(12)$,
hence multiplication by $e^{\omega}$ is an element of $\mathrm{Spin}(12)$. It follows
\[ \Delta(e^{\omega}) = \Delta(1) = 0 \]
where the last equality follows from \eqref{Delta_formula}. This shows part (b).

Finally, (c) follows again by a direct computer calculation.
\end{proof}

\providecommand{\bysame}{\leavevmode\hbox to3em{\hrulefill}\thinspace}
\providecommand{\MR}{\relax\ifhmode\unskip\space\fi MR }
\providecommand{\MRhref}[2]{%
\href{http://www.ams.org/mathscinet-getitem?mr=#1}{#2}
}
\providecommand{\href}[2]{#2}

\end{document}